\documentclass[12pt,twoside]{amsart}
\usepackage{a4wide, amsmath,amsbsy,amsfonts,amssymb,
stmaryrd,amsthm,mathrsfs,graphicx, amscd, tikz-cd, cancel}
\usepackage[normalem]{ulem}
\usepackage{soul}
\usetikzlibrary{matrix,arrows,decorations.pathmorphing}
\usepackage[active]{srcltx}
\usepackage[
	hypertexnames=false,
	hyperindex,
	pagebackref,  
	pdftex,
	breaklinks=true,
	bookmarks=false,
	colorlinks,
	linkcolor=blue,
	citecolor=red,
	urlcolor=red,
]{hyperref}

\usepackage[all]{xy}
\usepackage{subfig}
\usepackage{tikz}
\usetikzlibrary{shapes,arrows,shadows}
\usetikzlibrary{decorations.markings}
\usepackage{color}
\usepackage[normalem]{ulem}
\usepackage{hyperref}
\usepackage{wrapfig}
\usetikzlibrary{arrows}
\usepackage{pinlabel}
\long\def\symbolfootnote[#1]#2{\begingroup%
\def\thefootnote{\fnsymbol{footnote}}\footnote[#1]{#2}\endgroup}

\newtheorem{innercustomthm}{Theorem}
\newenvironment{customtheorem}[1]
  {\renewcommand\theinnercustomthm{#1}\innercustomthm}
  {\endinnercustomthm}

\newtheorem{theorem}{Theorem}[section]

\newtheorem{proposition}[theorem]{Proposition}
\newtheorem{corollary}[theorem]{Corollary}
\newtheorem{lemma}[theorem]{Lemma}

\newtheorem{conjecture}[theorem]{Conjecture}

\theoremstyle{definition}

\newtheorem{definition}[theorem]{Definition}
\newtheorem{remark}[theorem]{Remark}

\newtheorem*{namedtheorem}{\theoremname}
\newcommand{\theoremname}{testing}

\DeclareMathSymbol{\Alpha}{\mathalpha}{operators}{"41}
\DeclareMathSymbol{\Beta}{\mathalpha}{operators}{"42}
\DeclareMathSymbol{\Epsilon}{\mathalpha}{operators}{"45}
\DeclareMathSymbol{\Zeta}{\mathalpha}{operators}{"5A}
\DeclareMathSymbol{\Eta}{\mathalpha}{operators}{"48}
\DeclareMathSymbol{\Iota}{\mathalpha}{operators}{"49}
\DeclareMathSymbol{\Kappa}{\mathalpha}{operators}{"4B}
\DeclareMathSymbol{\Mu}{\mathalpha}{operators}{"4D}
\DeclareMathSymbol{\Nu}{\mathalpha}{operators}{"4E}
\DeclareMathSymbol{\Omicron}{\mathalpha}{operators}{"4F}
\DeclareMathSymbol{\Rho}{\mathalpha}{operators}{"50}
\DeclareMathSymbol{\Tau}{\mathalpha}{operators}{"54}
\DeclareMathSymbol{\Chi}{\mathalpha}{operators}{"58}
\DeclareMathSymbol{\omicron}{\mathord}{letters}{"6F}
\newcommand{\R}{\mathbb{R}}
\newcommand{\C}{\mathbb{C}}
\newcommand{\N}{\mathbb{N}}
\newcommand{\Z}{\mathbb{Z}}
\newcommand{\ZZ}{{\widehat{\mathbb Z}}}

\newcommand{\Q}{\mathbb{Q}}

\newcommand{\bt}{\bullet}

\newcommand\Res{\operatorname{Res}}

\def\Aut{\operatorname{Aut}}
\def\Sp{\operatorname{Sp}}

\def\SL{\operatorname{SL}}
\def\PSL{\operatorname{PSL}}
\def\Inn{\operatorname{Inn}}
\def\inn{\operatorname{inn}}

\def\Out{\operatorname{Out}}
\def\Mod{\operatorname{Map}}

\def\Hom{\operatorname{Hom}}
\def\Der{\operatorname{Der}}
\def\PG{\mathrm{P}\Gamma}
\def\Star{\operatorname{Star}}
\def\Link{\operatorname{Link}}
\def\Isom{\operatorname{Isom}}
\def\Spec{\operatorname{Spec}}
\def\P{{\mathbb P}}
\def\H{{\mathbb H}}
\def\M{{\mathbb M}}
\def\k{\Bbbk}
\def\kk{{\overline{\k}}}
\def\cP{{\mathcal P}}

\def\cL{{\mathcal L}}
\def\cI{{\mathcal I}}
\def\ccI{{\mathscr I}}
\def\hL{\widehat{\mathcal L}}

\def\cA{{\mathcal A}}

\def\cT{{\mathcal T}}

\def\wT{{\widehat{\mathcal T}}}

\def\cC{{\mathscr C}}
\def\cF{{\mathcal F}}
\def\wF{{\widehat{\mathcal F}}}
\def\cG{{\mathcal G}}
\def\cH{{\mathcal H}}
\def\wH{{\widehat{\mathcal H}}}
\def\cM{{\mathcal M}}
\def\cP{{\mathcal P}}
\def\hP{{\widehat\Pi}}

\def\bM{{\overline{\mathcal M}}}
\def\wM{{\widehat{\mathcal M}}}

\def\dd{\partial}
\def\a{{\alpha}}

\def\l{{\lambda}}
\def\L{{\Lambda}}

\def\g{{\gamma}}
\def\G{{\Gamma}}
\def\kG{{\check\Gamma}}
\def\kPG{\mathrm{P}{\check\Gamma}}
\def\hG{{\widehat\Gamma}}
\def\hPG{\mathrm{P}{\widehat\Gamma}}

\def\s{{\sigma}}

\def\bt{\bullet}

\def\kC{{\check  C}}
\def\hC{{\widehat{C}}}
\def\hI{{\hat{\mathrm{I}}}}

\def\cD{{\mathcal D}}

\def\ssm{\smallsetminus}
\def\ol{\overline}
\def\wh{\widehat}
\def\wt{\widetilde}
\def\ul{\underline}
\def\ra{\rightarrow}
\def\hookra{\hookrightarrow}
\def\co{\colon\thinspace}

\begin{document}

\title{Automorphisms of procongruence curve and pants complexes}
 
 \author[M. Boggi]{Marco Boggi}\thanks{Marco Boggi was partially supported by CAPES - C\'odigo de Financiamento 001 -
and by Institut Fourier, Laboratoire de Mathematiques UMR 5582,   Universit\'e Grenoble Alpes, France.}
\address{Instituto de Matem\'atica e Estadistica, Universidade Federal Fluminense \\
S\~ao Domingos, Niter\'oi - State of Rio de Janeiro, 24210-200, Brazil.}
\email{marco.boggi@gmail.com}

\author[L. Funar]{Louis Funar}
\address{Univ. Grenoble Alpes, CNRS, Institut Fourier,   
38000 Grenoble, France}
\email{louis.funar@univ-grenoble-alpes.fr} 



\begin{abstract}    

In this paper we study the automorphism group of the procongruence mapping class group through its action on the associated procongruence curve 
and pants complexes. Our main result is  a rigidity theorem for the procongruence completion of the pants complex. As an application we prove that
moduli stacks of smooth algebraic curves satisfy a weak anabelian property in the procongruence setting.
\bigskip

\noindent {\bf AMS Math Classification:} 14G32, 20E18, 14D23, 20F34, 57M10.

\end{abstract}

\maketitle

\section{Introduction}
Let $S=S_{g,n}$ be a closed orientable surface of genus $g(S)=g$ from which $n$ points have been removed.   We assume that $S$ has 
negative Euler characteristic, i.e.\ $2-2g-n<0$. Let $\Mod(S)$ be the extended mapping class group of the surface $S$, namely the group of isotopy 
classes of diffeomorphisms of $S$. The mapping class group $\G(S)$ is the subgroup of $\Mod(S)$ consisting of those elements which preserve a
fixed orientation of the surface. 

Let $\cM(S)$ be the Deligne--Mumford (briefly DM) moduli stack over $\Q$ parameterizing smooth algebraic 
curves whose complex model is diffeomorphic to $S$. 
For an algebraic stack $Y$ defined over $\k$ and a field extension $\k\subseteq\k'$, let $Y_{\k'}:=Y\times_\k\Spec\k'$. 
Then, after a choice of base point, the mapping class group $\G(S)$ identifies 
with the topological fundamental group of the complex moduli stack $\cM(S)_\C$.

In a series of papers (cf.\ \cite{[I1]}, \cite{[I2]}, \cite{[McC]}), Ivanov, for $g\geq 3$, and McCarthy, for $g\leq 2$, 
determined the automorphisms groups of $\G(S)$ and $\Mod(S)$. An essential tool was the \emph{complex of curves} $C(S)$. 
This is the (abstract) simplicial complex whose simplices consist of sets of isotopy classes of nonperipheral simple closed curves on $S$ 
which admit disjoint representatives. Its dimension is $d(S)-1=3g-2+n$, where $d(S)$ is called the \emph{modular dimension} of $S$, because it is also 
the dimension of the moduli stack $\cM(S)$. There is a natural action of $\Mod(S)$ on $C(S)$.
Ivanov observed that this action factors through the group of inner automorphisms $\Inn\Mod(S)$ and extends to a homomorphism 
$\Aut(\Mod(S))\to\Aut(C(S))$. Away from a few exceptions (for $d(S)\leq 3$), the latter homomorphism is injective. 
A fundamental result of Ivanov then states that, for $d(S)>2$,
the homomorphism $\Inn(\Mod(S))\to\Aut(C(S))$ is also surjective. This immediately implies that, for $d(S)> 3$, we have $\Aut(\Mod(S))=\Inn(\Mod(S))$
and $\Aut(\G(S))=\Inn(\Mod(S))$ (see Corollary~\ref{Theorem 2.4} for a more precise result).

Let $\hG(S)$ be the profinite completion of the mapping class group $\G(S)$. The automorphism group $\Aut(\hG(S))$ is of great arithmetic significance. 
By Bely\v{\i} theorem (cf.\ \cite{Belyi}), there is indeed a natural faithful representation $G_\Q\hookra\Out(\hG(S))$, where $G_\Q$ is the absolute 
Galois group of the rationals. Classical Grothendieck-Teichm\"uller theory can then be described as the attempt to corner the image of $G_\Q$ inside 
$\Out(\hG(S))$ (cf.\ \cite{Drinfeld}, \cite{Ihara}, \cite{LNS}).

The basic idea of Grothendieck-Teichm\"uller theory is to restrict to the subgroup of automorphisms of $\hG(S)$ which satisfy the same conditions 
naturally satisfied by those coming from the Galois action. This leads to consider the subgroup $\Out^\ast(\hG(S))$ of $\Out(\hG(S))$, 
roughly described as the subgroup of elements which preserve the conjugacy classes of a special set of geometrically significant subgroups. 
In fact, thanks to a result of Harbater and Schneps (cf.\ \cite{HS} for details),
the profinite Grothendieck-Teichm\"uller group $\wh{\mathrm{GT}}$ can be identified with $\Out^\ast(\hG(S))$
for $S$ the $5$-punctured sphere. Harbater and Schneps then proceeded to show that there is a natural isomorphism
$\wh{\mathrm{GT}}\cong\Out^\ast(\hG(S))$ for all $S$ such that $g(S)=0$ and $d(S)>1$. 

The main stumbling block in extending genus $0$ Grothendieck-Teichm\"uller theory to higher genus is the fact that, in the profinite setting contrary
to the topological one, there is no satisfactory geometric description of the subgroups of $\hG(S)$ arising as centralizers of the "profinite" Dehn twists
(cf.\ Section~\ref{section:prodehntwists}). A way to circumvent this difficulty is to consider, instead of $\hG(S)$, the \emph{procongruence mapping 
class group} $\kG(S)$, defined as the image of the natural representation $\hG(S)\to\Out(\wh{\pi}_1(S))$.
This approach was systematically developed in \cite{[B3]} (cf.\ also \cite{HM2}) where a complete description of centralizers of procongruence Dehn
twists is given. 

Note that, since the congruence subgroup property holds in genus $0$, there are natural isomorphisms $\hG(S)\cong\kG(S)$ and then
$\wh{\mathrm{GT}}\cong\Out^\ast(\kG(S))$ for $g(S)=0$. Moreover, by Corollary~7.6 in \cite{[B3]}, for instance, there is a natural faithful
representation $G_\Q\hookra\Out^\ast(\kG(S))$ for all genera (cf.\ Definition~\ref{inertiapreservingdef} for the precise definition of $\Out^\ast(\kG(S))$). 
Therefore, Grothendieck-Teichm\"uller theory can be rephrased in this context.

A key property of the group $\Aut^\ast(\kG(S))$ is that it admits a natural action on a profinite version of the curve complex and 
can then be approached in a manner similar to the group $\Aut(\G(S))$. The {\em procongruence curve complex} $\kC(S)$ of $C(S)$
is an abstract simplicial profinite complex (cf.\ Definition~3.2 in \cite{[B3]}) naturally associated to the congruence completion $\kG(S)$ of the mapping 
class group. There is a natural continuous action of $\kG(S)$ on $\kC(S)$ which factors through an action of $\Inn\kG(S)$ and then extends to an
action $\Aut^\ast(\kG(S))\to\Aut(\kC(S))$. In complete analogy with the topological case, this homomorphism is injective, except for a few cases 
when $d(S)\leq 3$ (cf.\ Theorem~\ref{faithfulness}). A completely different matter is to understand whether this homomorphism is surjective.

In this paper, we will deal with a somewhat more treatable but related problem. Let $C_P(S)$ be the \emph{pants graph} associated to the surface $S$. 
The vertices of $C_P(S)$ are the facets of $C(S)$, namely the maximal multicurves. Two vertices are connected by an edge if the corresponding 
multicurves share a subset of $d(s)-1$ elements while the remaining pair of curves has minimal nontrivial geometric intersection 
(cf.\ Section~\ref{pantsgraphdef}). As done for the curve complex, we can associate to the congruence completion $\kG(S)$ a 
profinite version of the pants graph. This is the {\em procongruence pants graph} $\kC_P(S)$, a $1$-dimensional 
abstract simplicial profinite complex endowed with a natural continuous action of $\kG(S)$. 

A \emph{level structure} $\cM^\l$ is a finite, geometrically connected, \'etale covering of the moduli stack $\cM(S)$ (cf.\ Section~\ref{levels}). 
This admits a canonical extension to a (in general, ramified) covering $\bM^\l\to\bM(S)$,  where $\bM(S)$ is the \emph{DM compactification} 
of $\cM(S)$ (cf.\ Section~\ref{moduli}).
Then, a basic feature of $\kC_P(S)$ is that it can be realized as the inverse limit of $1$-skeletons of some natural triangulations of the $1$-dimensional 
strata in the DM boundary of $\bM^\l_\C$. It is thus natural to expect  $\kC_P(S)$ to be a more rigid
object than the procongruence curve complex $\kC(S)$. The main result of the paper shows that this is indeed the case.
 
More precisely, the natural action of $\kG(S)$ on $\kC_P(S)$ factors through a homomorphism $\Inn(\kG(S))\to\Aut(\kC_P(S))$,
which is injective for $d(S)>1$. Then, the main result of the paper is that, in analogy with what happens in the topological setting, 
this map has cofinite image (cf.\ Theorem \ref{completepantsrigidity}):

\begin{customtheorem}{A}
For $S\neq S_{1,2}$ a connected hyperbolic surface such that $d(S)>1$, there is an exact sequence:
\[1\to\Inn(\kG(S))\to \Aut(\kC_P(S))\to \prod_{O(S)} \{\pm 1\},\] 
where $O(S)$ is the finite set of the topological types of $(d(S)-1)$-multicurves on $S$. For $S$ of type $(1,2)$, 
the group $\Aut(\kC_P(S_{1,2}))$ must be replaced with the subgroup of those automorphisms preserving the set of separating curves.
\end{customtheorem}  

Note that $O(S)$ can equivalently be described as the set of $\G(S)$-orbits of $(d(S)-1)$-multicurves: $C(S)_{d(S)-2}/\G(S)=\kC(S)_{d(S)-2}/\kG(S)$.
 
The second part of this paper is devoted to arithmetic applications of the procongruence rigidity of pants complexes.
Let $\k$ be a sub-$p$-adic field, that is to say a subfield of a finitely generated extension of $\Q_p$ for some prime $p$ and let $G_\k$ be its Galois group.  
By a classical result of S. Mochizuki, given a smooth hyperbolic curve $C$ and any smooth variety $X$, both defined over $\k$, there is a natural 
bijection between the set of dominant morphisms $\Hom_\k^\mathrm{dom}(X,C)$ and the set of $G_\k$-equivariant open homomorphism, up to inner 
automorphisms, $\Hom_{G_\k}^\mathrm{open}((\pi_1^\mathrm{et}(X_\kk),\pi_1^\mathrm{et}(C_\kk))^\mathrm{out}$  (cf.\ Theorem~A in \cite{[Mo]}). 
This property, known as the {\em anabelian property for hyperbolic curves}, had been conjectured by Grothendieck in \cite{[G]} for number fields.

In this paper, building on Mochizuki theorem and as an application of Theorem~A, we will be able to prove a procongruence version of the 
anabelian conjecture for moduli stacks of smooth hyperbolic curves, possibly with level structure.

Working with stacks, instead of varieties, entails the additional difficulty that they form a strict $(2,1)$-category, 
that is to say their $\Hom$ functors take values in groupoids
rather than in sets. More precisely, for each pair of stacks $X$ and $Y$, it is given a category $\Hom(X,Y)$, whose objects are the
$1$-morphisms $f\co X\to Y$ and such that morphisms between two $1$-morphisms $f,g\co X\to Y$, denoted by a double arrow 
$\alpha\co f\Rightarrow g$ and called $2$-morphisms, are invertible. We will denote the objects of the category $\Hom(X,Y)$ 
by $1$-$\Hom(X,Y)$ and its morphisms by $2$-$\Hom(X,Y)$.

Let  $\cM^\l$ and $\cM^\mu$ be level structures over the moduli stack $\cM(S)$ and $\cM(S')$, respectively, where, for simplicity, we further assume 
that both $S$ and $S'$ are different from $S_{0,4}$.
We then formulate an {\em anabelian conjecture for moduli stacks of curves with level structure} (cf.\ Conjecture~\ref{anabelianconjmod} for a more
general version of this conjecture) as the statement that: 
\begin{enumerate}
\item there is a natural bijection 
$1$-$\Isom_\k(\cM^\l_\k,\cM^\mu_\k)\stackrel{\sim}{\to}\Isom_{G_\k}(\pi_1^\mathrm{et}(\cM^\l_\kk),\pi_1^\mathrm{et}(\cM^\mu_\kk))^\mathrm{out}$,
where the latter is the set of outer $G_\k$-equivariant isomorphisms; 
\item the group of generic automorphisms $Z^\mu$ of the DM stack $\cM_\k^\mu$ identifies with the center of $\pi_1^\mathrm{et}(\cM^\mu_\kk)$.
\end{enumerate}

Since there is a natural isomorphism $Z^\mu\cong\Aut_\k(\mathrm{id}_{\cM^\mu_\k})$ (cf.\ (i) of Lemma~\ref{1-2-aut}) and the set 
$2$-$\Isom_\k(\cM^\l_\k,\cM^\mu_\k)$, if nonempty, is a trivial $\Aut_\k(\mathrm{id}_{\cM^\mu_\k})$-torsor over $1$-$\Isom_\k(\cM^\l_\k,\cM^\mu_\k)$, 
the above two conditions imply that the groupoid $\Isom_\k(\cM^\l_\k,\cM^\mu_\k)$ can indeed be recovered from the \'etale fundamental groups 
$\pi_1^\mathrm{et}(\cM^\l_\kk)$ and $\pi_1^\mathrm{et}(\cM^\mu_\kk)$ (cf.\ Section~\ref{sect:anabelian} for details).   

In our setting, we need to replace the geometric \'etale fundamental group  of the moduli 
stack of curves with the procongruence completion of the mapping class group or, equivalently, its image under the universal monodromy representation. 
Of course, if the congruence subgroup property holds, the prongruence completion coincides with the profinite completion. 
This property is known to hold for genus at most $2$ (cf.\ \cite{Asada,hyp}). 
Moreover, we will restrict to $\ast$-isomorphisms between open subgroups of  
procongruence completions $\kG$, namely those which preserve the set of stabilizers for the action on the procongruence curve complex.  
 
We can now formulate our main anabelian result (cf.\ Theorem \ref{mainanabelian}):

\begin{customtheorem}{B}
Let $\kG^\l$ and $\kG^\mu$ be open subgroups of $\kG(S)$ and $\kG(S')$, respectively, (where we assume that $S,S'\neq S_{0,4}$) 
and $\k$ a sub-$p$-adic field over which the associated level structures are both defined. Let $\Isom_{G_\k}^\ast(\kG^\l,\kG^\mu)^\mathrm{out}$ be 
the set of orbits for the action of $\kG^\mu$ by inner automorphisms on the set of $G_k$-equivariant isomorphisms. We then have:
\begin{enumerate}
\item there is a natural bijection $1$-$\Isom_\k(\cM^\l_\k,\cM^\mu_\k)\stackrel{\sim}{\to}\Isom_{G_\k}^\ast(\kG^\l,\kG^\mu)^\mathrm{out}$;
\item the group of generic automorphisms $Z^\mu$ of $\cM_\k^\mu$ identifies with the center of $\kG^\mu$.
\end{enumerate}
\end{customtheorem}

\begin{remark}Since the moduli space of genus $0$ curves with $4$ \emph{ordered} labels $\cM_{0,4}$ identifies with a level structure over 
the moduli stack of $1$-pointed, genus $1$ curves $\cM_{1,1}$, the above theorem also holds 
for $S= S_{0,4}$ (resp.\ $S'= S_{0,4}$) as soon as we assume that the level $\G^\l$ (resp.\ $\G^\mu$) is contained in the \emph{pure} 
mapping class group. For the general case, the problem is to deal with the group of generic automorphisms of the associated level
structure, which, for $S=S_{0,4}$, might be rather complicated. In fact, if we "erase" the group of generic automorphisms (cf.\ Section~\ref{sect:anabelian}), 
we can then state the anabelian result without any restriction on $S$ (cf.\ Theorem~\ref{mainanabelian3}). 
\end{remark}

Thanks to this result, we can also determine an arithmetic version of $\Aut^\ast(\kG(S))$. 
More precisely, we define the \emph{arithmetic procongruence mapping class group} $\kG(S)_\Q$ to be the image 
of the \'etale fundamental group of $\cM(S)_\Q$ via the monodromy representation associated to the universal punctured curve over $\cM(S)_\Q$.
This group fits in the short exact sequence:
\[1\to\kG(S)\to\kG(S)_\Q\to G_\Q\to 1.\]

There is a natural action of $\kG(S)_\Q$ on the procongruence curve complex $\kC(S)$. For an open subgroup $U$ of $\kG(S)_\Q$, we define
$\Aut^\ast(U)$ to be the group of automorphisms of $U$ which preserve the set of stabilizers for the action of $U$ on the set of $0$-simplices $\kC(S)_0$.
We then have the following absolute anabelian (cf.\ Remark~\ref{absan}) version of Theorem~B (cf.\ Theorem~\ref{absoluteanabelian}):

\begin{customtheorem}{C}For $d(S)>1$, let $U$ be an open normal subgroup of the arithmetic procongruence mapping class group 
$\kG(S)_\Q$ and let us denote by $Z(U)$ its center. There is then a short exact sequence:
\[1\to \Hom(U/Z(U),Z(U))\to\Aut^\ast(U)\to\Inn(\kG(S)_\Q)\to 1.\]
In particular, for $S\neq S_{1,2},S_{2}$, we have:
$\Aut^\ast(U)=\Inn(\kG(S)_\Q)\cong\kG(S)_\Q$.
\end{customtheorem}

\tableofcontents

The structure of this article is as follows. 
In Section~\ref{sect1} we collect the relevant definitions for later use, in particular 
we introduce curve complexes and pants complexes in the topological setting.  
In Section~\ref{rigiditycurvecomplex}, first, we explain the rigidity of the curve complex, after Ivanov. 
Then, we give a short proof of the rigidity of the pants complexes, after Margalit, which prepares the ground 
for our approach to the profinite case.   

The main objects of study of the article are the profinite avatars of curve and pants complexes, simple closed curves, braid and Dehn twists.
These are introduced in Section~\ref{sect3}. 

In order to extend Ivanov's approach to the rigidity of curve complexes in the profinite context, we 
need a parameterization of profinite Dehn twists and a description of their centralizers. 
Even though this is presently lacking, the first author obtained these results 
after replacing the full profinite topology on mapping class groups by the more tractable congruence topology. 

In order to keep the exposition self-contained  we state the results from \cite{[B3]} which will be needed later.  
Specifically, Theorem~\ref{multitwists}  gives necessary and sufficient conditions for two procongruence multitwists 
to be equal and Corollary~\ref{centralizers multitwists} shows that the centralizers of procongruence multitwists coincide 
with the stabilizers of the corresponding simplices of the procongruence curve complex. 
In particular, we can define the topological type of a procongruence multitwist, as in the topological case. 
 Along with Theorem~\ref{stabmulticurves}, this also allows to determine
the centralizers of open subgroups of the procongruence mapping class group (cf.\ Theorem~\ref{slim}). 

The main result (Theorem~\ref{typepreservation}) of Section~\ref{rigidityprocurvecomplex} establishes the procongruence 
analogue of the  fact that every automorphism of a curve complex (except for a $2$-punctured torus) 
preserves the topological types of multicurves.  As a first step, we show that two 
procongruence curve complexes are isomorphic if and only if, in the topological case, 
the corresponding curve complexes are isomorphic (cf.\ Theorem \ref{nonisomorphism}). 

Section~\ref{sect6} gives a geometric interpretation of the pants graph and of its procongruence completion in terms, 
respectively, of the Bers bordification of the Teichm\"uller space and of the Deligne-Mumford compactifications of level structures. 
The key observation is that the pants graph describes the $1$-skeleton of a natural triangulation of the $1$-dimensional stratum
of the Bers bordification of the Teichm\"uller space. This implies that the quotient of the pants graph by a level of the mapping class 
group describes the $1$-skeleton of a triangulation of the $1$-dimensional stratum of the Deligne-Mumford compactification of
the associated level structure over the moduli space of curves (cf.\ Proposition \ref{triangulation}). 
The main result of this section is that the procongruence curve complex can be reconstructed from the procongruence pants
graph (cf.\ Theorem \ref{autinjection}). 

In Section~\ref{automorphisms}, we show that the natural actions of the procongruence mapping class group 
on the procongruence curve and pants complexes, with the usual low genera exceptions, are both faithful (cf.\ Theorem \ref{faithfulness}).
We then observe that the action of the procongruence mapping class group on the procongruence curve complex extends to an action of a special
subgroup of the automorphisms group of the procongruence mapping class group. This is the closed subgroup consisting of those automorphisms
which preserve the set of stabilizers for the action of the procongruence mapping class group on the procongruence curve complex
(the so called $\ast$-condition).

The rigidity of the procongruence pants complexes is proved in Section~\ref{prorigidity} (cf.\ Theorem \ref{completepantsrigidity}). 
We start by analyzing the $1$-dimensional case. In this case, the finite quotients by levels of the pants graph identify with the $1$-skeletons 
of natural triangulations of the Deligne-Mumford compactification of the associated level structures and the result is more or less straightforward. 
Once orientations are taken into account, the results of Sections~\ref{rigidityprocurvecomplex} and~\ref{sect6} allow to "globalize"
the $1$-dimensional case to higher dimension.

Section~\ref{sect:anabelian} is devoted to a procongruence anabelian result for moduli stacks of curves with level structures, 
which, roughly speaking, states that isomorphisms between them correspond to orbits of Galois-equivariant $\ast$-isomorphisms  
between the corresponding procongruence levels (cf.\ Theorem \ref{mainanabelian}). 
In the last section, we deduce from it the absolute anabelian result stated at the end of the introduction, which describes the 
$\ast$-automorphisms of arithmetic proncongruence mapping class groups (cf.\ Theorem~\ref{absoluteanabelian}). 
\medskip

\noindent
{\bf Acknowledgements}. 
The starting point of this work was an unachieved project by Marco Boggi and Pierre Lochak (cf.\  \cite{BP}). 
While the first draft of this paper was being elaborated, Pierre Lochak
dropped from it, due to divergences with the other two authors. He later accepted to cosign
the first version of the paper, but, eventually, decided to withdraw from the revised version of the paper. 
We thank him for the discussions we had. We are grateful to the referees for their valuable comments  which helped us 
to improve the paper.

\section{Definitions}\label{sect1}

\subsection{} 
A surface $S$ is of (topological) type $(g,n)$ if 
it is diffeomorphic to $S_{g,n}$, namely the closed orientable surface of genus $g$ with $n$ deleted points.
We occasionally write $g(S)$ for the genus of $S$. 
The surface $S_{g,n}$ is {\it hyperbolic} if  $2g-2+n>0$. We will restrict henceforth to hyperbolic surfaces.

\subsection{}
Throughout this paper, we denote by $Z(G)$ the center of a group $G$ and, given a subgroup $H$ of $G$, we denote by $Z_G(H)$ the centralizer 
of $H$ in $G$ and by $N_G(H)$ the normalizer of $H$ in $G$.

\subsection{}\label{moduli}
Attached to a hyperbolic surface $S$ of type $(g,n)$ are the {\it Teichm\"uller space} $\cT(S)$ and the \emph{ Deligne--Mumford (DM) moduli stack} 
$\cM(S)$, parameterizing smooth algebraic curves whose complex model is diffeomorphic to $S$. 
We will also consider the \emph{DM compactification} $\bM(S)$ of $\cM(S)$. This is the DM moduli stack which parameterizes the \emph{stable curves}
obtained as nodal degenerations of the curves parameterized by $\cM(S)$. All these moduli stacks have dimension $d(S)=d_{g,n}=3g-3+n$, which we then 
call the \emph{modular dimension} of $S$ or of the corresponding type. 

For $S=S_{g,n}$, sometimes, we simply write $\cT_{g,n}$ for $\cT(S_{g,n})$, $\cM_{g,[n]}$ for $\cM(S_{g,n})$ and $\bM_{g,[n]}$ for $\bM(S_{g,n})$.
We use the brackets ($[n]$) to stress that the points are unlabelled, that is to say, they are considered as an unordered set.  
Instead, we denote by $\cM_{g,n}$ (resp.\  $\bM_{g,n}$) the moduli stack of algebraic projective curves of genus $g$ with $n$ labelled points (or punctures).

\subsection{} Let $\Mod(S)$ denote the \emph{extended mapping class 
group} of $S$, i.e.\ the group of isotopy classes of diffeomorphisms of $S$. The index 2
subgroup of orientation preserving isotopy classes is denoted $\Mod^+(S)$.  More generally 
an upper $+$ will mean {\it orientation preserving}.  We write 
$\G(S)=\Mod^+(S)$ and call it the {\it (Teichm\"uller) modular group} or, simply, the \emph{mapping class group}. It can be identified with
the topological fundamental group of the complex DM stack $\cM(S)_\C$ and then with the covering transformation group of the 
unramified cover $\cT(S)\to\cM(S)_\C$. So we have the tautological exact sequence:
\begin{equation}\label{A1}1\rightarrow \G(S)\rightarrow \Mod(S) \rightarrow \{\pm 1\} \rightarrow 1.
\end{equation} 

Let $\PG(S)$ be the \emph{pure} mapping class group of $S$, that is to say, the subgroup of $\G(S)$ consisting of elements which pointwise 
preserve each puncture of $S$ (note that in \cite{[B1]}, \cite{hyp}, \cite{[B3]}, \cite{[BZ2]} and \cite{congtop} the pure 
mapping class group is denoted by $\G(S)$). It is described by the short exact sequence:
\begin{equation}\label{A2}1\rightarrow \PG(S)\rightarrow \G(S) \rightarrow \Sigma_n\rightarrow 1,
\end{equation} 
where $\Sigma_n$ is the symmetric group on $n$ letters.

The group $\G(S)$ (resp.\ $\PG(S)$) is centerfree, except for the types $(1,1)$, $(1,2)$ and $(2,0)$ (resp.\ $(1,1)$ and $(2,0)$), 
where the center is isomorphic to the cyclic group of order $2$ and is generated by the hyperelliptic involution. 

Sometimes, for $S=S_{g,n}$, we will denote $\PG(S)$ and $\G(S)$ by $\G_{g,n}$ and $\G_{g,[n]}$, respectively.

\subsection{}\label{levels} A \emph{level} $\G^\l$ is a finite index subgroup of $\G(S)$. The \emph{level structure} $\cM^\l$ is the finite \'etale covering 
of the DM stack $\cM(S)$ which is naturally associated to the level $\G^\l$ by the above identification of $\G(S)$
with the topological fundamental group of $\cM(S)_\C$. An important role in this paper will be played by the \emph{DM compactification}
$\bM^\l$ of $\cM^\l$, which is defined to be the normalization of the DM compactification $\bM(S)$ in the function field of $\cM^\l$. 
We also call $\bM^\l$ the \emph{level structure over $\bM(S)$ associated to $\G^\l$}.

\subsection{}
We now briefly summarize the definitions pertaining to various 
{\it curve complexes}, referring to any of the many references (e.g. \cite{[I1], [I2], [L]} etc.) for more details. 

Given a hyperbolic surface $S$ of finite type, let $\cL(S)$ denote the set 
of isotopy classes of simple closed curves on $S$ and $\cL(S)_0$ the subset consisting of non peripheral curves. 
A {\it multicurve} $\s$ is a set of distinct elements of $\cL(S)_0$ which admit disjoint representatives. We say that $\s$ is a $k$-multicurve if it
consists of $k$ elements. The curve complex $C(S)$ is the abstract simplicial complex whose $k$-simplices
are $(k+1)$-multicurves $\ul \a=(\a_0,\ldots,\a_k)$. 

Note that $C(S)$ is a (non locally finite) simplicial complex of dimension $d(S)-1$ where $d(S)$ is the 
modular dimension of $S$ (see \S 2.2).  We will write $C^{(k)}(S)$ for 
the $k$-dimensional skeleton of $C(S)$ and use a similar notation for the other complexes.
There is a natural action of $\G(S)$ on $C(S)$.

\subsection{}\label{DehnBraids}
The set $\cL(S)_0$ parameterizes the \emph{Dehn twists} of $\G(S)$ (cf.\ Section~3.1.1 in \cite{[FM]} for the definition). 
We denote by $\tau_\g$ the right 
Dehn twist associated to an element $\g\in\cL(S)_0$. Such Dehn twists generate the pure mapping class group $\PG(S)$. 
Multicurves then parameterize sets of  pairwise commuting distinct Dehn twists.

Although Dehn twists generate the pure mapping class group $\PG(S)$, we need more elements 
to generate the full mapping class group $\G(S)$.
Denote by $\cL^b(S)\subset \cL(S)$ the classes of simple closed curves bounding a 
$2$-punctured disc on $S$. For $\g\in\cL^b(S)$, let $D\subset S$ be a disc with boundary $\dd D=\g$. 
The mapping class group $\G(D,\dd D)$ is isomorphic to $\Z$ and a standard 
generator is the braid turning once to the right and interchanging the two punctures.   
The \emph{braid twist} $b_\g$ about $\g$ is the image in $\G(S)$ of this braid via the natural monomorphism $\G(D,\dd D)\hookra\G(S)$.  
Note that $b_\g^2=\tau_\g$. Moreover, the mapping class group $\G(S)$ is generated by Dehn twists and braid twists.

\subsection{}\label{pantsgraphdef}
The \emph{pants complex} $C_P(S)$ was mentioned in
the appendix of the classical  paper by A. Hatcher and W. Thurston (see \cite{[HLS]} or \cite{[M]})
and first studied in \cite{FG} and \cite{[HLS]} where it is shown to be connected and simply connected for $d(S)>2$. 
It is a two dimensional, not locally finite, simplicial complex whose vertices are given by the pants
decomposition (i.e.\ maximal multicurves) of $S$; these correspond to the facets (simplices of highest
dimension $=d(S)-1$) of $C(S)$. Given two vertices
$\ul \a,\ul \a'\in C_P(S)$, they are connected by an edge if
 $\ul \a$ and $\ul \a'$ have $d(S)-1$ elements in common,
so that up to relabelling $\a_i=\a'_i$, $i=1,\ldots, d(S)-1$,
whereas $\a_0$ and $\a'_0$ differ by an {\it elementary move}, which means the following.
Cutting $S$ along the $\a_i$'s, $i>0$, there remains a surface $S'$ of modular
dimension $1$, so $S'$ is of type $(1,1)$ or $(0,4)$. Then, $\a_0$ and $\a'_0$,
which are supported on $S'$, should intersect in a minimal way, i.e.\ they 
have geometric intersection number $1$ in the first case and $2$, in the second case. 

We have thus defined the $1$-skeleton $C_P^{(1)}(S)$
of $C_P(S)$ which, following \cite{[M]}, we call the {\it pants graph} of $S$. We will not give
here the definition of the $2$-cells of $C_P(S)$ (see \cite{[HLS]} or \cite{[M]}), since we will not actually
use them. Here, it suffices to say that, for $d(S)=1$, the pants complex coincides with the Farey tessellation of the hyperbolic plane.
It is shown in \cite{[M]} how to recover the  $2$-dimensional pants complex from the pants graph.

We will only use the pants graph, i.e.\ the $1$-skeleton $C_P^{(1)}(S)$ of $C_P(S)$, which in order to simplify notation we will simply denote by $C_P(S)$. 
For $d(S)=1$, this is the $1$-skeleton of the Farey tessellation which we call the \emph{Farey graph} and denote by $F$.

\subsection{}
Sometimes it will be useful to consider a \emph{disconnected} surface $S$ such that all its connected components are hyperbolic
surfaces of the above type. It is easy to reformulate all the above definitions for this case. Thus the mapping class groups $\G(S)$
and $\Mod(S)$ are just the direct product of the corresponding mapping class groups of the connected components of $S$.  The same holds 
for the moduli stack and Teichm\"uller space associated to $S$. 
The curve complexes $C(S)$ and $C_P(S)$ are defined exactly in the same way in the connected and disconnected case.
It is not difficult to see that, if $S=\coprod_{i=1}^k S_i$ is the decomposition of $S$ in connected components, then we have:
\begin{itemize}
\item $C(S)=\star_{i=1}^kC(S_i)$, where $\star$ denotes the join of simplicial complexes;
\item $C_P(S)=\coprod_{i=1}^k C_P(S_i)$. 
\end{itemize}
Let us observe that $C(S_i)=C_P(S_i)=\emptyset$, when $S_i$ is a $3$-punctured sphere, and the empty set is the neutral element both
for the join and the disjoint union (coproduct) operators.

\section{Rigidity of curve complexes}\label{rigiditycurvecomplex}
In this section we prepare the ground by recalling some rigidity results for the various curve complexes introduced above in a manner
which will be later adapted to the procongruence setting; as a side benefit, it provides a simpler proof of the 
main result of \cite{[M]}, that is the rigidity of the pants graph. 

\subsection{Automorphisms of the curve complex} Let $\Aut(C(S))$ be the group of simplicial 
automorphisms of the curve complex $C(S)$. There is a natural homomorphism $\Mod(S)\rightarrow \Aut(C(S))$ induced by the action of 
diffeomorphisms on the set of simple closed curves. 

It is useful to introduce a group theoretic version of this map. Let $C_\cI(S)$ be the abstract simplicial complex whose
set of $k$-simplices consists of the set of abelian subgroups of $\G(S)$ of rank $k+1$ generated by Dehn twists. We denote by $\tau_\g$ the Dehn twist
about the simple closed curve $\g$ on $S$. For all $f\in\Mod(S)$, we then have the identity:
\[f\cdot\tau_\g\cdot f^{-1}=\tau_{f(\g)}^{\epsilon(f)},\]
where $\epsilon\co\Mod(S)\to\{\pm 1\}$ is the orientation character. Therefore, conjugation determines
a simplicial action of the extended mapping class group $\Mod(S)$ on $C_\cI(S)$ and there is a natural $\Mod(S)$-equivariant surjective map of
simplicial complexes $C(S)\to C_\cI(S)$, defined assigning to a multicurve $\s$ the abelian subgroup of $\G(S)$ generated by the Dehn twists $\tau_\g$,
for $\g\in\s$. Since a Dehn twist $\tau_\g$ is determined by the isotopy class $\g$, this map is an isomorphism. We will hence identify $C(S)$ with 
$C_\cI(S)$ and the natural geometric action of $\Mod(S)$ on multicurves with its action by conjugation on abelian subgroups of $\G(S)$.

In particular, we see that the homomorphism $\Mod(S)\rightarrow \Aut(C(S))$ factors through a homomorphism 
$\theta\co \Inn(\Mod(S))\rightarrow \Aut(C(S))$, where $\Inn(G)$ denotes the group of inner automorphisms of a group $G$. 
From the description of centralizers of Dehn twists in $\Mod(S)$, it follows that, for $S\neq S_{0,4}$, this homomorphism is injective.
A fundamental result by Ivanov (cf.\ \cite{[I2]}) then asserts that, in most cases, $\theta$ is also surjective. 
Ivanov's theorem was subsequently refined by Luo  (cf.\ \cite{[L]}), who settled the exceptional cases. The precise statement is as follows:

\begin{theorem}\label{rigiditycurves}Let $S$ be a connected hyperbolic surface of type $(g,n)$ with $d(S)>1$. Then, 
the natural map $\theta: \Inn(\Mod(S))\to\Aut(C(S))$ is an isomorphism
except if $(g,n)=(1,2)$, in which case it is injective but not surjective; in fact
$\theta$ maps $\Inn(\Mod(S_{1,2}))$ onto the proper subgroup consisting of the elements of $\Aut(C(S_{1,2}))$ 
which globally preserve the set of vertices representing nonseparating simple closed curves.
\end{theorem}

As Ivanov showed, Theorem \ref{rigiditycurves} allows to determine the automorphisms groups
of the mapping class groups. The basic result needed here is that the group-theoretical action $\theta\co \Inn(\Mod(S))\rightarrow \Aut(C(S))$
extends to an action of the automorphism group of $\Mod(S)$. As shown by McCarthy in Section~1 of \cite{[McC]}, in genus $\leq 2$, 
it is not always the case that an automorphism of $\Mod(S)$ preserves the cyclic subgroups generated by Dehn twists. So we need to
tweak a little the definition of $\theta$ in order to be able to extend it to all $\Aut(\Mod(S))$. 
The crucial remark is that the stabilizer $\Mod(S)_\s$ of a simplex $\s\in C(S)$ for the action of $\Mod(S)$
 is a self-normalizing group. Therefore, if we define $C_\cG(S)$ to be the abstract simplicial 
complex whose set of $k$-simplices consists of the set of stabilizers of $k$-simplices in 
$C(S)$, there is a natural $\Mod(S)$-equivariant
isomorphism $C(S)\to C_\cG(S)$, defined by the assignment $\s\mapsto\Mod(S)_\s$. It is a deep fact 
that the set of subgroups $\{\Mod(S)_\s\}_{\s\in C(S)}$ is preserved by an automorphism of $\Mod(S)$. Therefore, we get a natural homomorphism: 
\[\Theta\co\Aut(\Mod(S)))\rightarrow \Aut(C(S)),\] 
which extends $\theta$. The complete result by Ivanov (cf. \cite{[I1]} and references therein) is the following:

\begin{theorem}\label{inertiapreserving} 
The group theoretic action of $\Inn(\Mod(S))$ on $C(S)$ extends to an action of $\Aut(\Mod(S))$ and, if the center of 
$\Mod(S)$ is trivial and $S\neq S_{0,4}$, then the resulting natural homomorphism $\Theta\co\Aut(\Mod(S))\to\Aut(C(S))$ is injective.
\end{theorem}

\begin{proof}We already observed that the natural homomorphism $\theta: \Inn(\Mod(S))\to\Aut(C(S))$ is injective for $S\neq S_{0,4}$.
Then, the theorem follows from this fact and the group theoretic lemma:

\begin{lemma}\label{grouplemma}Let $G$ be a group with trivial center and $H$  a subgroup of $\Aut(G)$ containing $\Inn(G)$. If
$\phi\co H\to K$ is a homomorphism of groups such that its restriction to $\Inn(G)$ is injective, then the homomorphism $\phi$ is also injective.
\end{lemma}

\begin{proof}For all $f\in H$ and $x\in G$, we have the identity:
\[\phi(f)\cdot\phi(\inn x)\cdot\phi(f)^{-1}=\phi(f\circ\inn x\circ f^{-1})=\phi(\inn f(x)).\]
If $f\neq\mathrm{id}$, then there is some  $a\in G$ such that $\inn a\neq \inn f(a)$, because the center of $G$ is trivial. 
If $\phi(f)$ were trivial, then the above identity would imply that $\phi(\inn a)=\phi(\inn f(a))$, in contradiction
with the hypothesis that the restriction of $\phi$ to $\Inn(G)$ is injective.
\end{proof}
\end{proof}

Throughout this paper we denote by $Z(G)$ the center of a group $G$.
For future applications, it is interesting to consider in detail also the cases when the center $Z(\Mod(S))$ of $\Mod(S)$ is not trivial. 
The $\Mod(S)$-equivariant isomorphism $C(S)\cong C_\cG(S)$ implies, more generally, that, for a finite index subgroup $\G^\l$ of $\Mod(S)$, 
there is a natural homomorphism: 
\[\Theta^\l\co\Aut(\G^\l)\rightarrow \Aut(C(S)).\] 

\begin{proposition}\label{kernelaut}
For $\G^\l$ a finite index subgroup of $\Mod(S)$, where $d(S)>1$, there is a natural isomorphism $\ker\Theta^\l\cong\Hom(\G^\l/Z(\G^\l),Z(\G^\l))$.\end{proposition}

\begin{proof}By Theorem~\ref{inertiapreserving}, for $d(S)>1$, an automorphism $f\in\Aut(\G^\l)$ which acts trivially on the curve complex $C(S)$ 
descends to the identity on the quotient $\G^\l/Z(\G^\l)$. The conclusion then follows from the lemma:

\begin{lemma}\label{kernelautbis}Let $G$ be a group and $\Aut(G)_A$ the subgroup of elements of $\Aut(G)$ which preserve a normal
abelian subgroup $A$ of $G$. Then, there is a natural exact sequence:
\[1\to\Der(G/A,A)\to\Aut(G)_A\to\Aut(G/A)\times\Aut(A),\]
where the action of $G/A$ on $A$ is induced by the inner action of $G$ on $A$. 
If, moreover, the subgroup $A$ is central, there is also an exact sequence:
\[1\to H^1(G/A,A)\to\Out(G)_A\to\Out(G/A)\times\Aut(A).\]
\end{lemma}
 
\begin{proof}The first exact sequence is a weak form of Wells' Theorem in \cite{Wells} which can be proved as follows. 
An automorphism $f\in\Aut(G)_A$ whose image in $\Aut(G/A)$ is trivial and which restricts to the identity on $A$, is of the form $x\mapsto a\cdot x$, 
for some $a\in A$ depending on $x\in G$. Note that, for $x\in A$, we have $a=1$.
Let us then define a map $L_f\co G\to A$ by $L_f(x):=f(x)x^{-1}$. 
This map factors through a map $\log f\co G/A\to A$ which satisfies the identity:
\[\log f(xy)=f(x)f(y)(xy)^{-1}=f(x)x^{-1}\inn(x)(f(y)y^{-1})=\log f(x)\cdot\inn(x)(\log f(y)),\] 
where $\inn(x)$ denotes conjugation by $x$, so that $\log f$ is a derivation from $G/A$ to $A$. 

Conversely, for $\phi\in\Der(G/A,A)$, the map $\exp\phi\co G\to G$, defined by $\exp\phi(x):=\phi(\bar{x})x$, 
where we let $\overline{x}$ be the image of $x\in G$ in the quotient $G/A$, is an automorphism of $G$ which descends 
to the identity on $G/A$ and restricts to the identity on $A$. 
 
The second exact sequence then follows from the first one observing that the group of inner automorphism of $G$ intersects
$\Der(G/A,A)$ inside $\Aut(G)_A$ exactly in the subgroup of principal derivations.
\end{proof}
\end{proof}

Thanks to Theorem~\ref{inertiapreserving} and Proposition~\ref{kernelaut}, we can use Theorem~\ref{rigiditycurves} in order to study 
the automorphisms of $\G(S)$. Actually, it turns out to be no more difficult to study isomorphisms between 
two finite index subgroups of $\Mod(S)$ (cf.\ \cite{[I2]}, Theorem 2). 

\begin{definition}\label{transporter}Given two groups $G_1$ and $G_2$, let us denote by $\Isom(G_1,G_2)$ the set of isomorphisms 
between them. Given two subgroups $H_1$ and $H_2$ of a group $G$, we define the \emph{transporter} of $H_1$ onto $H_2$ to
be the set $T_{G}(H_1,H_2)\subseteq\Inn(G)$ of inner automorphisms of $G$ which map $H_1$ onto $H_2$.
Observe that the first set has a natural structure of $\Aut(G_2)$-torsor and the second one of $N_{G}(H_2)/Z(G)$-torsor.
\end{definition}

We then have:

\begin{corollary}\label{Theorem 2.4} For $d(S)>1$, let $\G_1$ and $\G_2$ be finite index subgroups of $\Mod(S)$. 
Then, the set $\Isom(\G_1,\G_2)$ is a $\Hom(\G_2/Z(\G_2),Z(\G_2))$-torsor over the set $T_{\Mod(S)}(\G_1,\G_2)$. 
In particular, if $\G_1=\G_2$ is a normal subgroup of $\Mod(S)$ and $Z(\G_1)=\{1\}$, then $\Aut(\G_1)=\Inn(\Mod(S))$.
\end{corollary}

\subsection{Automorphisms of the pants graph} \label{discrec}
In this paragraph, we study the group of simplicial automorphisms $\Aut(C_P(S))$ of the pants graph. A result similar 
to Theorem~\ref{rigiditycurves} was proved by Margalit (see \cite{[M]}): 

\begin{theorem}\label{rigiditypants}
Let $S$ be a hyperbolic surface of type $(g,n)$ with $d(S)>1$. Then the natural homomorphism 
\[\theta_P:\Inn(\Mod(S))\rightarrow \Aut(C_P(S))\]
is an isomorphism.
\end{theorem}

Note that, in contrast with Theorem~\ref{rigiditycurves}, the case $(1,2)$ is no exception, see  the last paragraph of \cite{[M]}.
However, as we shall see, for $(g,n)\neq(1,2)$, Theorem~\ref{rigiditypants} is  a  consequence of Theorem~\ref{rigiditycurves}.

Let us recall some generalities on (abstract) simplicial complexes. The maximal simplices of a simplicial complex $X$ are called \emph{facets}. 
A simplicial complex whose facets all have the same dimension is \emph{pure}. In this case, $\dim X$ is finite and equal to the dimension of the facets 
of $X$. 

\begin{definition}\label{dualgraph}The \emph{dual graph} $X^\ast$ of an abstract pure simplicial complex $X$ is the $1$-dimensional abstract
simplicial complex whose set of vertices $X_0^\ast$ consists of the facets of $X$ and whose edges are pairs of facets $\{v_0,v_1\}$
such that $v_0$ and $v_1$ intersect in a simplex of dimension $\dim X-1$.
\end{definition}

Note that the set of all facets of $X$ which contain a given $(\dim X-1)$-simplex span a complete subgraph of $X^\ast$. 
The following lemma seems known to experts but, for lack of a reference, we give a proof:

\begin{lemma}\label{maxcomplete}Let $X$ be a pure simplicial complex of dimension $d$ such that every $(d-1)$-simplex is contained in at least
$d+3$ facets. Then, the $(d-1)$-simplices of $X$ are parameterized by the maximal complete subgraphs of $X^\ast$ with at least $d+3$ vertices. 
Moreover, two $(d-1)$-simplices of $X$ are contained in the same facet if and only if the corresponding maximal complete subgraphs of $X^\ast$ 
have a common vertex. Therefore, from the dual graph $X^\ast$, we can reconstruct the $1$-skeleton of $X$.
\end{lemma}
 
\begin{proof}To a $(d-1)$-simplex $\s$ of $X$, we associate the complete subgraph $\cG_\s$ of $X^\ast$ spanned by the facets which contain $\s$.
Let us show that $\cG_\s$ is maximal with that property. Observe that a facet $f$ which does not 
contain $\s$ can share with a given facet $f'$ containing $\s$ at most one  $(d-1)$-simplex $\sigma'$. 
Moreover, $\sigma'$ cannot be shared with any other facet $f''$ containing $\s$ but distinct from $f'$. 
But $\s$ is contained in at least $d+3$ facets, by hypothesis, while 
a facet has only $d+1$ faces of dimension $(d-1)$.

Note further that the cardinality of a set of facets, whose intersection has dimension $< d-1$
and each pair of which have a common $(d-1)$-simplex face, is bounded by $d+2$.
This proves the first part of the lemma.

The second statement of the lemma is obvious. For the third, observe that a vertex of $X$ is determined by a facet $f$, which contains it, and by a 
$(d-1)$-dimensional face of $f$, which does not contain it. Therefore there is a natural bijection between the vertices of $X$ and pairs consisting of
a vertex $v$ of $X^\ast$ and a maximal complete subgraph of $X^\ast$ containing $v$, with at least $d+3$ vertices. Observe then that two vertices
of $X$ are joined by an edge if and only if they are both contained in a facet $f$ of $X$. The conclusion then follows from the previous point.
\end{proof}

Recall that a \emph{flag complex} $X$ is an (abstract) simplicial complex such that every set of 
vertices of $X$ which pairwise belong to a simplex of the complex is itself a simplex of $X$.

\begin{proof}[Proof of Theorem~\ref{rigiditypants} for $(g,n)\neq(1,2)$]  By its definition, the curve complex $C(S)$ is a flag complex. 
In particular, $C(S)$ is determined by its $1$-skeleton $C^{(1)}(S)$.
Let $C^\ast(S)$ be the dual graph of $C(S)$. From this remark and Lemma~\ref{maxcomplete}, it follows that there is a series of natural isomorphisms
\[\Aut(C(S))\cong\Aut(C^{(1)}(S))\cong\Aut(C^\ast(S)),\] 
compatible with the action of the mapping class group $\Mod(S)$.

Therefore, Theorem~\ref{rigiditypants}, for $(g,n)\neq(1,2)$, is a direct consequence of Theorem~\ref{rigiditycurves} along with the following lemma:

\begin{lemma}\label{proofrigiditypants}There is a natural monomorphism $\Aut(C_P(S))\hookrightarrow\Aut(C^\ast(S))$ compatible with the action of 
the mapping class group $\Mod(S)$.
\end{lemma}

\begin{proof}As we observed in Section~\ref{pantsgraphdef}, for $d(S)=1$, the pants graph $C_P(S)$ coincides with the Farey graph $F$. 
In general, we can associate to a $(d(S)-2)$-simplex $\s$ of $C(S)$ the full subcomplex $F_\s$ of $C_P(S)$ 
whose vertices are the facets of  $C(S)$ which contain $\s$. The graph $F_\s$ is clearly naturally isomorphic to the Farey graph associated to the 
connected component $S_\s$ of $S\smallsetminus\s$ of positive modular dimension. The Farey subgraphs $F_\s$, for $\s\in C(S)_{d(S)-2}$, cover the
pants graph $C_P(S)$. The following remarks are then elementary:
\begin{enumerate}
\item Every edge and every triangle of $C_P(S)$ is contained in a unique Farey subgraph.
\item The Farey subgraph $F_e$, containing the edge $e$ of $C_P(S)$, is obtained by the following procedure: intersect the stars 
$\Star_{v_0}$ and $\Star_{v_1}$ in $C_P(S)$ of the vertices $v_0,v_1$ of $e$; take the full subcomplex of $C_P(S)$ generated by 
$e\cup(\Star_{v_0}\cap \Star_{v_1})$; iterate the procedure on all the new edges thus obtained and continue this way; the union of all 
these subcomplexes is then the Farey subgraph $F_e$ associated to the edge $e$. 
\end{enumerate}

The above remarks imply that every automorphism $\phi$ of $C_P(S)$ brings the Farey subgraph $F_e$, associated to an edge $e$ of $C_P(S)$, to
the Farey subgraph $F_{\phi(e)}$, associated to the edge $\phi(e)$. In particular, the automorphisms of $C_P(S)$ preserve the subdivision of $C_P(S)$
in Farey subgraph. At this point, Lemma~\ref{proofrigiditypants} (and so Theorem~\ref{rigiditypants}, for $(g,n)\neq(1,2)$) follows from the remark 
that $C^\ast(S)$ is obtained from $C_P(S)$ by replacing each Farey subgraph $F_\s$ of $C_P(S)$ with the complete graph on the same set of 
vertices of $F_\s$.
\end{proof}
\end{proof}

\section{The procongruence mapping class group and procongruence curve complex}\label{sect3}
In this section we define the profinite completions of the mapping class group and of the simplicial complexes which were introduced in the
previous section. We focus on the procongruence completion because crucial results are {\it not} available to-date 
for the {\it full} profinite completions (if $g\geq 3$), as will become clear below. These concepts were introduced in \cite{[B1], [B3]} and we refer to 
those papers for more details and complete proofs.

\subsection{Profinite groups and the procongruence completion of the mapping class group} 
Given a group $G$ and a downward directed poset of normal subgroups of finite index $\{H_i,\subseteq\}_{i\in\ccI}$, we define on $G$ the 
\emph{pro-$\ccI$ topology}, which has for basis of open subsets all cosets of the subgroups in $\{H_i,\subseteq\}_{i\in\ccI}$. In this way, $G$ becomes
a topological group. The \emph{pro-$\ccI$ completion} $\wh{G}^\ccI$ of the group $G$ is the completion with respect to this topology and can be explicitly
constructed observing that the set of quotients $\{G/H_i\}_{i\in\ccI}$ together with the natural maps $p_{ij}\co G/H_i\to G/H_j$, for
$H_i\subseteq H_j$, forms an \emph{inverse system}. We have:
\[\widehat{G}^\ccI=\varprojlim_{i\in\ccI}G/H_i.\]
The profinite group $\widehat{G}^\ccI$ comes with a natural homomorphism with dense image $G\to\widehat{G}^\ccI$. The closures of the subgroups 
$\{H_i,\subseteq\}_{i\in\ccI}$ in $\widehat{G}^\ccI$ form a fundamental system of \emph{clopen} (i.e.\ closed and open at the same time) neighborhood 
of the identity for the topology on $\widehat{G}^\ccI$. We will use the 
generic terms of pro-topology and pro-completion, unless $\ccI$ needs to be explicitly mentioned. When $\{H_i,\subseteq\}_{i\in\ccI}$ is the full poset of subgroups of finite index, 
we simply write $\widehat{G}$ for $\widehat{G}^\ccI$ and call it the \emph{profinite completion} of $G$.
 
The profinite completion $\hP$ of the fundamental group $\Pi:=\pi_1(S)$ (we omit base points because they are irrelevant here) 
determines a pro-completion of the mapping class group $\G(S)$ in the following way. 
A characteristic (or just $\G(S)$-invariant) finite index subgroup $K$ of $\Pi$ determines the {\it geometric level}
$\G_K$, defined to be the kernel of the induced representation:
\[\rho_K\co \G(S)\to\Out(\Pi/K).\]
The abelian levels $\G(m)$ of order $m$ are a particular case and are defined to be 
the kernel of the natural representations, for $m\geq 2$:
\[\rho_{(m)}\co \G(S)\to\Sp(H_1(\ol{S},\Z/m)),\]
where $\ol{S}$ is the closed surface obtained from $S$ filling in the punctures. 

The \emph{congruence topology on} $\G(S)$ is the pro-topology which has for fundamental system of neighborhoods of the identity 
the set of geometric levels $\{\G_K\}_{K\lhd\Pi}$. A \emph{congruence level} $\G^\l$ of $\G(S)$ is a finite index subgroup which is open 
for the congruence topology, i.e.\ contains a geometric level $\G_K$, for some characteristic finite index subgroup $K$ of $\Pi$.

The {\em procongruence completion of the mapping class group} (or simply {\em procongruence mapping class group}) $\kG(S)$ is the completion of $\G(S)$
with respect to this topology. By definition, there is a natural faithful continuous representation $\kG(S)\hookrightarrow\Out(\hP)$ and, since $\Pi$ 
is conjugacy separable, the natural homomorphism $\G(S)\to\kG(S)$ is injective.

There is a natural surjective homomorphism 
$ \widehat\G(S) \to \kG(S)$ and the \emph{congruence subgroup problem} (first proposed by  Ivanov) asks whether 
this is actually an isomorphism, which amounts to asking whether the geometric levels
form a cofinal system in the poset of all finite index subgroups of $\G(S)$. 
The conjecture has been proved for $g(S)\leq 2$ (cf.\ \cite{Asada}, \cite{hyp} and \cite{congtop}) but remains open for $g\geq 3$.

\subsection{Pro-completions of $G$-simplicial sets}\label{G-completion}
Let $X_\bullet$ be a simplicial set endowed with a simplicial action of a group $G$ such that $X_k$, for all $k\geq 0$, consists of finitely many $G$-orbits. 
For a downward directed poset of normal finite index subgroup $\{H_i,\subseteq\}_{i\in \ccI}$, we define its pro-completion 
${\widehat X}^\ccI_\bullet$ to be the simplicial profinite set:
\[{\widehat X}^\ccI_\bt=\varprojlim_{i\in \ccI} X_\bt/H_i.\]
Let us observe that, for any $\ccI$-open subgroup $\wh{H}^\ccI$ of $\wh{G}^\ccI$, the $\wh{H}^\ccI$-completion of $X_\bt$ is naturally isomorphic to
the $\wh{G}^\ccI$-completion ${\widehat X}^\ccI_\bullet$.
There is a natural map with dense image $\iota\co X_\bt\to{\widehat X}^\ccI_\bullet$ and the simplicial profinite set ${\widehat X}^\ccI_\bullet$ is 
characterized by the universal property that, for any simplicial profinite set $Y_\bt$ endowed with a continuous $\wh{G}^\ccI$-action and a 
$G$-equivariant map $f\co X_\bt\to Y_\bt$, there exist a unique continuous $\wh{G}^\ccI$-equivariant map $f'\co \wh{X}^\ccI_\bt\to Y_\bt$ such that
$f=f'\circ\iota$. 

\subsection{Automorphisms of simplicial profinite sets} 
Homomorphisms between simplicial profinite sets are always meant to be continuous. In particular, for a simplicial profinite set $X_\bt$, 
we define $\Aut(X_\bt)$ to be the group of continuous automorphisms of $X_\bt$. We have:

\begin{proposition}\label{autprofinite}The automorphism group $\Aut(X_\bt)$ of a simplicial profinite set $X_\bt$ is a profinite group.
\end{proposition}

The proposition is a particular case of the following more general topological lemma: 

\begin{lemma}\label{lem:autprofinite}
Let $\mathcal C$ be a category, $\mathcal H$ be the category of Hausdorff topological spaces
and $F\co\mathcal C\to\mathcal H$ be a functor. There is a canonical injective map: 
\[\Aut(F) \to \prod_{X\in {\rm Ob}(\mathcal C)} \Aut(F(X)),\]
which identifies $\Aut(F)$ with a closed subgroup of $\prod_{X\in {\rm Ob}(\mathcal C)} \Aut(F(X))$, where we endow $ \Aut(F(X))$ 
with the compact open topology. In particular, if each F(X) is a profinite set, then $\Aut(F)$ is a profinite group.
\end{lemma}

\begin{proof}The proof is essentially the same as that of Lemma~\cite[\href{https://stacks.math.columbia.edu/tag/0BMR}{Tag 0BMR}]{stacks-project}.
Let $\xi=(\g_X)\in\prod_{X\in {\rm Ob}(\mathcal C)} \Aut(F(X))$ be an element not in $ \Aut(F(X))$. Then, there exists a morphism $f\co X\to X'$
of $\mathcal C$ and a point $x\in F(X)$ such that $F(f)(\g_X(x))\neq\g_{X'}(F(f)(x))$. By the Hausdorff property, there are open disjoint neighborhoods 
$V\ni F(f)(\g_X(x))$ and $V'\ni\g_{X'}(F(f)(x))$. 
Consider the open neighborhoods $U=\{\g\in\Aut(F(X))|\,\g(x)\in V\}$ and $U'=\{\g'\in\Aut(F(X'))|\,\g'(F(f)(x))\in V'\}$. 
Then, $U\times U'\times\prod_{X''\neq X,X'} \Aut(F(X''))$ is an open neighborhood of
$\xi$ not meeting $\Aut(F) $. 

The final statement of the lemma follows from the fact that $\prod_{X\in {\rm Ob}(\mathcal C)} \Aut(F(X))$ is a profinite group, 
since the group of continuous bijection of a profinite set, endowed with the compact open topology, is a profinite group. 
\end{proof}

Let us now consider simplicial profinite sets of the type introduced in Section~\ref{G-completion}. Let then ${\widehat X}^\ccI_\bullet$ be the
profinite $\wh{G}^\ccI$-completion of a $G$-simplicial set $X_\bt$. Recall that a profinite group is strongly complete when every finite
index subgroup is open or, equivalently, when it coincides with its own profinite completion. We have:

\begin{proposition}\label{autGcompletion}
If the profinite group $\wh{G}^\ccI$ is strongly complete, then $\Aut({\widehat X}^\ccI_\bullet)$ is also a strongly complete profinite group.
\end{proposition}

\begin{proof}There is a natural faithful continuous action of the profinite group $\Aut({\widehat X}^\ccI_\bullet)$ on ${\widehat X}^\ccI_\bullet$ such that
the natural action of $\wh{G}^\ccI$ on ${\widehat X}^\ccI_\bullet$ factors through it and a natural continuous homomorphism of profinite groups
$\wh{G}^\ccI\to\Aut({\widehat X}^\ccI_\bullet)$. It is then clear that the $\wh{\Aut}({\widehat X}^\ccI_\bullet)$-completion 
$({\widehat X}^\ccI_\bullet)^\wedge$ of
${\widehat X}^\ccI_\bullet$ is dominated by its $(\wh{G}^\ccI)^\wedge$-completion. From $(\wh{G}^\ccI)^\wedge=\wh{G}^\ccI$, it  follows that 
$({\widehat X}^\ccI_\bullet)^\wedge={\widehat X}^\ccI_\bullet$. Since the profinite group $\wh{\Aut}({\widehat X}^\ccI_\bullet)$ acts faithfully on 
$({\widehat X}^\ccI_\bullet)^\wedge$, we then have that $\wh{\Aut}({\widehat X}^\ccI_\bullet)=\Aut({\widehat X}^\ccI_\bullet)$.
\end{proof}

\subsection{Congruence completions of curve complexes} 
In order to define, the profinite completions of the (abstract) simplicial complexes $C(S)$ and $C_P(S)$ associated to some directed poset
of finite index subgroups of $\G(S)$, we need first to consider the associated simplicial sets $C(S)_\bt$ and $C_P(S)_\bt$ and then take 
pro-completions in the above sense. In this paper, we will be essentially interested in the pro-completions of $C(S)$ and $C_P(S)$ 
associated to the full profinite and congruence topologies on $\G(S)$. We will then denote by $\widehat{C}(S)_\bt$, $\widehat{C}_P(S)_\bt$ 
the profinite  completions and by $\check{C}(S)_\bt$, $\check{C}_P(S)_\bt$ the procongruence completions, respectively. 
By Proposition~3.3 in \cite{[B3]}, the natural maps from the simplicial sets $C(S)_\bt$ and 
$C_P(S)_\bt$ to all their profinite completions above are injective.

\subsection{Profinite simple closed curves}\label{section:pscc}
A more concrete avatar of the curve complex $C(S)$ can be constructed in the procongruence case (cf.\ Section~4 of \cite{[B3]}).
This is done introducing the profinite set $\hL(S)$  of \emph{profinite simple closed curves on $S$}. This set comes with a natural
injective map $\cL(S)\hookrightarrow\hL(S)$, where $\cL(S)$ is the usual set of isotopy classes of simple closed curves on $S$. 
There are two equivalent definitions. 

For the first one, we proceed as follows. For a set $X$, let $\cP(X)$ denote the set of {\it unordered} 
pairs of elements of $X$ and, for a group $G$, let $G/\!\!\sim$ denote the set of conjugacy classes in $G$. 
Then, for $x\in\hP$, let us denote by $x^\pm$ the pair $(x,x^{-1})$ in $\cP(\hP)$ and by
$[x^\pm]$ the equivalence class of $x^\pm$ in $\cP(\hP/\!\!\sim)$. Note that the latter has a natural structure 
of profinite set. There is a natural map $\iota\co \cL(S)\to \cP(\hP/\!\!\sim)$, defined as follows.
For $\g\in\cL(S)$, let $\vec{\g}\in\Pi\subset\hP$ be an element with free isotopy class $\g$
and define $\iota(\g)=[\vec\g^\pm]$, which is plainly independent of the choice of the representative 
$\vec\g$ of $\g$. Since the group $\Pi$ is conjugacy separable, the map $\iota$ is an embedding and
we define the set $\hat\cL=\hat\cL(S)$ to be the closure of the image $\iota(\cL(S))$ inside $\cP(\hP/\!\!\sim)$.

For this definition to be entirely meaningful, it would be desirable that, intersecting the profinite set $\hL(S)$ with $\cP(\Pi/\!\!\sim)$, we get back the
set of simple closed curves $\cL(S)$, that is to say the latter set is closed for the induced topology on $\cP(\Pi/\!\!\sim)$. This is indeed the case:

\begin{proposition}\label{scc:closed}With the above notation, let us identify the sets $\cL(S)$ and $\cP(\Pi/\!\!\sim)$ with their images in $\cP(\hP/\!\!\sim)$.
Then we have: $\hL(S)\cap\cP(\Pi/\!\!\sim)=\cL(S)$.
\end{proposition}

\begin{proof}In the paper \cite{[B3]}, this was eventually proved in Remark~5.14 but we are now able to provide a more straightforward proof
thanks to Proposition~4.1 and Theorem~4.3 in \cite{[BZ1]} and their generalization to the open surface case in \cite{congtop} (cf.\ Theorem~7.1 and 
Proposition~7.5 ibid.). Those results imply in fact the stronger statement that, if we denote by $\Aut^\ast(\hP)$ the (closed) subgroup of elements of 
$\Aut(\hP)$ which preserve the conjugacy classes of the procyclic subgroups of $\hP$ generated by small loops around the punctures of $S$, then, 
for $\g\in\Pi$ simple, the intersection of the $\Aut^\ast(\hP)$-orbit of $\g$ with $\Pi$ also consists of simple elements.
\end{proof}

For the second definition, let $\cG(\hP)/\!\!\sim$ be the (profinite) set of conjugacy classes of closed subgroups of $\hP$
and define a map $\iota'\co\cL(S)\to\cG(\hP)/\!\!\sim$ sending $\g\in\cL(S)$ to the conjugacy class of the procyclic subgroup of $\hP$ generated by 
$\vec\g$ (same notation as above). We then define $\hL'(S)$ to be the closure of the image of $\iota'$.
There is a natural continuous map $\psi\co\hL(S)\to\hL'(S)$ and the LERF property of surface groups implies that this map is a bijection of profinite sets.
Indeed, since finitely generated subgroups of $\Pi$ are closed in the profinite topology, we have, in particular, that, for all elements $x\in\Pi$, 
there holds $x^\ZZ\cap\Pi=x^\Z$, where we denote by $x^\Z$ and $x^\ZZ$ the cyclic and the \emph{closed procyclic} subgroups generated by $x$ in $\Pi$ 
and $\hP$, respectively. It is then easy to construct an inverse to $\psi$. Thus, in what follows, we will identify $\hL'(S)$ with $\hL(S)$ and 
denote both of them simply by $\hL(S)$.

\subsection{The procongruence curve complex}\label{section:isomorphism}
For every $k\geq 0$, there is a natural embedding $C(S)_k\hookrightarrow\cP_{k+1}(\hL(S))$ from the set of $k$-simplices of the
curve complex into the profinite set $\cP_{k+1}(\hL(S))$ of unordered subsets of $k+1$ elements of $\hL(S)$. 
The \emph{procongruence curve complex} $\kC_\cL(S)$ is then defined as the \emph{abstract simplicial profinite complex} (cf.\ (ii) of Definition~3.2 in \cite{[B3]})
whose set of $k$-simplices, for $k\geq 0$, is the closure of the image of $C(S)_k$ in $\cP_{k+1}(\hL(S))$. \emph{Note that, as an abstract simplicial complex,
$\kC_\cL(S)$ is a flag complex because $C(S)$ is.} This remark will play an important role later on in the paper.

So far we have two congruence completions of the curve complex $C(S)$: the simplicial profinite set $\kC(S)_\bt$ obtained as the $\kG(S)$-completion 
of the simplicial set $C(S)_\bt$ and the abstract simplicial profinite complex $\kC_\cL(S)$ defined above. Each construction has its own 
advantage. As we just remarked, $\kC_\cL(S)$ is a flag simplicial complex. On the other hand, the stabilizers for the action of $\kG(S)$ on $\kC(S)_\bt$ 
are obtained taking the closure in $\kG(S)$ of the stabilizers for the action of $\G(S)$ on $C(S)_\bt$ (cf.\ Proposition~6.5 in \cite{[B1]}). Let $\kC_\cL(S)_\bt$
be the simplicial profinite set associated (formally) to the abstract simplicial profinite complex $\kC_\cL(S)$. This is endowed with a continuous
$\kG(S)$-action and there is a natural $\G(S)$-equivariant map of simplicial sets $C(S)_\bt\to\kC(S)_\bt$, with dense image. 
By the universal property of $\kG(S)$-completions, there is then a natural surjective map of simplicial profinite sets $\kC(S)_\bt\rightarrow \kC_\cL(S)_\bt$.
According to Remark~7.11 of \cite{[B4]}, also $\kC(S)_\bt$ is the simplicial profinite set associated to an abstract simplicial profinite 
complex $\kC(S)$. Therefore, there is also an induced surjective map of abstract simplicial profinite complexes $r\co\kC(S)\rightarrow \kC_\cL(S)$.

The main technical result of (\cite{[B3]},  see \ Thm. 4.2) is:
  
\begin{theorem}\label{isomorphism}The natural map $r\co\kC(S)\to\kC_\cL(S)$ is a $\kG(S)$-equivariant isomorphism 
of abstract simplicial profinite complexes.
\end{theorem}

\begin{proof}[Sketch of the proof.] We will present here the main ideas of the proof. Since we are dealing with abstract
simplicial complexes, it is enough to show that the map $r$ is injective at the level of $0$-simplices. The first step
is then to show that $r$ induces a bijective correspondence between the $\kG(S)$-orbits in $\kC(S)_0$ and $\kC_\cL(S)_0$:

\begin{lemma}\label{orbits}The orbits for the action of $\kG(S)$ on $\kC(S)_0$ and $\kC_\cL(S)_0$ are parameterized by the topological types 
of nonperipheral simple closed curves on $S$, namely by the orbits of the action of $\G(S)$ on $\cL(S)_0$.
\end{lemma}

\begin{proof}The orbits of $\kG(S)$ in $\kC(S)_0$ parameterize the irreducible components of the DM boundary of $\ol{\cM}(S)$ which, in their turn,
are parameterized by the topological types of nonperipheral simple closed curves on $S$. To see that the same is true for the orbits of $\kG(S)$
in $\kC_\cL(S)_0$, we need to consider the monodromy representation $\hat\rho_S\co\hG(S)\to\Out(\hP)$ associated to the universal family of curves
$\cC(S)\to\cM(S)$. First observe that, by definition of the procongruence mapping class group, the representation $\hat\rho_S$ factors through a 
representation $\check\rho_S\co\kG(S)\to\Out(\hP)$ and it is this representation which induces the action of $\kG(S)$ on the set of profinite simple
closed curves $\cL(S)$. The lemma will follow if we show that the $\kG(S)$-orbit of an element 
$\g$ of $\cL(S)\subset\hL(S)$ intersects $\cL(S)$ precisely
in the $\G(S)$-orbit of $\gamma$. 
But this follows from the geometric nature of this action (for more details, see the first paragraphs of the proof of Theorem~4.2 in \cite{[B3]}).
\end{proof}

Thanks to the above lemma and its proof we can define the topological type also for profinite simple closed curves:

\begin{definition}\label{toptype}The topological type of a simplex $\s\in\kC(S)$ is the topological type of a simplex $\s'$ in the intersection
$\kG(S)\cdot\s\cap C(S)$, namely the topological type of the surface $S\ssm\s'$. We then define $\hL(S)_0$ to be the closed subset of 
$\hL(S)$ corresponding to profinite nonperipheral simple closed curves.
\end{definition}

The second and last step in the proof of Theorem~\ref{isomorphism} then consists of showing that the $\kG(S)$-stabilizer of a $0$-simplex
in $\kC(S)$ is naturally isomorphic to the $\kG(S)$-stabilizer of a $0$-simplex in $\kC_\cL(S)$:

\begin{lemma}Let us identify a $0$-simplex $\s$ in the curve complex $C(S)$ with its image in $\kC(S)$. Then we have
$\ol{\G}(S)_\s=\kG(S)_\s=\kG(S)_{r(\s)}$, where $\ol{\G}(S)_\s$ denotes the closure in the congruence completion $\kG(S)$ 
of the stabilizer $\G(S)_\s$ for the action of $\G(S)$ on $C(S)$.
\end{lemma}

\begin{proof}[Sketch of the proof.] The first identity is relatively simple to prove and is the content of Proposition~6.5 in \cite{[B1]}. 
The second identity is much more difficult and represents the bulk of the proof of Theorem~4.2 in \cite{[B3]}. Here, we explain the main ideas.

The basic result on which the proof rests is Corollary~7.8 in \cite{[B4]} which allows us to replace the simplex $\s\in\kC(S)_0$ with some 
group-theoretic combinatorial data (cf.\ Definition~7.5 in \cite{[B4]}). This reduces the proof of the second identity to showing that the stabilizer 
$\kG(S)_{r(\s)}$ preserves the equivalence class of this data. This happens essentially because geometric subgroups of $\hP$ are self-normalizing
(cf.\ Lemma~4.3 in \cite{[B3]}, which states that procyclic subgroups of $\hP$ generated by simple elements are actually \emph{malnormal}, 
and the more general Proposition~3.6 and Theorem~3.7 in \cite{[BZ2]}).
\end{proof}
\end{proof}

\begin{corollary}\label{stabilizers}Let us identify a simplex $\s$ in the curve complex $C(S)$ with its image in $\kC_\cL(S)$. 
Then we have $\kG(S)_\s=\ol{\G}(S)_\s$, where $\ol{\G}(S)_\s$ denotes the closure in the congruence completion $\kG(S)$ 
of the stabilizer $\G(S)_\s$ for the action of $\G(S)$ on $C(S)$.
\end{corollary}

For a simplex $\s\in C(S)$, let $\mathrm{I}_\s$ be the abelian subgroup of $\G(S)$ generated by the Dehn twists 
about the curves contained in $\s$. Let then $\hat{\mathrm{I}}_\s$ be the profinite completion of $\mathrm{I}_\s$. 
If $\s=\{\g_0,\ldots,\g_h\}\in C(S)$, then we denote by $\vec\s$ the set of oriented simple curves
$\{\vec\g_0,\ldots,\vec\g_h\}$, where  $\vec{\g}\in\Pi\subset\hP$ is an element with free isotopy class $\g\in\cL(S)$. 
Let also $\vec\s^{\pm}=\{\vec\g_0^\pm,\ldots,\vec\g_h^\pm\}$  be the set of oriented simple closed curves in $\s$ with both orientations and $\Sigma_{\vec\s^{\pm}}$ the group of permutations on the set $\vec\s^{\pm}$. 
The stabilizers for the action of $\kG(S)$ on the procongruence curve complex  $\kC_\cL(S)$ can be further characterized as follows:

\begin{theorem}\label{stabmulticurves}Let $\s=\{\g_0,\ldots,\g_h\}\in\kC_\cL(S)$ be a simplex in the image of $C(S)$ and 
$S\ssm\s=S_1\coprod\ldots\coprod S_k$.
\begin{enumerate}
\item The stabilizer $\kPG(S)_\s$ of $\s$ for the action of $\kPG(S)$ on $\kC_\cL(S)$ is described by the two exact sequences:
\[\begin{array}{c}
1\ra\kPG(S)_{\vec{\s}}\to\kPG(S)_\s\to\Sigma_{\vec\s^{\pm}},\\
\\
1\ra\hat{\mathrm{I}}_\s\to\kPG(S)_{\vec{\s}}\to\kPG(S_1)\times\ldots\times\kPG(S_k)\to 1.
\end{array}\]
\item For $i=1,\ldots,k$, let $B_i$ be the set of punctures of $S_i$ bounded by a simple closed curve of $S$ and let $\kG(S_i)_{B_i}$ be the 
pointwise stabilizer of the set $B_i$ (for the natural action of $\kG(S_i)$ on the set of punctures of $S_i$). 
Then, the stabilizer $\kG(S)_\s$ of $\s$ for the action of $\kG(S)$ on $\kC_\cL(S)$ is described by the two exact sequences:
\[\begin{array}{c}
1\ra\kG(S)_{\vec{\s}}\to\kG(S)_\s\to\Sigma_{\vec\s^{\pm}},\\
\\
1\ra\hat{\mathrm{I}}_\s\to\kG(S)_{\vec{\s}}\to\kG(S_1)_{B_1}\times\ldots\times\kG(S_k)_{B_k}\to 1.
\end{array}\]
\end{enumerate}
\end{theorem}

\begin{proof}(i): This is just Theorem~4.5 in \cite{[B3]}. 
\medskip

\noindent
(ii): The natural homomorphism $\G(S)_{\vec{\s}}\to\kG(S)_{\vec{\s}}$ induces a homomorphism:
\begin{equation}\label{quotinduced}
\G(S)_{\vec{\s}}/\mathrm{I}_\s\cong\G(S_1)_{B_1}\times\ldots\times\G(S_k)_{B_k}\to\kG(S)_{\vec{\s}}/\hat{\mathrm{I}}_\s.
\end{equation}
The first item of the theorem implies, in particular, that the map $\G(S_i)_{B_i}\to\kG(S)_{\vec{\s}}/\hat{\mathrm{I}}_\s$ extends
to a homomorphism $\kG(S_i)_{B_i}\to\kG(S)_{\vec{\s}}/\hat{\mathrm{I}}_\s$, for $i=1,\ldots,k$. 
Therefore, the homomorphism~(\ref{quotinduced}) induces a homomorphism: 
\[\kG(S_1)_{B_1}\times\ldots\times\kG(S_k)_{B_k}\to\kG(S)_{\vec{\s}}/\hat{\mathrm{I}}_\s.\] 
In order to prove the second item of the theorem, it is enough to show that this is an isomorphism.

Let $Q_i$ be the set of punctures of $S_i$ in the complement of $B_i$ and $\Sigma_{Q_i}$ the symmetric group on the set $Q_i$, for $i=1,\ldots,k$.
There is a commutative diagram with exact rows:
\[\begin{array}{ccccccccc}
1&\to&\kPG(S)_{\vec{\s}}/\mathrm{I}_\s&\to&\kG(S_1)_{B_1}\times\ldots\times\kG(S_k)_{B_k}&\to&
\Sigma_{Q_1}\times\ldots\times\Sigma_{Q_k}\,&\to&1\\
&&\parallel&&\downarrow&&\parallel&&\\
1&\to&\kPG(S)_{\vec{\s}}/\hat{\mathrm{I}}_\s&\to&\kG(S)_{\vec{\s}}/\hat{\mathrm{I}}_\s&\to&\Sigma_{Q_1}\times\ldots\times\Sigma_{Q_k}.&&
\end{array}\]
The commutativity of the diagram implies that also the rightmost bottom horizontal map is surjective. Hence, the middle vertical map is an isomorphism.
\end{proof}

\subsection{Profinite Dehn twists in $\kG(S)$}\label{section:prodehntwists}
The upshot of Theorem~\ref{isomorphism} is that the set of nonperipheral profinite simple closed curves $\hL(S)_0$ can be used to parameterize the
\emph{profinite Dehn twists} of $\kG(S)$. This set is simply defined to be the closure $\ol{\cD}(S)$ of the image of the set of Dehn twists $\cD(S)$ of 
$\G(S)$ via the natural homomorphism $\G(S)\to\kG(S)$. However, it is far from obvious \emph{a priori} that this is a meaningful definition. 
For instance, it is not clear whether $\ol{\cD}(S)\cap\G(S)=\cD(S)$ (cf.\ Proposition~\ref{scc:closed}).

More in detail, an orientation of $S$ defines a natural (injective) map $\tau\co C(S)_0\to\G(S)$ which assign to the isotopy class of a 
nonperipheral closed simple closed curve $\gamma$ on $S$ the {\em left} Dehn twist $\tau_{\gamma}$ about that curve. By
composition, we get a map, which we denote in the same way, $\tau\co C(S)_0\to\kG(S)$. This map is clearly $\G(S)$-equivariant. Therefore, 
by the universal property of the $\kG(S)$-completion, we get an induced continuous $\kG(S)$-equivariant map $\check\tau\co\kC(S)_0\to\kG(S)$.
In Theorem~\ref{isomorphism}, we showed that the profinite set $\kC(S)_0$ can be identified with the profinite set of profinite nonperipheral  
simple closed curves $\hL(S)_0$. Thus, finally, we get a $\kG(S)$-equivariant map: 
\[\check\tau\co\hL(S)_0\to\kG(S).\]

The advantage of having defined this map from the profinite set $\hL(S)_0$ is that we now know that a profinite Dehn twist is determined by some
element in the profinite surface group $\hP$. The (profinite) combinatorial theory of this group is much more approachable than that of $\kG(S)$ and,
in the paper \cite{[BZ2]}, this was used to show that the parameterization of profinite Dehn twists provided by the map $\check\tau$ has almost all
the nice properties enjoyed, in the topological case, by the map $\tau$. For instance, $\check\tau$ is injective. A much stronger result, in analogy
with the topological case, actually holds. 

Let a \emph{profinite multicurve on the surface $S$} be a simplex of $\kC_\cL(S)$. 
For the sake of simplicity we will denote by $\tau_{\g}$ the profinite Dehn twist $\check\tau_{\g}$, 
also when $\g$ is a profinite simple closed curve. We then have  (cf.\ Theorem~6.6 in  \cite{[B3]} and Theorem~4.2 in \cite{congtop}):

\begin{theorem}\label{multitwists}Let $\s=\{\g_1,\ldots,\g_s\}$ and
$\s'=\{\delta_1,\ldots,\delta_t\}$ be two profinite multicurves on the surface $S$.
Suppose that, for the multi-indices $(h_1,\ldots,h_s)\in m_\s\cdot(\Z\ssm\{0\})^s$ and $(k_1,\ldots,k_t)\in m_{\s'}\cdot(\Z\ssm\{0\})^t$, 
where $m_\s,m_{\s'}\in(\ZZ)^\ast$, there is an identity:
\[\tau_{\g_1}^{h_1}\cdots\tau_{\g_s}^{h_s}=
\tau_{\delta_1}^{k_1}\cdots\tau_{\delta_t}^{k_t}\in\kG(S).\]
Then, we have that:
\begin{enumerate}
\item $t=s$;
\item there is a permutation $\phi\in\Sigma_s$ such that $\delta_i=\g_{\phi(i)}$ and $k_i=h_{\phi(i)}$,
for $i=1,\ldots,s$.
\end{enumerate}
\end{theorem}

A crucial observation, at this point, is that in the profinite, exactly like in the topological case, for any 
profinite multicurve $\s=\{\g_1,\ldots,\g_s\}$ on $S$ and all $f\in\kG(S)$, we have the identity:
\[f\cdot(\tau_{\g_1}^{h_1}\cdots\tau_{\g_s}^{h_s})\cdot f^{-1}=\tau_{f(\g_1)}^{h_1}\cdots\tau_{f(\g_s)}^{h_s}.\]

A \emph{weight function} $w\co\cL(S)\to\Z\ssm\{0\}$ is defined to be a $\kG(S)$-equivariant map where $\kG(S)$ 
acts trivially on $\Z\ssm\{0\}$. An immediate consequence of Theorem~\ref{multitwists} and the above remark is then that, for all $u\in\ZZ^\ast$, 
there is a natural injective and continuous $\kG(S)$-equivariant map
\[\Delta^{w,u}_k\co\kC_\cL(S)_k\hookrightarrow\kG(S),\]
defined by the assignment $\s=\{\g_0,\ldots,\g_k\}\mapsto\tau_{\g_0}^{u\cdot w(\g_0)}\cdots\tau_{\g_k}^{u\cdot w(\g_k)}$. 
This can also be slightly improved in the following way. Let $\cG(\kG(S))$ be the profinite set of all closed subgroups of $\kG(S)$. 
Then, sending the profinite multicurve 
$\s\in\kC_\cL(S)_k$ to the procyclic subgroup generated by $\Delta^w_k(\s)$ defines a natural injective and continuous $\kG(S)$-equivariant map
\[M^w_k\co\kC_\cL(S)_k\hookrightarrow\cG(\kG(S)).\]
An alternative version of this map is defined by sending the profinite multicurve $\s\in\kC_\cL(S)_k$ to the closed subgroup of $\kG(S)$ 
generated by the elements $\tau_{\g_0}^{w(\g_0)},\ldots,\tau_{\g_k}^{w(\g_k)}$. We then get the natural injective and continuous $\kG(S)$-equivariant map
\[{\mathcal I}^w_k\co\kC_\cL(S)_k\hookrightarrow\cG(\kG(S)).\]
Therefore, Theorem~\ref{multitwists} and Proposition~\ref{scc:closed} imply:

\begin{corollary}\label{centralizers multitwists}\leavevmode
\begin{enumerate}
\item For $\s=\{\g_0,\ldots,\g_k\}$ a profinite multicurve on $S$ and a weight function $w\co\cL(S)\to \Z\ssm\{0\}$ as above, we have:
\[Z_{\kG(S)}(\tau_{\g_0}^{w(\g_0)}\cdots\tau_{\g_k}^{w(\g_k)})
= N_{\kG(S)}(\langle \tau_{\g_0}^{w(\g_0)}\cdots\tau_{\g_k}^{w(\g_k)}\rangle)
=N_{\kG(S)}(\langle \tau_{\g_0}^{w(\g_0)},\ldots,\tau_{\g_k}^{w(\g_k)}\rangle)=\kG(S)_\s,\]
where $\kG(S)_\s$ is the stabilizer of $\s$ described in (ii) of Theorem~\ref{stabmulticurves}. In particular, $\kG(S)_\s$,
for all $\s\in\kC_\cL(S)$, is a self-normalizing subgroup of $\kG(S)$.

\item Let $\cD^\infty(S)$ and $\check{\cD}^\infty(S)$ be the union of all cyclic subgroups  generated by the Dehn twists in $\G(S)$ and the 
union of all procyclic subgroup generated by the profinite Dehn twists in $\kG(S)$, respectively. 
Then, we have $\check{\cD}^\infty(S)\cap\G(S)=\cD^\infty(S)$. 
In other terms, the set of all powers of Dehn twists in $\G(S)$ is closed for the congruence topology.
\end{enumerate}
\end{corollary}

\begin{remark}\label{grouptheoretic}
From (i) Corollary~\ref{centralizers multitwists}, we deduce a further and, for us, the most important group-theoretic realization of the complex 
of profinite curves $\kC_\cL(S)$. Let $\kG^\l$ be an open subgroup of $\kG(S)$. By Theorem~\ref{multitwists} and Theorem~\ref{stabmulticurves},
we have that $\kG^\l_\s=\kG^\l_{\s'}$, for $\s,\s'\in\kC_\cL(S)$, if and only $\s=\s'$ and, by (i) Corollary~\ref{centralizers multitwists}, we have that
$N_{\kG^\l}(\kG^\l_\s)=\kG^\l_\s$. Therefore, the assignment $\s\mapsto\kG_\s^\l$, for $\s\in\kC_\cL(S)$, defines, for all $k\geq 0$, an injective and continuous $\kG^\l$-equivariant map
\[\cG^\l_k\co\kC_\cL(S)_k\hookrightarrow\cG(\kG(S)).\]
The existence of such an injective continuous $\kG^\l$-equivariant map $\cG^\l_k$ is a key ingredient in establishing 
the anabelian result of Theorem \ref{mainanabelian2}. 
\end{remark}

\subsection{Profinite braid twists in $\kG(S)$}\label{section:probraidtwists}In Section~\ref{DehnBraids}, we associated to a simple closed curve
$\g\in\cL^b(S)$ the braid twist $b_\g$. We can then define $\hL^b(S)$ to be the closure of $\cL^b(S)$ in the profinite set $\hL(S)$ and a 
\emph{profinite braid twist} in $\kG(S)$ to be an element in the closure of the set of the images of braid twists in $\kG(S)$.

Theorem~\ref{isomorphism} implies that there is a natural continuous map $\check{b}\co\hL(S)_0\to\kG(S)$ whose image is the set of profinite braid twists.
The fact that $\check\tau$ is injective implies, in particular, that $\check{b}$ is injective. It is not difficult to show that also a version of
Theorem~\ref{multitwists} holds in this case. We will not pursue this here since we will not be using it in the sequel.

\subsection{The center of $\kG(S)$}\label{section:center}Thanks to Corollary~\ref{centralizers multitwists}, we can also describe
the center of the full procongruence mapping class group. More generally, we have:

\begin{theorem}\label{slim}Let $U$ be an open subgroup of $\kG(S)$. 
\begin{enumerate}
\item For $S\neq S_{0,4}, S_{1,1}, S_{1,2}$ and $S_2$, the centralizer of $U$ in $\kG(S)$ is trivial.
\item For $S= S_{1,1}, S_{1,2}$ and $S_2$, the centralizer of $U$ in $\kG(S)$ is generated by the hyperelliptic involution $\iota$.
\item For $S= S_{0,4}$, we have $Z(\hG(S))=\{1\}$ and  if $U\subseteq\kPG(S)$, then $Z_{\hG(S)}(U)=K_4$, where $K_4\cong\{\pm 1\}\times\{\pm 1\}$
is the Klein subgroup of $\G_{0,[4]}$. 
\end{enumerate}
\end{theorem}

\begin{proof} The case of genus $g(S)\geq 1$ follows by induction on the number  $n(S)$ of punctures of $S$.
For $n(S)=0,1$, this is just Corollary~6.2 in \cite{[B3]}. Assume now  that the claim holds for $g(S)\geq 1$ and $n(S)=k$
and let us prove it for $n(S)=k+1$.

Let $\g$ be a separating simple closed curve on $S$ bounding a $2$-punctured open disc $D$. 
The open subgroup $U$ contains some power $\tau_\g^k$, for $k\in\N^+$ and, by Corollary~\ref{centralizers multitwists}, we have 
$Z_{\kG(S)}(\tau_\g^k)=\kG(S)_\g$. Therefore, we have $Z_{\kG(S)}(U)\subseteq Z_{\kG(S)_\g}(U\cap\kG(S)_\g)$.

Note that $\kG(S)_\g=\kG(S)_{\vec\g}$. Let $S\ssm\g=S'\coprod D$. By item (ii) of Theorem~\ref{stabmulticurves}, 
there is then a short exact sequence:
\[1\to b_\g^\ZZ\to\kG(S)_\g\to\kG(S')_P\to 1,\]
where $\kG(S')_P$ is the stabilizer of the puncture $P$ on $S'$ corresponding to $\g$.
Since $b_\g^\ZZ\cap Z_{\kG(S)}(U)=\{1\}$, we can identify $Z_{\kG(S)}(U)$ with a subgroup of
$Z_{\kG(S')_P}(U')$, where $U'$ is the image of $U\cap\kG(S)_\g$ in $\kG(S')_P$. This is then a finite index subgroup of $\kG(S')$, 
where $n(S')=n(S)-1$. Therefore, by the induction hypothesis, we have that $Z_{\kG(S')_P}(U')$ is either trivial or generated by the 
hyperelliptic involution (when $S'=S_{1,1}$). This concludes the proof of the case of genus $g(S)\geq 1$ of Theorem~\ref{slim}.

We treat the genus $0$ case of the theorem starting with the case $S=S_{0,4}$. By Proposition~2.7 in \cite{[FM]}, 
there is a natural isomorphism $\G_{0,[4]}\cong\PSL_2(\Z)\ltimes K_4$, where $\PSL_2(\Z)$ acts on the Klein group
$K_4$ through its quotient $\PSL_2(\Z/2\Z)\cong\Sigma_3\cong\Aut(K_4)$ and the normal subgroup $\G_{0,4}$ of $\G_{0,[4]}$ identifies with the kernel of the 
natural epimorphism $\PSL_2(\Z)\to\PSL_2(\Z/2\Z)$. This implies that there is a natural isomorphism
$\hG_{0,[4]}\cong\wh{\PSL}_2(\Z)\ltimes K_4$ with similar properties, from which item (iii) of Theorem~\ref{slim} follows. 

We treat next the case  $S=S_{0,5}$. Let $\g$ be a separating simple closed curve on $S$ bounding a $2$-punctured open disc $D$ and let
$S\ssm\g=S'\coprod D$. By item (ii) of Theorem~\ref{stabmulticurves}, there is a short exact sequence:
\[1\to b_\g^\ZZ\to\hG(S)_\g\to\hG(S')_P\to 1,\]
where $\hG(S')_P$ is the stabilizer of the puncture $P$ on $S'$ bounded by $\g$. This group fits in the short exact sequence:
\[1\to\hPG(S')\to\hG(S')_P\to\Sigma_3\to 1.\]

The isomorphism $\hG_{0,[4]}\cong\wh{\PSL}_2(\Z)\ltimes K_4$ then implies that there is an isomorphism $\hG(S')_P\cong\wh{\PSL}_2(\Z)$.
 
In order to show that, for an open subgroup $U$ of $\hG(S)$, we have $Z_{\hG(S)}(U)=\{1\}$, as in the proof of the genus $\geq 1$ case, 
it is then enough to show that $Z_{\hG(S')_P}(U')=\{1\}$, where $U'$ is the finite index image of $U\cap\hG(S)_\g$ in $\hG(S')_P$.
But this follows from the isomorphism $\hG(S')_P\cong\wh{\PSL}_2(\Z)$ and well known properties of the latter group.

For $S=S_{0,n}$, with $n\geq 5$, we proceed by induction on the number of punctures, 
arguing exactly as in the proof of the genus $\geq 1$ case.
\end{proof}

\section{Automorphisms of the procongruence and pants curve complexes}\label{rigidityprocurvecomplex}
In this section, we will study more in detail the procongruence curve complex $\kC(S)$. By Theorem~\ref{isomorphism}, this complex can and will
be identified with the complex of profinite curves $\kC_\cL(S)$. Therefore, from now on, both complexes will be simply denoted by $\kC(S)$.

\subsection{Distinguishing procongruence curve complexes}
\begin{theorem}\label{nonisomorphism}
Let $S=S_{g,n}$ and $S'=S_{g',n'}$ be two connected hyperbolic surfaces of different types $(g,n)$
and $(g',n')$. Then the procongruence complexes $\kC(S)$ and $\kC(S')$ are not isomorphic, 
except for the exceptional cases $\kC(S_{1,1})\cong \kC(S_{0,4})$, 
$\kC(S_{1,2})\cong \kC(S_{0,5})$ and $\kC(S_{2,0})\cong \kC(S_{0,6})$.
\end{theorem}

From the topological case, we know that there are exceptional isomorphisms $C(S_{1,1})\cong C(S_{0,4})$, $C(S_{1,2})\cong C(S_{0,5})$ and 
$C(S_2)\cong C(S_{0,6})$. We also know that, correspondingly, $\G_{1,1}$ and $\G_{0,4}$, $\G_{1,2}$ and $\G_{0,5}$ and 
$\G_2$ and $\G_{0,6}$ are pairs of commensurable group in a way which is compatible with their respective actions on curve complexes
and the above isomorphisms between them. Since the $\hG(S)$-completion does not change if we 
replace $\G(S)$ with any finite index subgroup, by the congruence subgroup property in genus $\leq 2$, these exceptional isomorphisms induce 
isomorphisms also on their respective congruence completions.

In order to prove Theorem \ref{nonisomorphism}, we then need first to drastically reduce the number of 
possible  isomorphisms between two complexes $\kC(S_{g,n})$ and $\kC(S_{g',n'})$ for different types 
$(g,n)$ and $(g',n')$. To this purpose we introduce two invariants. The first one is  the modular dimension of $S_{g,n}$, which is also
the dimension of the curve complex, $d_{g,n}=\dim(C(S_{g,n}))=3g-3+n$. We will then introduce another invariant, or rather two closely 
related ones, which will require some preliminary lemmas. 

The first lemma-definition introduces an useful invariant in the topological case, which will subsequently be shown 
to survive completion. For an abstract simplicial complex $X$, let $X^-$ be the $1$-dimensional abstract simplicial complex
with the same vertex set as $X$ and an edge joining two vertices of $X^-$ if and only if these are not connected by an edge in $X$.
For a simplex $\s\in X$, we then define the \emph{dual link} of $\s$ as the simplicial complex $\Link_X(\sigma)^-$ and denote it simply by $L^-_X(\sigma)$. 
We say that a profinite simple closed curve $\a\in\hL(S)$) is of {\em boundary type} if a simple closed curve 
$\a'\in\kG(S)\cdot\a\cap\cL(S)$ bounds a subsurface of type $(0,3)$ in $S$. We have the following lemma whose topological version is immediate:

\begin{lemma}\label{connectedduallink}
Let $\a\in\hL(S)_0$ be a profinite simple closed curve on $S$. 
The dual link  $L^-_{\kC(S)}(\a)$ is nonempty when $S$ is different from $S_{1,1}$ ($\a$ nonseparating) 
and  $S_{0,4}$ ($\a$ separating and of boundary type). In this case, the dual link  $L^-_{\kC(S)}(\a)$ 
is connected if and only if $\a$ is either nonseparating or of boundary type.
\end{lemma} 

\begin{proof}Since the $\kG(S)$-orbit of $\a$ contains an element in $\cL(S)$, we can assume that $\a\in\cL(S)_0\subset \hL(S)_0$. 
By Remark 4.7 in \cite{[B3]}, the link of $\a$ in $\kC(S)$ is the procongruence curve complex $\kC(S\ssm\a)$. 
If $S\ssm\a$ has two connected components $S_1$ and $S_2$, we have that
\[\kC(S\ssm\a)=\kC(S_1)\star\kC(S_2).\]
Thus, if $d(S_i)>0$, for $i=1,2$, then $L^-_{\kC(S)}(\a)=\kC(S_1)^-\coprod\kC(S_2)^-$ has both components nonempty and so is disconnected.
Instead, if either $S\ssm\a$ is connected or if one of the connected components of $S\ssm\a$ has modular dimension zero, then 
$L^-_{\kC(S)}(\a)$ is connected, since it contains $L^-_{C(S)}(\a)$ as a dense path-connected subspace.
\end{proof}

For $S=S_{g,n}$ hyperbolic, let $\mathrm{Sep}(S)$ (resp.\ $\mathrm{NSep}(S)$) denote the \emph{maximal number of pairwise disjoint separating curves 
not of boundary type} (resp.\ \emph{maximal number of disjoint curves which are either nonseparating or of boundary type})
which a simplex $\s\in\kC(S)$ can contain. By Lemma~\ref{orbits}, this is of course purely a topological invariant of the surface $S$.
It is easy to compute these numbers explicitly. Thus, we leave the proof of the following lemma to the reader as an easy exercise:

\begin{lemma}\label{calculsep}

We have
\[ \mathrm{Sep}(S_{g,n})=
\left\{\begin{array}{ll}
\max(n-5,0), & {\rm if}\; g=0;\\ 
\max(n-2,0), & {\rm if }\; g=1;\\
2g+n-3, & {\rm if }\; g\geq 2.\\
\end{array}
\right.
\]
\[ \mathrm{NSep}(S_{g,n})=
\left\{\begin{array}{ll}
\left[\frac n 2 \right], & {\rm if}\; g=0;\\ 
3g+n-3, & {\rm if }\; g\geq 1.\\
\end{array}
\right.
\]

\end{lemma}

Let now $\hL^\mathrm{sep}(S)$ (resp.\ $\hL^\mathrm{nsep}(S)$) be the (profinite) subset of $\hL(S)$ consisting of profinite simple closed curves which
are separating but not of boundary type (resp.\ either nonseparating or of boundary type). Then, we have:

\begin{lemma}\label{sepinvariant}
An isomorphism $\phi\co\kC(S)\stackrel{\sim}{\to}\kC(S')$ induces bijections $\hL^\mathrm{sep}(S)\stackrel{\sim}{\to}\hL^\mathrm{sep}(S')$ 
and $\hL^\mathrm{nsep}(S)\stackrel{\sim}{\to}\hL^\mathrm{nsep}(S')$. 
\end{lemma}

\begin{proof}The isomorphism $\phi$ induces an isomorphism $\Link_{\kC(S)}(\a)\cong\Link_{\kC(S')}(\phi(\a))$ and then also 
$L^-_{\kC(S)}(\a)\cong L^-_{\kC(S')}(\phi(\a))$, for all $\a\in\hL(S)_0$. The conclusion then immediately follows from Lemma~\ref{connectedduallink}.
\end{proof}

\begin{proof}[Proof of Theorem~\ref{nonisomorphism}]
Assume that there exists an isomorphism $\phi\co \kC(S)\to \kC(S')$, where $S=S_{g,n}$ and $S'=S_{g',n'}$. 
Then, by Lemma~\ref{sepinvariant}, we have the equalities:
\[\dim(S)=\dim(S'), \quad \mathrm{Sep}(S)=\mathrm{Sep}(S'), \quad \mathrm{NSep}(S)=\mathrm{NSep}(S').\]

Straightforward bookkeeping using Lemma~\ref{calculsep} shows that, if the two types are different, 
the only possible isomorphisms occur for:   

\begin{enumerate}

\item $g=2, n\geq 0$, $g'=0, n'=n+6$ and 
$n'=\left[\frac{n+6}{2} \right]$, so that $n=0$. In this case, $(g,n)=(2,0)$ and $(g',n')=(0,6)$.  
\item  $g=1, n\geq 1$, $g'=0, n'=n+3$ and $n'=\left[\frac{n+3}{2} \right]$, so that $n\in\{2,3\}$.
In this case, either $(g,n)=(1,2)$ and $(g',n')=(0,5)$ or  $(g,n)=(1,3)$ and $(g',n')=(0,6)$.   
\item  $g=2, n\geq 0$ and $g'=1, n'=n+3$.  

\end{enumerate} 

The first case is one of the exceptional isomorphisms.
In the second case, we have to exclude the isomorphism between $\kC(S_{1,3})$ and $\kC(S_{0,6})$.
Assume then that such an isomorphism $\phi\co \kC(S_{1,3})\to \kC(S_{0,6})$ exists. 
From Lemma~\ref{sepinvariant}, it follows that $\phi$ sends profinite simple closed curves both of nonseparating and of boundary type on $S_{1,3}$
to profinite simple closed curves of boundary type on $S_{0,6}$. This implies that the link of a simple closed curve of nonseparating type 
and one of boundary type on $S_{1,3}$ are isomorphic. But, by Remark 4.7 in \cite{[B3]}, the first one is isomorphic to $\kC(S_{0,5})$ and 
the second one to $\kC(S_{1,1})\cong\kC(S_{0,4})\not\cong\kC(S_{0,5})$. Thus, we get a contradiction. 
This also excludes the case $(g,n)=(2,0)$ and $(g',n')=(1,3)$ of item (iii), since $\kC(S_{2,0})$ is isomorphic to $\kC(S_{0,6})$.

We have to exclude the third case for $n\geq 1$. Let us consider first the case $(g,n)=(2,1)$ and $(g',n')=(1,4)$. Let us assume that there is an
isomorphism $\phi\co \kC(S_{1,4})\to \kC(S_{2,1})$. By Lemma~\ref{sepinvariant}, both profinite simple closed curves of nonseparating and of boundary 
type on $S_{1,4}$ are sent to  profinite nonseparating simple closed curves on $S_{2,1}$. As above, it follows that the links of a simple closed curve 
of nonseparating and of boundary type, respectively, on $S_{1,4}$ are isomorphic. But these are isomorphic to $\kC(S_{0,6})$ and to $\kC(S_{1,3})$, 
respectively, which we have already proved not to be isomorphic, a contradiction.

Let us consider eventually the third case for $n\geq 2$. In this case, an isomorphism $\phi\co \kC(S_{2,n})\to \kC(S_{1,n+3})$ sends a profinite 
simple closed curve of boundary type on $S_{2,n}$ either to a profinite simple closed curve of boundary type or to a profinite nonseparating 
simple closed curve on $S_{1,n+3}$. Considering the isomorphism induced on the respective links and possibly using induction on $n$, 
we are reduced to one of the situations which we have already excluded.
\end{proof}

\subsection{Automorphisms of the procongruence curve complex}
We say that an element $f\in\Aut(\kC(S))$ is \emph{type preserving} if $\s$ and $f(\s)$ have the same topological type for all $\s\in\kC(S)$.
We then have:

\begin{theorem}\label{typepreservation}Let $S=S_{g,n}$ be a hyperbolic surface; if $S$ is not of type $(1,2)$, every
automorphism of $\kC(S)$ is type preserving. If $S=S_{1,2}$, an element of $\Aut(\kC(S))$ is type 
preserving if and only if it preserves the set of separating profinite simple closed curves. 
\end{theorem}

\begin{proof}The statement of the theorem is empty for $d(S)=0$ and obvious for $d(S)=1,2$. Hence, from now on, we assume that $d(S)>2$. 
As a first step, let us prove the following lemma:

\begin{lemma}\label{separatingpreserved}Let $S$ be such that $d(S)>2$. Then, every automorphism of $\kC(S)$ preserves the subsets of 
profinite simple closed curves of $\kC(S)_0$ which are of separating, nonseparating and of boundary type.
\end{lemma}

\begin{proof}By Lemma~\ref{sepinvariant}, we just have to show that the bijection $\hL^\mathrm{nsep}(S)\simeq\hL^\mathrm{nsep}(S)$ induced
by an element $f\in\Aut(\kC(S))$ preserves topological types. But this follows observing that, by Remark 4.7 in \cite{[B3]},
the dimension of the link in $\kC(S)$ of a profinite curve of boundary type is smaller than the dimension of the link of a profinite nonseparating curve.
\end{proof}

The next step completes the proof of Theorem~\ref{typepreservation} at the level of $0$-simplices: 

\begin{lemma}\label{typepreservationandseparation}Let $S$ be such that $d(S)>2$. Then, every automorphism of $\kC(S)$ preserves the
topological types of $0$-simplices.
\end{lemma} 

\begin{proof}Precomposing and then composing a given $\phi\in Aut(\kC(S))$ with the actions by suitable elements of $\kG(S)$,
we are reduced to consider the case when both $\a$ and $\phi(\a)$ belong to $C(S)_0\subset\kC(S)_0$.
By Lemma~\ref{separatingpreserved}, it is then enough to show that $\phi$ sends a separating simple closed curve
$\a$ on $S$, such that the connected components of $S\ssm\a$ both have modular dimension $>0$, to a separating simple closed curve of the
same topological type. Since $\phi$ induces an isomorphism $L^-_{\kC(S)}(\a)\cong L^-_{\kC(S)}(\phi(\a))$, from Lemma~\ref{connectedduallink},
it immediately follows that also the connected components of $S\ssm\phi(\a)$ have modular dimension $>0$.

Let $S\ssm\a=S_1\coprod S_2$ and $S\ssm\phi(\a)=S_1'\coprod S_2'$. By Theorem~\ref{stabmulticurves}, the inclusion $S_i\subset S$ (resp.\
$S'_i\subset S$) induces a monomorphism of procongruence mapping class groups $\kG(S_i)\subset \kG(S)$ (resp.\
$\kG(S'_i)\subset \kG(S)$) and then of procongruence curve complexes $\kC(S_i)\subset \kC(S)$ (resp.\ $\kC(S'_i)\subset\kC(S)$), for $i=1,2$.
As we saw in the proof of Lemma~\ref{connectedduallink}, we have:
\[L^-_{\kC(S)}(\a)=\kC(S_1)^-\coprod\kC(S_2)^-\hspace{1cm}\mbox{and}\hspace{1cm}L^-_{\kC(S)}(\phi(\a))=\kC(S'_1)^-\coprod\kC(S'_2)^-.\]
Since $C(S)^-$ is nonempty and path-connected for $d(S)\geq 1$ and identifies with a dense subset of $\kC(S)^-$, it follows that $\kC(S)^-$ is also connected
for $d(S)\geq 1$. We conclude that both $L^-_{\kC(S)}(\a)$ and $L^-_{\kC(S)}(\phi(\a))$ consist of two connected components which are then preserved by 
$\phi$. Let us suppose, for instance, that $\phi$ maps $\kC(S_i)^-$ to $\kC(S_i')^-$, for $i=1,2$. This implies that $\phi$ induces an isomorphism
of procongruence curve complexes $\kC(S_i)\cong\kC(S_i')$, for $i=1,2$. Then, Theorem~\ref{nonisomorphism}, together with the identities 
$g(S_1)+g(S_2)=g(S_1')+g(S_2')$ and $n(S_1)+n(S_2)=n(S_1')+n(S_2')$, yields the lemma.
\end{proof}

We conclude the proof of Theorem~\ref{typepreservation} by double induction on the dimension of the simplex $\s\in\kC(S)$ and on $d(S)$.
For $d(S)\leq 2$, the statement of the theorem is obvious and, for $d(S)>2$ and $\dim\s=0$, it is given by
Lemma~\ref{typepreservationandseparation}. As usual, we can assume that $\s\in C(S)\subset\kC(S)$ and, by the induction hypothesis, 
that a face $\s'$ of $\s$ of maximal dimension is fixed by the given automorphism $\phi\in\Aut(\kC(S))$. Then, also that the connected components 
of the dual link of $\s'$ are fixed by $\phi$. 

The conclusion then follows from the induction hypothesis applied to the connected component of the surface $S\ssm\s'$ which contains 
the simple closed curve $\g=\s\ssm\s'$ and the restriction of $\phi$ to that component.
\end{proof}

\section{Curve complexes and moduli spaces}\label{sect6}
\subsection{Geometric interpretation of the curve complex and of its congruence completion}\label{curvecomplexint}
Let $\cT(S)$ be the Teichm\"uller space associated to the surface $S$. The \emph{Harvey cuspidal bordification} $\wh{\cT}(S)$ 
of $\cT(S)$ (cf.\  \cite{Harvey}) is described as follows. 
Let $\bM(S)$ be the DM compactification of the stack  $\cM(S)$ and 
$\dd\bM(S)=\bM(S)\ssm \cM(S)$ be its DM boundary, which is a  normal crossings divisor.  
Let $\wh{\cM}(S)$ be the real oriented blow-up of the complex DM stack $\bM(S)_\C$  along the DM boundary $\dd\bM(S)_\C$. 
This is a real analytic DM stack with corners whose boundary $\dd\wM(S):=\wM(S)\ssm\cM(S)_\C$ 
is homotopic to a deleted tubular neighborhood of the DM boundary of $\bM(S)_\C$. 

For an intrinsic construction of $\wh{\cM}(S)$, we
apply Weil's restriction of scalars $\Res_{\C/\R}$ to the complex DM stack $\bM(S)_\C$ and blow up $\Res_{\C/\R}\bM(S)_\C$ along the codimension
two substack $\Res_{\C/\R}\dd\bM(S)_\C$. The real oriented blow-up $\wh{\cM}(S)$ is then obtained taking the set of real points of this blow-up 
and cutting it along the exceptional divisor (which has codimension $1$).
The natural projection $\wh{\cM}(S)\to\bM(S)_\C$ restricts to a bundle in $k$-dimensional tori over each codimension $k$  open stratum. 

The Harvey cuspidal bordification $\widehat{\cT}(S)$ is the universal cover of the real analytic stack
$\wh{\cM}(S)$. It can be shown that $\widehat{\cT}(S)$ is representable and thus a real analytic manifold with corners (cf.\ Section~4 in \cite{[B4]}).
The inclusion $\cM(S)_\C\hookra\wh{\cM}(S)$  is a homotopy equivalence and then induces an inclusion of the respective
universal covers $\cT(S)\hookra\widehat{\cT}(S)$, which is also a homotopy equivalence. The \emph{ideal boundary} of the Teichm{\"u}ller space 
$\cT(S)$ is defined to be $\dd\,\widehat{\cT}(S):=\widehat{\cT}(S)\ssm\cT(S)$.
 
The Harvey cuspidal bordification $\widehat{\cT}(S)$ is endowed with a natural action of
the mapping class group $\G(S)$ and has the property that the nerve of the cover of the ideal boundary $\partial\,\widehat{\cT}(S)$ 
by (analytically) irreducible components is described by the curve complex $C(S)$ (cf.\ \cite{Harvey}).
More precisely, the stratum associated to a $k$-simplex $\s\in C(S)$ is naturally isomorphic to $\cT(S\ssm\s)\times\R^{k+1}$.
In particular, these strata are all contractible. Hence, there is a $\G(S)$-equivariant \emph{weak} homotopy equivalence  
between the ideal boundary $\partial\,\widehat{\cT}(S)$ and the geometric realization of $C(S)$ (cf.\ Theorem~2 in \cite{Harvey}). 

From this description of the curve complex $C(S)$, it is clear that, for $\G^\l$ a level of $\G(S)$, the simplicial finite set $C^\l(S)_\bt:=C(S)_\bt/\G^\l$
describes the nerve of the DM boundary of the level structure $\bM^\l$ over $\bM(S)$ corresponding to the subgroup $\G^\l$.
Therefore, the complex of profinite curves $\kC(S)$ describes the DM boundary of the inverse limit of all the geometric level structures.

\subsection{Geometric interpretation for the pants graph and its congruence completion}\label{geomintpants}
The role that, for the curve complex $C(S)$, is played by the Harvey bordification $\wh{\cT}(S)$ 
of Teichm\"uller space is  now played by the \emph{Bers bordification} $\ol{\cT}(S)$ for the pants graph $C_P(S)$. This is obtained from $\wh{\cT}(S)$ by 
collapsing the real affine spaces which appear in the boundary components associated to nonperipheral simple closed curves on $S$.
The Bers boundary of the Bers bordification $\ol{\cT}(S)$ is then the complement $\dd\ol{\cT}(S):=\ol{\cT}(S)\ssm\cT(S)$.
The irreducible closed stratum of the Bers boundary associated to a simplex $\s\in C(S)$ is then isomorphic to the Bers bordification
$\ol{\cT}(S\ssm\s)$ of $\cT(S\ssm\s)$. There is a natural map $\ol{\cT}(S)\to\bM(S)$ with infinite ramification along the Bers boundary
and with functorial restriction to the strata of the Bers boundary. 

The above description of the Bers boundary shows that there is a natural injective map from the set of facets of $C(S)$, which forms the set of vertices 
of the pants graph $C_P(S)$, to the Bers boundary $\dd\ol{\cT}(S)$. This map sends a maximal multicurve $\s$ to the
$0$-dimensional stratum $\ol{\cT}(S\ssm\s)$ of $\dd\ol{\cT}(S)$ and we can extend it to a continuous map from 
the geometric realization $|C_P(S)|$ of $C_P(S)$
to the topological space underlying $\dd\ol{\cT}(S)$ in the following way. 
If  $\s'$ is a submaximal (i.e.\ of dimension $\dim C(S)-1$) simplex of $C(S)$, then the associated closed $1$-dimensional
stratum $\ol{\cT}(S\ssm\s')$ of $\dd\ol{\cT}(S)$ is  identified with the cuspidalization 
$\ol{\H}(S')$ of the hyperbolic plane $\H(S')$, which we regard as the 
universal cover of the unique connected component of positive modular dimension $S'$ of $S\ssm\s'$. 
On the other hand, the geometric realization of the pants graph $C_P(S')$ is naturally identified with the $1$-skeleton of the Farey triangulation 
$F(S')$ of $\ol{\H}(S')$. Since, as we already observed in the proof of Lemma~\ref{proofrigiditypants},  every edge of $C_P(S)$ is contained in a
unique Farey subgraph $F(S')=C_P(S')$, we can use the previous identification to define a continuous injective map $\Psi\co |C_P(S)|\hookra\dd\ol{\cT}(S)$.

It is quite remarkable that a similar property is enjoyed by the finite quotients $C_P^\l(S)_\bt:=C_P(S)_\bt/\G^\l$, for all levels $\G^\l$ of $\G(S)$
satisfying some mild conditions. We need the following definition:

\begin{definition}For $S$ a hyperbolic surface of modular dimension $d=d(S)$, the \emph{Fulton curve} $\cF(S)$ is the $1$-dimensional 
closed substack of the moduli stack $\bM(S)$ which parameterizes curves with at least $d-1$ nodes. More generally, for a level structure
$\bM^\l(S)$ over $\bM(S)$, the \emph{Fulton curve} $\cF^\l(S)$ is the inverse image of $\cF(S)$ via the natural morphism $\bM^\l(S)\to\bM(S)$.
\end{definition}

Let us observe that, in general, $\cF^\l(S)$ is a $1$-dimensional proper reduced DM stack curve with at most multicross singularities. 
We will be essentially interested in the case 
when the level structure $\bM^\l(S)$ and then the Fulton curve $\cF^\l(S)$ are representable. The importance of this $1$-dimensional closed
stratum in $\bM(S)$ was first recognized in a conjecture formulated by W. Fulton, hence our name and the notation $\cF$ (see \cite{[GKM]}).

\begin{proposition}\label{triangulation}Let $\G^\l$ be a level of $\G(S)$ contained in an abelian level of order $m\geq 2$. Then, the quotient 
$C_P(S)_\bt/\G^\l$ is the simplicial set associated to a simplicial complex $C^\l_P(S)$ whose geometric realization  is identified with the $1$-skeleton 
of a triangulation of (the coarse moduli space of) $\cF^\l(S)_\C$. This triangulation restricts on each irreducible component of $\cF^\l(S)_\C$ to the one 
induced by the Farey triangulation of $\ol{\H}$.
\end{proposition}

\begin{proof}A subgroup $H$ of $\SL_2(\Z)$ contained in an abelian level of order $m\geq 2$
acts without inversions on the Farey triangulation $F(S')$ of $\ol{\H}(S')$ and its quotient $F(S')/H$ is a triangulation of the closed surface $\ol{\H}(S')/H$.
From Theorem~5.2 in \cite{[B4]}, it follows that, for $\s$ a submaximal simplex of $C(S)$, the image $\bar\G^\l_\s$ of the stabilizer $\G^\l_\s$ in the 
mapping class group $\G(S\ssm\s)$ is contained in an abelian level of order $m\geq 2$. Since $\G^\l_\s$ is also the stabilizer of the associated closed
stratum $\ol{\cT}(S\ssm\s)$ and $\G^\l_\s$ acts on it through its quotient $\bar\G^\l_\s$, the conclusion follows.
\end{proof}

From Proposition~\ref{triangulation}, we immediately deduce the nontrivial fact that, for $\G^\l$ as above, the finite quotient $C_P^\l(S)_\bt$ is in fact 
the simplicial set associated to an abstract simplicial complex $C_P^\l(S)$ and that the same is true for the simplicial profinite set $\kC_P(S)_\bt$. 
We denote by $\kC_P(S)$ the abstract simplicial profinite complex whose associated simplicial profinite set is $\kC_P(S)_\bt$ and henceforth this 
will be the object to which we refer when we talk about the procongruence pants graph. 

Let $|\kC_P(S)|$ be the geometric realization of $\kC_P(S)$ and $\dd\ol{\M}(S):=\varprojlim_{\l\in\Lambda}\dd\bM(S)^\l$ be the DM boundary 
of the inverse limit $\ol{\M}(S):=\varprojlim_{\l\in\Lambda}\bM(S)^\l$ of all geometric 
level structures over $\bM(S)$. We then see that the continuous injective map 
$\Psi\co |C_P(S)|\hookra\dd\ol{\cT}(S)$, defined above, extends to a continuous injective map 
\[\check\Psi\co |\kC_P(S)|\hookra\dd\ol{\M}(S)_\C.\]

The irreducible closed strata of $\dd\ol{\M}(S)$ are parameterized by the simplices of the curve complex $\kC(S)$. For $\s\in\kC(S)$, let us then denote
by $\Delta_\s$ the corresponding irreducible closed stratum of $\dd\ol{\M}(S)$. In particular, the $1$-dimensional strata are parameterized by 
$(d(S)-2)$-simplices of $\kC(S)$ and, for $\s\in \kC(S)_{d(S)-2}$, the intersection $\check\Psi(|\kC_P(S)|)\cap(\Delta_\s)_\C$ is the $1$-skeleton 
of a triangulation of the complex stratum $(\Delta_\s)_\C$. Let us denote by $\wh{F}_\s$ the subgraph of $\kC_P(S)$ whose geometric realization is 
$\check\Psi(|\kC_P(S)|)\cap(\Delta_\s)_\C$.

\begin{definition}\label{profarey}
For $\s\in \kC(S)_{d(S)-2}$, we call $\wh{F}_\s$ the \emph{profinite Farey subgraph of $\kC_P(S)$ associated to $\s$}. 
It is clear that, for $\s\in \kC(S)_{d(S)-2}$, the profinite Farey subgraphs $\wh{F}_\s$ cover the procongruence pants graph $\kC_P(S)$. 
\end{definition}

In Section~\ref{discrec}, we denoted by $F_\s$ the Farey subgraph of $C_P(S)$ associated to a simplex $\s\in C(S)_{d(S)-2}$ and noticed that all 
Farey subgraphs of the pants graph $C_P(S)$ arise uniquely in this way. It is then clear that the closure of $F_\s$ in $\kC_P(S)$ is the 
profinite Farey subgraph $\wh{F}_\s$. 

\begin{remark}
Here, the "hat" notation is suggestive of the nontrivial fact that, by the subgroup congruence property in genus $\leq 2$, 
a set  $\{H^\l\}_{\l\in\Lambda}$ of finite index subgroups of $\SL_2(\Z)$ such that $\wh{F}_\s=\varprojlim_{\l\in\Lambda}F_\s/H^\l$ 
forms a basis of neighborhoods of the identity for the profinite topology. 
\end{remark}

\subsection{Relation with the procongruence curve complex}
In Section~\ref{rigiditycurvecomplex}, we saw how the $1$-skeleton of the curve complex and then the curve complex itself could be recovered
from the pants graph. This was the basis of Margalit rigidity Theorem~\ref{rigiditypants}. The same turns out to be true in the procongruence
setting thanks to the results of the previous sections. As in the topological case, the intermediary between the two objects is the dual graph
$\kC^\ast(S)$ of the procongruence curve complex $\kC(S)$. This is naturally an abstract simplicial profinite complex since the profinite topology
on the set of facets of $\kC(S)$ (i.e.\ vertices of $\kC^\ast(S)$) induces a profinite topology also on the set of edges of $\kC^\ast(S)$. 
In what follows, we regard the dual graph $\kC^\ast(S)$ of $\kC(S)$ as a $1$-dimensional abstract simplicial profinite complex.
We then have:

\begin{lemma}\label{produalgraph}The procongruence curve complex $\kC(S)$ can be reconstructed from its dual graph $\kC^\ast(S)$.
In particular, there is a natural isomorphism $\Aut(\kC^\ast(S))\cong\Aut(\kC(S))$.
\end{lemma}

\begin{proof}By Lemma~\ref{maxcomplete}, we can reconstruct the $1$-skeleton $\kC^{(1)}(S)$ of $\kC(S)$ from $\kC^\ast(S)$ as an abstract
simplicial complex. Therefore, since the profinite topologies on the sets of vertices and facets of $\kC(S)$ determine reciprocally, 
we can recover $\kC^{(1)}(S)$ from $\kC^\ast(S)$ as an abstract simplicial profinite complex. By Theorem~\ref{isomorphism},
$\kC(S)$ is a flag complex and so it can be recovered (again as an abstract simplicial profinite complex) from its $1$-skeleton and then 
from its dual graph $\kC^\ast(S)$. The last statement of the lemma also follows.
\end{proof}

The next step is to reconstruct the dual graph $\kC^\ast(S)$ of $\kC(S)$ from the procongruence pants graph $\kC_P(S)$:

\begin{lemma}\label{edgespantsgraph}\leavevmode
\begin{enumerate}
\item Two vertices of $\kC_P(S)$ are joined by an edge only if they have in common exactly $d(S)-1$ profinite simple closed curves. Therefore,
an edge of $\kC_P(S)$ is contained in a unique profinite Farey subgraph $\wh{F}_\s$ and the same statement holds for a triangle in $\kC_P(S)$.
\item The profinite Farey subgraph $\wh{F}_e$, containing the edge $e$ of $\kC_P(S)$, is intrinsically obtained by the following procedure: 
intersect the stars $\Star_{v_0}$ and $\Star_{v_1}$ in $\kC_P(S)$ of the vertices $v_0,v_1$ of $e$; take the full subcomplex of $\kC_P(S)$ 
generated by $e\cup(\Star_{v_0}\cap \Star_{v_1})$; iterate the procedure on all the new edges thus obtained and continue this way; 
the union of all these subcomplexes is then a dense subgraph of the profinite Farey subgraph $\wh{F}_e$ associated to the edge $e$.
\item Two vertices of $\kC^\ast(S)$ are joined by an edge if and only if they have in common exactly $d(S)-1$ profinite simple closed curves.
Therefore, there is a natural embedding $\kC_P(S)\subset\kC^\ast(S)$ and $\kC^\ast(S)$ is obtained from the procongruence pants graph 
replacing each profinite Farey subgraph $\wh{F}_\s$, for $\s\in \kC(S)_{d(S)-2}$, with the (profinite) complete subgraph on the vertex set of $\wh{F}_\s$.
\end{enumerate}
\end{lemma}

\begin{proof}(i): Every edge $e$ of $\kC_P(S)$ is contained in some profinite Farey subgraph $\wh{F}_\s$. The vertices $v_0$ and $v_1$ of $e$
then contain the set of $d(S)-1$ profinite simple closed curves $\s$. Since they are distinct and 
each one consists of a set of $d(S)$ profinite simple closed curves, they should have in common exactly the elements of $\s$. This also shows that the only profinite Farey subgraph of $\kC_P(S)$ which contains $e$ is $\wh{F}_\s$.
The same statement is true for any triangle of $\kC_P(S)$ which has $e$ for edge.
\smallskip

\noindent
(ii): This is a formal consequence of the previous item.
\smallskip

\noindent
(iii): The first statement is essentially the definition of the dual graph $\kC^\ast(S)$. The following statements follows from the first one, the 
previous items of the lemma and the parameterization of profinite Farey subgraphs of $\kC_P(S)$ by $(d(S)-1)$-simplices of $\kC(S)$.
\end{proof}

An immediate consequence of Lemma~\ref{edgespantsgraph}, Lemma~\ref{produalgraph} and Theorem~\ref{typepreservation} is then:

\begin{theorem}\label{autinjection}Every continuous automorphism of the procongruence pants graph $\kC_P(S)$ sends a profinite Farey subgraph
to another profinite Farey subgraph. Therefore there is a natural continuous monomorphism $\Aut(\kC_P(S))\hookra\Aut(\kC(S))$.
In particular, the continuous automorphisms of the procongruence pants graph  preserve the topological types of its vertices.
\end{theorem}

\section{Automorphisms of the procongruence mapping class group}\label{automorphisms}
\subsection{The $\ast$-condition}As we saw in Section~\ref{rigiditycurvecomplex}, a basic property of the mapping class group $\G(S)$
is that any automorphism of this group preserves the set of stabilizers for the action of $\G(S)$ on the curve complex $C(S)$. 
This property allows to define a natural representation $\Aut(\G(S))\to\Aut(C(S))$. It is not known whether a similar property holds for the procongruence 
mapping class group $\kG(S)$. For this reason, in order to be able to define an action on the procongruence curve complex, we need to restrict to
elements of $\Aut(\kG(S))$ which satisfy a similar property:

\begin{definition}\label{inertiapreservingdef}For $\kG^\l$ an open subgroup of $\kG(S)$, let $\Aut^\ast(\kG^\l)$ be the closed subgroup 
of $\Aut(\kG^\l)$ consisting of those automorphisms which preserve the set of subgroups $\{\kG^\l_\g\}_{\g\in\hL(S)_0}$ of $\kG^\l$. 
\end{definition}

\begin{proposition}\label{inertiapreservingprop}There is a natural continuous homomorphism: 
\[\check\Theta^\l\co\Aut^\ast(\kG^\l)\to\Aut(\kC(S)).\]
\end{proposition}

\begin{proof}The natural injective and continuous $\kG^\l$-equivariant map $\cG^\l_0\co\kC(S)_0\hookrightarrow\cG(\kG(S))$, defined in
Remark~\ref{grouptheoretic}, shows that an element $f\in\Aut^\ast(\kG^\l)$ induces a continuous action on the profinite set of $0$-simplices of 
$\kC(S)$. Since $\kC(S)$ is a flag complex, the proposition follows if we prove that, for $\{\g_0,\g_1\}\in\kC(S)_1$, we have $\{f(\g_0),f(\g_1)\}\in\kC(S)_1$. 

This immediately follows from Corollary~\ref{centralizers multitwists}, which 
implies that $\{\g_0,\g_1\}\in\kC(S)_1$ if and only if the centers of $\kG_{\g_0}^\l$ and $\kG_{\g_1}^\l$ commute and this condition is obviously 
preserved by an automorphism of $\kG^\l$.
\end{proof}
 
\begin{theorem}\label{faithfulness}\leavevmode
\begin{enumerate}
\item The natural action of $\kG(S)$ on $\kC(S)$ and $\kC_P(S)$ factors through:
\begin{enumerate}
\item for $S\neq S_{0,4}$, monomorphisms: 
\[\kG(S)/Z(\kG(S))\hookra\Aut(\kC(S))\hspace{0.4cm}\mbox{and}\hspace{0.4cm}\kG(S)/Z(\kG(S))\hookra\Aut(\kC_P(S));\] 
\item for $S= S_{0,4}$, monomorphisms: 
\[\kG_{0,[4]}/K_4\hookra\Aut(\kC(S))\hspace{0.4cm}\mbox{and}\hspace{0.4cm}\kG_{0,[4]}/K_4\hookra\Aut(\kC_P(S)),\] 
where $K_4\cong\{\pm 1\}\times\{\pm 1\}$ is the Klein subgroup of $\G_{0,[4]}$. 
\end{enumerate}
\item  Let $\kG^\l$ be an open subgroup of $\kG(S)$. Then, we have:
\begin{enumerate}
\item for $S\neq S_{0,4}$, the kernel of the natural homomorphism $\check\Theta^\l$ is naturally isomorphic to the group 
$\Hom(\kG^\l/Z(\kG^\l),Z(\kG^\l))$;
\item for $S= S_{0,4}$, there is a natural exact sequence:
\[1\to\Der(\kG^\l/(\kG^\l\cap K_4),\kG^\l\cap K_4)\to\ker\check\Theta^\l\to\Aut(\kG^\l\cap K_4),\]
where $\Der(\kG^\l/(\kG^\l\cap K_4),\kG^\l\cap K_4)$ denotes the group of continuous derivations (crossed homomorphisms).
\end{enumerate}

\end{enumerate}
\end{theorem}

\begin{proof}(i): For all $f\in\kG(S)$ and a profinite Dehn twist $\tau_\g$, with $\g\in\hL(S)_0$, there is the identity $f\cdot\tau_\g\cdot f^{-1}=\tau_{f(\g)}$.
This implies that an element $f$ of $\kG(S)$ acts trivially on the curve complex $\kC(S)$ if and only if it centralizes all profinite Dehn twists and so
centralizes the pure procongruence mapping class group $\kPG(S)$.
By Theorem~\ref{slim}, this implies that, for $S\neq S_{0,4}$, the element $f$ is in the center of $\kG(S)$ and, for $S= S_{0,4}$, that $f\in K_4$. 

We have thus proved the part of item (i) about homomorphisms to $\Aut(\kC(S))$. The part about homomorphisms to $\Aut(\kC_P(S))$ 
then follows from Theorem~\ref{autinjection}.
\smallskip

\noindent
(ii): Assume first that $Z(\kG^\l)=\{1\}$ and, if $S=S_{0,4}$, that $\kG^\l\cap K_4=\{1\}$. 
By the previous item, if $\kG^\l$ satisfies these hypotheses, 
the natural homomorphism $\Inn(\kG^\l)\to\Aut(\kC(S))$ is injective. By Lemma~\ref{grouplemma}, 
its extension $\Aut^\ast(\kG^\l)\to\Aut(\kC(S))$ is then also injective.

Consider now the general case. By Theorem~\ref{slim}, the center of $\kG^\l$ can be nontrivial only for $S=S_{g,n}$ of type $(2,0)$, $(1,1)$ 
or $(1,2)$. When this happens, the center of $\kG^\l$ is generated by the hyperelliptic involution $\iota$ and the quotient $\kG^\l/\langle\iota\rangle$ 
identifies with a center-free open subgroup of a genus $0$ procongruence mapping class group. Moreover, there is
an identification of $\kC(S)$ with the respective procongruence curve complex. From the previous item, it then follows that the homomorphism 
$\Aut^\ast(\kG^\l)\to\Aut(\kC(S))$ factors through the injective homomorphism $\Aut^\ast(\kG^\l/Z(\kG^\l))\hookra\Aut(\kC(S))$.
For $S=S_{0,4}$, we have that $\G(S)/K_4\cong\PSL_2(\Z)$ and the latter group acts faithfully on the set $C(S)$. 
Hence, $\kG(S)/K_4\cong\wh{\PSL}_2(\Z)$ acts faithfully on $\kC(S)$. 

Therefore, an automorphism $f\in\Aut^\ast(\kG^\l)$ acts trivially on $\kC(S)$ if and only if it induces the trivial automorphism on $\kG^\l/Z(\kG^\l)$ 
and $\kG^\l/\kG^\l\cap K_4$, respectively. The conclusion now follows
from a profinite version of Lemma~\ref{kernelautbis}:

\begin{lemma}\label{kernelautpro}Let $G$ be a profinite group and $\Aut(G)_A$ the subgroup of elements of $\Aut(G)$ which preserve a closed 
normal abelian subgroup $A$ of $G$. Then, there is a natural exact sequence:
\[1\to\Der(G/A,A)\to\Aut(G)_A\to\Aut(G/A)\times\Aut(A),\]
where the action of $G/A$ on $A$ is induced by the inner action of $G$ on $A$.
If, moreover, the subgroup $A$ is central, there is also an exact sequence:
\[1\to H^1(G/A,A)\to\Out(G)_A\to\Out(G/A)\times\Aut(A).\]
\end{lemma}

\begin{proof}This is a weak version of Lemma~1.5.5 in \cite{[N]} and its proof is identical to the proof of Lemma~\ref{kernelautbis}, 
except that all homomorphisms and maps considered there are to be taken continuous with respect to the profinite topologies of the
groups involved.
\end{proof}
\end{proof}

\section{Automorphisms of the procongruence pants complex}\label{prorigidity}
We now turn to the study of the automorphism group of the procongruence pants graph $\kC_P(S)$, 
where $S$ is hyperbolic and connected. First we recall that, by (i) of Theorem~\ref{faithfulness}, 
the natural representation of $\kG(S)\to\Aut(\kC_P(S))$ factors through a homomorphism: 
\[\Inn(\kG(S))\to \Aut(\kC_P(S)),\]
which is injective for $S\neq S_{0,4}$ and whose kernel, for $S= S_{0,4}$, is the Klein subgroup of $\G_{0,[4]}$.
In analogy with what happens in the topological setting, the image of this map has finite index in the codomain:

\begin{theorem}\label{completepantsrigidity}
For $S\neq S_{1,2}$ a connected hyperbolic surface such that $d(S)>1$, there is an exact sequence:
\[1\to\Inn(\kG(S))\to \Aut(\kC_P(S))\to \prod_{O(S)} \{\pm 1\},\] 
where $O(S)$ is the finite set of $\G(S)$-orbits of $(d(S)-1)$-multicurves. 
For $S$ of type $(1,2)$, the group $\Aut(\kC_P(S_{1,2}))$ must be replaced with the subgroup of 
those automorphisms preserving the set of separating curves.
\end{theorem}  

Note that both mapping class groups $\G_{0,[4]}$ and $\G_{1,1}=\SL_2(\Z)$ act on the hyperbolic plane $\H$ through their common
quotient $\PSL_2(\Z)$, with finite kernel $Z$, respectively, the Klein subgroup $K_4$ and the center $\{\pm 1\}$.
For $d(S)=1$, there holds the following simpler and more precise form of Theorem~\ref{completepantsrigidity}:

\begin{proposition}\label{completepantsrigidity1}There is a natural short exact sequence:
\[1\to\wh{\PSL}_2(\Z)\to\Aut(\wh{F})\to\{\pm 1\}\to 1,\]
where $\wh{F}$ denotes the profinite Farey graph. 
\end{proposition}

In contrast with Theorem~\ref{completepantsrigidity}, the proof of Proposition~\ref{completepantsrigidity1} is relatively elementary and given in Section \ref{proof82}. 
As we saw in Section~\ref{geomintpants}, for $d(S)=1$, the pants graph $C_P(S)$ is just the $1$-skeleton $F$ of the Farey ideal triangulation of the
hyperbolic plane $\H$ and the procongruence pants graph $\kC_P(S)$, which coincides with the profinite pants graph $\hC_P(S)$, is the inverse
limit $\wh{F}=\varprojlim_{\l\in\L}F/\G^\l$, where $\{\G^\l\}_{\l\in\L}$ is the set of finite index normal subgroups of $\PSL_2(\Z)$ which are contained in some
abelian level $\G(m)$ of order $m\geq 2$. Each finite quotient $F^\l:=F/\G^\l$ is the $1$-skeleton of the triangulation of the closed 
Riemann surface $\bM^\l:=\ol{\H}/\G^\l$ induced by the Farey triangulation on $\ol{\H}$. 
The proof of Proposition~\ref{completepantsrigidity1} is then based on the properties of flat surfaces recalled in the following section.

\subsection{Flat surfaces with conical points}\label{flatsurfaces}
A {\em piecewise flat} surface $S$ is a surface endowed with a triangulation $\Delta$ 
whose 2-simplexes are Euclidean triangles and transition maps between adjacent 
2-simplexes are plane isometries. The length metric provides a piecewise flat metric 
on $S$. To every vertex of the triangulation we can associate its conical angle which is 
the sum of Euclidean angles of triangles incident to the vertex. When the conical angle 
is $2\pi$ the vertex is called {\em regular}, otherwise it is a singular (conical) point of $S$. 
Let $S'=S-V$, where $V$ is the set of singular points. 
Then $S'$ carries a well-defined Riemannian flat metric. Around a singular point $p\in V$ of 
conical angle $\theta$ we can use polar coordinates $(r,\varphi)\in \R_+\times \R/\theta \Z$, where 
$r$ is the distance at $p$ and $\varphi$ the angular variable mod $\theta$ on the flat 
cone. In these local coordinates the flat Riemannian metric has the form 
\[  d r^2 + r^2 d\varphi^2 = 
 |z|^{\frac{2\theta -2}{2\pi}} |dz|^2, \;\, {\rm where} \; z= \frac{\theta}{2\pi}\left(re^{i\varphi}\right)^{\frac{\theta}{2\pi}}\]
The charts $(U,z)$ as above are complex charts around each singularity $p$, which can be chosen to cover $S$. 
The transition maps between two charts are obviously conformal maps, as the 
metric above shows that around each singularity the flat Riemannian metric is conformal to a smooth 
metric. It follows that the collection of charts  of the form $(U,z)$ define a holomorphic 
structure on $S$, which will be called the Riemann structure associated to the 
piecewise flat structure. Troyanov proved that conversely, for any compact connected 
Riemann surface without boundary, finite set of points $V\subset S$ and conical angles 
$\theta(p), p\in V$ satisfying the Gauss Bonnet formula 
\[ \chi(S)+  \sum_{p\in V} \left (\frac{\theta(p)}{2\pi}-1\right) =0\]
there exists a unique up to homothety conformal flat metric  on $S$ with conical angles $\theta(p)$ at $p\in V$  (see \cite{[Tr]} for more details).

Let now $\Delta$ be some triangulation of a closed orientable surface $S$. 
Consider now the {\em piecewise flat equilateral} surface $S(\Delta)$ 
obtained by endowing every 2-simplex in $\Delta$ with the Euclidean metric of an 
equilateral triangle. We also denote by $S(\Delta)$ the associated Riemann surface structure.

The crucial observation for what follows is that, given another triangulated surface $S(\Delta')$, a simplicial ramified finite covering 
$f\co S(\Delta)\to S(\Delta')$, i.e.\ a map which is a topological covering outside a finite set of points and preserves the triangulations, 
induces a holomorphic or antiholomorphic map $F\co S(\Delta)\to S(\Delta')$.

\subsection{Proof of Proposition~\ref{completepantsrigidity1}}\label{proof82}
We begin with the remark that the conformal class associated to the structure of Riemann surface on $\bM^\l$ can be recovered from the finite
graph $F^\l$ by taking on $\bM^\l$ the unique flat metric with singularities contained in the vertex set of $F^\l$ and such that each edge of this 
graph has length $1$. Indeed, for each subgroup of finite index $\G^\l$ of $\PSL_2(\Z)$ contained in some nontrivial abelian level, there is a 
finite correspondence between $\bM^\l$ and $\bM^{(2)}$ which preserves the respective Farey triangulations. Since $\bM^{(2)}\cong\P^1$ 
and its Farey triangulation consists of two equilateral triangles with vertex set $\{0,1,\infty\}=\dd\bM^{(2)}$, the claim above follows. 
Let us then denote by $\Delta^\l$ the abstract $2$-dimensional simplicial complex which has $F^\l$ for $1$-skeleton and such that the natural
embedding $|F^\l|\hookra\bM^\l$ extends to a homeomorphism $|\Delta^\l|\cong\bM^\l$ which becomes a conformal isomorphism when 
$S^\l:=|\Delta^\l|$ is given the flat structure described in Section~\ref{flatsurfaces}.

For $\G^\l\subseteq\G^\mu$, let us denote by $\hat{\pi}_\l\co\wh{F}\to F^\l$ and $\pi_{\l\mu}\co F^\l\to F^\mu$ the natural projections. Then, 
a continuous automorphism $\phi$ of $\wh{F}$ determines and is determined by the directed inverse system of maps constructed as follows.
For every $\l\in\L$, there is a $\mu\in\L$ such that the composition $\phi_\l:=\hat{\pi}_\l\circ\phi\co\wh{F}\to F^\l$ factors
through a map $\phi_{\mu\l}\co F^\mu\to F^\l$. It is clear that, if $\phi_{\mu'\l'}\co F^{\mu'}\to F^{\l'}$ is another such map with $\G^{\l'}\subseteq\G^\l$
and $\G^{\mu'}\subseteq\G^\mu$, we have $\pi_{\l'\l}\circ\phi_{\mu'\l'}=\phi_{\mu\l}\circ\pi_{\mu'\mu}$. Therefore, the set of maps $\{\phi_{\mu\l}\}_{\l,\mu\in\L}$
is a directed inverse system with inverse limit the given map $\phi$.

Let us observe that the same argument above applies to the inverse automorphism $\phi^{-1}$. Let then $\{(\phi^{-1})_{\l\nu}\}_{\mu,\nu\in\L}$ be the
inverse system which determines $\phi^{-1}$ and start from the filtering inverse subsystem provided by the domains of the inverse system of maps
$\{(\phi^{-1})_{\l\nu}\}_{\mu,\nu\in\L}$ to construct the inverse system $\{\phi_{\mu\l}\}_{\l,\mu\in\L}$ which determines $\phi$.
We then have, for every $\nu$ as above, a commutative diagram:
\[\begin{array}{crc}
F^\mu&&\\
\hspace{0.3cm}\downarrow^{\pi_{\mu\nu}}&\searrow^{\phi_{\mu\l}}&\\
F^\nu&\stackrel{(\phi^{-1})_{\l\nu}}{\longleftarrow}&\!\!\!\!\!F^\l.
\end{array}\]
Let us denote with the same letters the continuous maps from the triangulated surfaces $S^\mu$ and $S^\l$ 
induced by the above maps. Since $\pi_{\mu\nu}\co S^\mu\to S^\nu$ is a ramified covering, the same is true for the map
$\phi_{\mu\l}\co S^\mu\to S^\l$. By the last remark in Section~\ref{flatsurfaces}, we then see that $\phi_{\mu\l}$ is a conformal or 
an anticonformal map and so induces a holomorphic or antiholomorphic map $\phi_{\mu\l}\co\bM^\mu\to\bM^\l$. The latter map factors through a  
holomorphic or antiholomorphic automorphism $\Phi_\mu$ of $\bM^\mu$ and the projection $\pi_{\mu\l}\co\bM^\mu\to\bM^\l$. 
In conclusion, we have $\phi=\varprojlim_{\mu\in\L'}\Phi_\mu$, for some filtering subsystem $\Lambda'$ of $\Lambda$, where either all $\Phi_\mu$ are 
holomorphic or all are antiholomorphic. The claim of the proposition then follows.

\subsection{Finite pants graphs and Fulton curves}In Proposition~\ref{triangulation}, we saw that, for $\G^\l$ a level of $\G(S)$ contained in 
an abelian level of order $m\geq 2$, the quotient of the pants graph $C_P(S)$ by the action of $\G^\l$ exists as a finite abstract simplicial 
complex $C^\l_P(S)$ whose geometric realization is the $1$-skeleton of a triangulation of the Fulton curve $\cF^\l(S)$ (more precisely, of its 
coarse moduli space). This is the triangulation
which, on every irreducible component of $\cF^\l(S)$, is induced by the Farey ideal triangulation of the hyperbolic plane. If $C^\l$ is one of these
irreducible components, it is clear that there is a finite correspondence, preserving the triangulations, between $C^\l$ and $\ol{\H}/\G(2)\cong\P^1$,
where $\G(2)$ is the abelian level of order $2$ of $\PSL_2(\Z)$. Therefore, if $C^\l$ is representable, e.g.\ the level structure $\bM^\l(S)$ is representable, 
the conformal class of the complex structure on $C^\l$ is determined by the choice of an orientation on the topological surface underlying $C^\l$ and by
the metric on the graph $|C^\l_P(S)|\cap C^\l$ where each edge is assigned length one. 

A \emph{Farey subgraph} $F^\l$ of $C_P^\l(S)$ is a subcomplex such that its geometric realization is the $1$-skeleton of the Farey triangulation of 
an irreducible component of $\cF^\l$. Modulo orientations, it is then clear that the complex structure of $\cF^\l$ is determined by the graph $C_P^\l(S)$
and its partition into Farey subgraphs. 

We say that a level $\G^\l$ of $\G(S)$ is \emph{representable} if the associated level structure $\bM^\l(S)$ is such. Let us summarize the above discussion
in the following proposition:

\begin{proposition}\label{conformalstructure}Let $\G^\l$ be a representable level of $\G(S)$ contained in some abelian level of order $m\geq 2$.
Except for the orientations of its irreducible components, the complex structure on the Fulton curve $\cF^\l(S)$ is determined by the graph 
$C^\l_P(S)$ and its partition into Farey subgraphs.
\end{proposition}

\subsection{Orientation data}\label{section:orientations}
Let $\Delta$ be the abstract $2$-dimensional simplicial complex which has the Farey graph $F$ for $1$-skeleton and such that the natural
embedding $|F|\hookra\ol{\H}$ extends to a homeomorphism $|\Delta|\cong\ol{\H}$ which becomes a conformal isomorphism when 
$S:=|\Delta|$ is given the flat structure described in Section~\ref{flatsurfaces}. An \emph{orientation of the Farey graph} $F$ is then
simply a choice of orientation for $\Delta$. We denote by "$+$" the orientation compatible with the one associated to the complex structure on $\H$
and by "$-$" the opposite orientation.

An orientation of the Farey graph $F$ induces an orientation on 
the finite quotients by means of the natural projections $\pi_\l\co\Delta\to\Delta^\l$ and  
thus on the profinite Farey graph $\wh F$. Indeed, in this way we get compatible orientations on every finite quotient $\Delta^\l$, for $\l\in\Lambda$, 
and Proposition~\ref{completepantsrigidity1} shows that any continuous automorphism of $\wh{F}$ either preserves or reverses such orientation. 
As above, we denote by "$+$" the orientation associated to the complex structure on $\H$ and by "$-$" the opposite orientation.

\begin{definition}\label{orientation}An orientation of the procongruence pants graph $\kC_P(S)$ is a choice of orientation for each profinite
Farey subgraph $\wh{F}_\s$ of $\kC_P(S)$, for $\s\in\kC(S)_{d(S)-2}$. Therefore, an orientation of $\kC_P(S)$ is determined by an element
of the profinite group $\prod_{\kC(S)_{d(S)-2}} \{\pm 1\}$.
\end{definition}

\begin{remark}\label{compatibility}By Theorem~\ref{autinjection}, the group of automorphisms $\Aut(\kC_P(S))$ preserves the partition of $\kC_P(S)$
in Farey subgraphs and, by Proposition~\ref{completepantsrigidity1}, the stabilizer of a Farey subgraph for this action either preserves or reverses
the orientation of this subgraph. Therefore, the action of $\Aut(\kC_P(S))$ on $\kC_P(S)$ preserves the set of orientations 
$\prod_{\kC(S)_{d(S)-2}} \{\pm 1\}$. In this way, we get a natural representation:
\[\Aut(\kC_P(S))\to\prod_{\kC(S)_{d(S)-2}} \{\pm 1\}.\]
We denote by $\Aut(\kC_P(S))^+$ its kernel and call it the subgroup of \emph{orientation preserving} automorphisms.
\end{remark}

\begin{lemma}\label{basiclemma}With the above notations, there is a natural isomorphism:
\begin{equation}\label{completepantsrigidity2}
\kG(S)/Z(\kG(S))\cong\Aut(\kC_P(S))^+.
\end{equation}
\end{lemma}

The next sections will be devoted to the proof of this lemma.

\subsection{From automorphisms of $\kC_P(S)$ to towers of holomorphic maps of Fulton curves}\label{fromauto}
Let $\{\G^\l\}_{\l\in\Lambda}$ be the set of normal levels satisfying the hypotheses of Proposition~4.3 in \cite{[B4]} 
(and so in particular the hypotheses of Proposition~\ref{conformalstructure}). This is a filtering subsystem of
the inverse system of all levels of $\G(S)$ and therefore we have that $\kC_P(S)=\varprojlim_{\l\in\L}C^\l_P(S)$. Moreover, 
by Proposition~4.3 in \cite{[B4]}, for all $\l\in\L$, the irreducible components of the Fulton curve $\cF^\l(S)$ are smooth.
As in the proof of Proposition~\ref{completepantsrigidity1}, an automorphism $\phi\in\Aut^+(\kC_P(S))$ is then determined 
by an inverse system of maps $\phi_{\mu\l}\co C_P^\mu(S)\to C_P^\l(S)$ commuting with the natural projections. We have:

\begin{lemma}\label{holmaps}For $\phi\in\Aut^+(\kC_P(S))$, let $\{\phi_{\mu\l}\}_{\l\in\L}$ be an inverse system of maps which determines it.
Then, every map $\phi_{\mu\l}\co C_P^\mu(S)\to C_P^\l(S)$ is induced by an holomorphic map $\Phi_{\mu\l}\co\cF^\mu(S)\to\cF^\l(S)$. 
Moreover, the maps $\Phi_{\mu\l}$ commute with the natural projections to the Fulton curve $\cF(S)\subset\bM(S)$.
\end{lemma}

\begin{proof}
Choose some arbitrary profinite Farey subgraph $\wh{F}_\s$. 
By Theorem~\ref{autinjection}, an automorphism of $\kC_P(S)$ acts on the curve complex $\kC(S)$ preserving the topological type of simplices.
Since these are the orbits by the action of $\kG(S)$, by composing $\phi$ with the automorphism induced by a suitable element of $\kG(S)$, we can assume that $\phi$ fixes $\wh{F}_\s$. 

For $\l\in\L$, let us denote by $F^\l_\s$ the image of $\wh{F}_\s$ in $C_P^\l(S)$ by the natural projection $\kC_P(S)\to C_P^\l(S)$. 
We then have $\phi_{\mu\l}(F^\mu_\s)=F^\l_\s$ and,
from Proposition~\ref{completepantsrigidity1}, it follows that the restriction $\phi_{\mu\l}\co F^\mu_\s\to F^\l_\s$ is induced by a holomorphic map 
between the corresponding irreducible components of the Fulton curves $\cF^\mu(S)$ and $\cF^\l(S)$. Since $\s\in\kC(S)_{d(S)-2}$ was chosen in an 
arbitrary way, this is then true for all Farey subgraphs of $C^\mu_P(S)$. Moreover, since $\cF^\mu(S)$ has only multicross singularities, 
these holomorphic maps assemble together and determine a holomorphic map $\Phi_{\mu\l}\co\cF^\mu(S)\to\cF^\l(S)$ which induces the
map of graphs $\phi_{\mu\l}\co C_P^\mu(S)\to C_P^\l(S)$. Note that the restriction $|\phi_{\mu\l}|\co |F^\mu_\s|\to|F^\l_\s|$ commutes 
with the natural maps to $\cF(S)$. Arguing as above, we conclude that the induced holomorphic map $\Phi_{\mu\l}\co\cF^\mu(S)\to\cF^\l(S)$
commutes with the projections to $\cF(S)$.
\end{proof}

As in the proof of Proposition~\ref{completepantsrigidity1}, in order to complete the proof of Lemma~\ref{basiclemma}, it is enough to show that:

\begin{lemma}\label{mainlemma}With the notations of Proposition~\ref{holmaps}, 
the holomorphic maps $\Phi_{\mu\l}\co\cF^\mu(S)\to\cF^\l(S)$ are obtained by composing an automorphism $\Phi_\mu$
of the Fulton curve $\cF^\mu(S)$ over $\cF(S)$ with the natural projection $\pi_{\mu\l}\co\cF^\mu(S)\to\cF^\l(S)$; in other words, for all $\l\in\L$,
we have that $\Phi_{\mu\l}=\pi_{\mu\l}\circ\Phi_\mu$, for some $\Phi_\mu\in \G(S)/\G^\mu$. 
\end{lemma}

Indeed, Lemma~\ref{mainlemma} implies that the given $\phi\in\Aut^+(\kC_P(S))$ is induced by the action of the element:
\[\Phi:=\varprojlim_{\mu}\Phi_\mu\in\varprojlim_{\mu}\G(S)/\G^\mu=\kG(S)\]

\subsection{Proof of Lemma~\ref{mainlemma}}
Let $\cH(S)$ be the inverse image of the Fulton curve $\cF(S)\subset\bM(S)$ in the Bers bordification $\ol{\cT}(S)$ of the Teichm\"uller space $\cT(S)$
and let $\Delta_\cH$ be the triangulation of $\cH(S)$ which on each irreducible component restricts to the standard Farey triangulation.
Let us observe that, in order to prove Lemma~\ref{mainlemma}, it would be enough to show that the holomorphic maps 
$\Phi_{\mu\l}\co\cF^\mu(S)\to\cF^\l(S)$ lift to holomorphic maps $\wt{\Phi}_{\mu\l}\co\cH(S)\to\cH(S)$. Indeed, such a lift $\wt{\Phi}_{\mu\l}$
would induce an automorphism of the triangulation $\Delta_\cH$ and hence of its $1$-skeleton $|C_P(S)|$. By Margalit rigidity theorem 
(cf.\ Theorem~\ref{rigiditypants}),  this implies that $\wt{\Phi}_{\mu\l}$ is induced by an element $\wt{\Phi}_{\mu}\in\Mod(S)$ which, since $\wt{\Phi}_{\mu\l}$
is orientation preserving, actually lies in $\G(S)$. The image $\Phi_\mu$ of $\wt{\Phi}_{\mu\l}$ in $\G(S)/\G^\mu$ has then the properties stated in 
Lemma~\ref{mainlemma}. 

In order to construct the above lift, we need to pass through the Harvey bordification $\wT(S)$ of the Teichm\"uller space $\cT(S)$ introduced in 
Section~\ref{curvecomplexint}. Let then $\wH(S)$ (resp.\ $\wF^\l(S)$) be the inverse image of the Farey curve $\cF(S)$ in $\wT(S)$
(resp.\ in $\wM^\l(S)$) via the natural map $\wT(S)\to\bM(S)$ (resp.\ $\wM^\l(S)\to\bM(S)$). Note that the natural maps
$\wH(S)\to\cH(S)$ and $\wF^\l(S)\to\cF^\l(S)$ are, outside the singular loci of $\cH(S)$ and $\cF^\l(S)$, fibrations in $(d(S)-1)$-dimensional tori. 

By Proposition~\ref{holmaps}, the maps $\Phi_{\mu\l}\co\cF^\mu(S)\to\cF^\l(S)$ commute with the natural projections of $\cF^\mu(S)$ and $\cF^\l(S)$
to $\cF(S)$. Since, by construction, there are canonical isomorphisms $\wF^\mu(S)\cong\cF^\mu(S)\times_{\cF(S)}\wF(S)$ and 
$\wF^\l(S)\cong\cF^\l(S)\times_{\cF(S)}\wF(S)$, these holomorphic maps lift to real-analytic maps 
\[\wh{\Phi}_{\mu\l}\co\wF^\mu(S)\to\wF^\l(S),\]
which commute with the natural projections of $\wF^\mu(S)$ and $\wF^\l(S)$ to $\wF(S)$. Since these projections are \'etale maps, it follows that
also the maps $\wh{\Phi}_{\mu\l}\co\wF^\mu(S)\to\wF^\l(S)$ are \'etale.

Let us consider the short exact sequence 
\begin{equation}\label{shortexact}
1\ra\pi_1(\wH(S))\ra\pi_1(\wF^\mu(S))\ra\G^\mu\ra 1
\end{equation}
associated to the natural (\'etale) projection map $\wH(S)\to\wF^\mu(S)$. 

In order to prove that the maps $\wh{\Phi}_{\mu\l}\co\wF^\mu(S)\to\wF^\l(S)$ lift to real-analytic automorphisms 
\begin{equation}\label{lift}
\wh{\Phi}_{\mu\l}^\sim\co\wH(S)\to\wH(S),
\end{equation}
which then are determined up to the action of an element of $\G^\mu$,
we have to show that, at the level of fundamental groups, $\wh{\Phi}_{\mu\l}$ induces an inclusion:

\begin{equation}\label{fundinclusion}
\wh{\Phi}_{\mu\l\ast}(\pi_1(\wH(S)))\subseteq \pi_1(\wH(S)).
\end{equation}

For $\mu,\nu\in\L$, we write $\nu\leq\mu$ if the level $\G^\nu$ is contained in $\G^\mu$.
By the short exact sequence~(\ref{shortexact}) and the fact that the congruence topology on $\G^\mu$ is separated, we then have that 
$\pi_1(\wH(S))=\bigcap_{\nu\leq\mu}\pi_1(\wF^{\nu}(S))$.
Moreover, as we observed above, the map $\wh{\Phi}_{\mu\l}$ is \'etale and so the induced map $\wh{\Phi}_{\mu\l\ast}$ is injective. Therefore, we have

\[\wh{\Phi}_{\mu\l\ast}(\pi_1(\wH(S)))=\bigcap_{\nu\leq\mu}\wh{\Phi}_{\mu\l\ast}(\pi_1(\wF^{\nu}(S)))\]
and, in order to prove the inclusion~(\ref{fundinclusion}), we are now reduced to show that there holds

\begin{equation}\label{reduction1}
\bigcap_{\nu\leq\mu}\wh{\Phi}_{\mu\l\ast}(\pi_1(\wF^{\nu}(S)))\subseteq\pi_1(\wH(S)).
\end{equation}

It is actually enough to prove the inclusion~(\ref{reduction1}) for some subset of levels $\{\G^\nu\}_{\nu\in\Nu}$ such that $\G^\nu\leq\G^\mu$ 
for all $\nu\in\Nu$. We can construct this subset in the following way. Let $\L'=\{\xi\in\L|\,\xi\leq\l\}$. Then, for every $\xi\in\L'$, there is $\nu\in\L'$
such that we have a map $\wh{\Phi}_{\nu\xi}\co\wF^\nu(S)\to\wF^\xi(S)$ in the inverse system $\{\wh{\Phi}_{\mu\l}\}_{\l\in\L}$. This map fits in the
commutative diagram

\[ \begin{array}{rcr}
\wF^{\nu}(S) & \stackrel{\widehat\Phi_{\nu\xi}}{\longrightarrow} & \wF^{\xi}(S)\\
\downarrow^{\pi_{\nu\mu}} & &  \downarrow^{\pi_{\xi\l}}\\
\wF^{\mu}(S) & \stackrel{\widehat\Phi_{\mu\l}}{\longrightarrow} & \wF^{\l}(S)\\
\end{array}\]
and so we have $\widehat\Phi_{\mu\l\ast}(\pi_1(\wF^{\nu}(S)))=\widehat\Phi_{\nu\xi\ast}(\pi_1(\wF^{\nu}(S)))\subseteq\pi_1(\wF^{\xi}(S))$. 
Let $\{\G^\nu\}_{\nu\in\Nu}$ be the set of levels thus obtained. Since $\{\G^\xi\}_{\xi\in\L'}$ is a filtering subsystem of the set of all congruence levels,
we have:
\[\bigcap_{\nu\in\Nu}\wh{\Phi}_{\mu\l\ast}(\pi_1(\wF^{\nu}(S)))\subseteq\bigcap_{\xi\in\L'}\pi_1(\wF^{\xi}(S))=\pi_1(\wH(S)).\]

This proves the inclusion~(\ref{reduction1}) and then~(\ref{fundinclusion}) and so the existence of the lift~(\ref{lift}).

Let us denote by $r\co\wH(S)\to\cH(S)$ the natural retraction and by $p\co\cH(S)\to\cF(S)$ the natural projection. 
By construction, the automorphism $\wh{\Phi}_{\mu\l}^\sim\co\wH(S)\to\wH(S)$ 
commutes with the projection $p\circ r$ to the Fulton curve $\cF(S)$. This implies that, for every point $x\in\cH(S)$, we have 
$\wh{\Phi}_{\mu\l}^\sim(r^{-1}(x))=r^{-1}(y)$ for some other point $y\in\cH(S)$ such that $p(y)=p(x)$. Therefore, the automorphism 
$\wh{\Phi}_{\mu\l}^\sim$ is compatible with the retraction map to $\cH(S)$ and so induces a real analytic automorphism
$\wt{\Phi}_{\mu\l}\co\cH(S)\to\cH(S)$ which is indeed a lifting of the map $\Phi_{\mu\l}\co\cF^\mu(S)\to\cF^\l(S)$ with which we started out.
In particular, $\wt{\Phi}_{\mu\l}$ is holomorphic and respects the Farey triangulation of $\cH(S)$. Thus, the lemma follows.

\subsection{Conclusion of the proof of Theorem~\ref{completepantsrigidity}}
By Lemma~\ref{basiclemma}, in order to complete the proof of Theorem~\ref{completepantsrigidity}, it is enough to show that 
the orientation representation $\Aut(\kC_P(S))\to\prod_{\kC(S)_{d(S)-2}} \{\pm 1\}$ of Remark~\ref{compatibility} is constant on 
the $\kG(S)$-orbits of profinite Farey subgraphs of $\kC_P(S)$, that is to say, if an automorphism of $\kC_P(S)$ preserves the
orientation of a profinite Farey subgraph $\wh{F}_\s$, for $\s\in\kC(S)_{d(S)-2}$, then it preserves the orientation of all profinite Farey subgraphs 
$\wh{F}_{f(\s)}$, for $f\in\kG(S)$. But this immediately follows from the fact that, by Lemma~\ref{basiclemma}, the image of $\kG(S)$ in
$\Aut(\kC_P(S))$ is a normal subgroup. Hence, for $\phi\in\Aut(\kC_P(S))$ as above, we have the identity:
\[\phi|_{\wh{F}_{f(\s)}}=\phi\circ\bar f|_{\wh{F}_{\s}}=\bar f'\circ\phi|_{\wh{F}_{\s}},\]
where, for $f,f'\in\kG(S)$, we denote by $\bar f,\bar f'$ their images in $\Aut(\kC_P(S))$. This identity clearly implies the above claim since an element
of $\Aut(\kC_P(S))$ in the image of $\kG(S)$ (by definition) preserves the orientations of all profinite Farey subgraphs of $\kC_P(S)$.

\section{Anabelian properties of moduli stacks of curves}\label{sect:anabelian}
In this section, we present some arithmetic consequences of Theorem~\ref{completepantsrigidity}.
According to Grothendieck's anabelian philosophy (cf.\ \cite{[G]}), moduli stacks of curves should be prime examples of \emph{anabelian} objects,
that is to say they should be reconstructible from their \'etale fundamental groups. Theorem~\ref{mainanabelian} will partially vindicate this conjecture,
with some limitations explained below. In this section, the base field $\k$ is a field over which the anabelian conjecture for hyperbolic curves is valid.
By Mochizuki's theorem in \cite{[Mo]}, we can take $\k$ to be a sub-$p$-adic field, 
that is to say a subfield of a finitely generated extension of $\Q_p$ for some prime $p$ (e.g.\ a number field). 

\subsection{Isomorphisms of stacks}
As recalled in the introduction, the $\hom$ functor on the category of stacks takes values in groupoids. More precisely, $\Hom(X,Y)$,
for each pair of stacks $X$ and $Y$, is a category whose objects are the $1$-morphisms $f\co X\to Y$
and whose (invertible) $2$-morphisms, between two $1$-morphisms $f,g\co X\to Y$, are denoted by a double arrow $\alpha\co f\Rightarrow g$ and
called $2$-morphisms. We denoted the objects of the category $\Hom(X,Y)$ by $1$-$\Hom(X,Y)$ and its morphisms by $2$-$\Hom(X,Y)$.
If $X$ and $Y$ are defined over $\k$, we then denote by $\Hom_\k(X,Y)$ the subcategory of $\Hom(X,Y)$ formed by $1$-morphisms and
$2$-morphisms which are defined over $\k$.

The group of {\em generic automorphisms} of a DM stack $X$ is the local system $\cA$ formed by the maximal subgroups of automorphisms shared
locally by points of $X$ (cf.\ Section~5 in \cite{Ro}, for the precise definition). In the case of the moduli stack of curves $\cM(S)$, this is a nontrivial but
\emph{constant} local system for $S=S_2, S_{1,1}$ and $S_{1,2}$.  It is nontrivial and \emph{nonconstant} for the case $S=S_{0,4}$.

Let us denote by $X\!\!\fatslash \cA$ the DM stack obtained from $X$ by erasing its group of generic automorphisms $\cA$. 
This is called the {\em rigidification of $X$ along $\cA$} (cf.\ Definition~5.1.9 in \cite{ACV} and Section~5 in \cite{Ro}).

Let us assume, for simplicity, that $X$ is connected.
Since $X$ is an $\cA$-gerbe over $X\!\!\fatslash \cA$, it follows that there is a short exact sequence (cf.\ Section~4 in \cite{Noohi}):
\begin{equation}\label{gerbex}
1\to A\to\pi_1^\mathrm{et}(X)\to\pi_1^\mathrm{et}(X\!\!\fatslash\cA)\to 1,
\end{equation}
where $A$ is the fiber of $\cA$ at the base point of $\pi_1^\mathrm{et}(X)$. The associated action of $\pi_1^\mathrm{et}(X\!\!\fatslash\cA)$ on $A$ 
defines a local system on $X\!\!\fatslash\cA$ which we also denote by $\cA$. We then denote by $Z(\cA)$ the local subsystem formed by the
centers of the fibers of $\cA$.

There is a natural outer action of $1\mbox{-}\Aut(X)$ on $\pi_1^\mathrm{et}(X)$ which preserves $A$. 
Let $1\mbox{-}\underline{\Aut}(X)$ be the kernel of the natural homomorphism $1\mbox{-}\Aut(X)\to1\mbox{-}\Aut(X\!\!\fatslash \cA)$. 
Note that, since $1\mbox{-}\underline{\Aut}(X)$ preserves points, it acts on $\pi_1^\mathrm{et}(X)$ (and so on $A$) by genuine automorphisms. 

\begin{lemma}\label{1-2-aut}With the above hypothesis and notations, we have:
\begin{enumerate}
\item The group of $2$-automorphisms $\Aut(\mathrm{id}_X)$ of the identity map $\mathrm{id}_X\co X\to X$ identifies with 
$H^0(X\!\!\fatslash\cA, Z(\cA))$. In particular, $2$-$\Aut(X)$ is a trivial $H^0(X\!\!\fatslash \cA, Z(\cA))$-torsor over $1$-$\Aut(X)$.
\item There is an exact sequence: 
\[1\to H^1(X\!\!\fatslash \cA, Z(\cA))\to 1\mbox{-}\underline{\Aut}(X)\to\Aut(A).\]
\end{enumerate}
\end{lemma}

\begin{proof}(i): The question is local. So we can reduce to the case of the classifying stack $\operatorname{B}\!A$ for which the claim is an easy 
consequence of the various definitions involved.
\smallskip

\noindent
(ii): A $1$-isomorphism of $X$, which induces the identity isomorphism on both $X\!\!\fatslash\cA$ and $A$, restricts to the trivial automorphism
on a local trivializing chart for the gerbe $X\to X\!\!\fatslash\cA$ and then define a $1$-cocycle for the \v{C}ech cohomology of $X\!\!\fatslash\cA$
with coefficients in $Z(\cA)$ (the locally constant sheaf of $2$-automorphisms of $X$ over $X\!\!\fatslash\cA$). This defines an element of 
$H^1(X\!\!\fatslash\cA, Z(\cA))$. The other direction is also clear.
\end{proof}

\begin{remark}\label{2-iso}Note that $H^0(X\!\!\fatslash \cA, Z(\cA))$ identifies with the invariants $(Z(A))^{\pi_1^\mathrm{et}(X)}$
for the natural action of $\pi_1^\mathrm{et}(X)$ on $Z(A)$ induced by conjugation (cf.\ the sequence~(\ref{gerbex})). 
It follows that given another DM stack $Y$, the set $2$-$\Isom(Y,X)$, if nonempty, is a trivial $(Z(A))^{\pi_1^\mathrm{et}(X)}$-torsor over $1$-$\Isom(Y,X)$.
\end{remark}

\subsection{The anabelian conjecture for moduli stacks of curves}
Let $\cM^\l\ra\cM(S)$ be a (possibly trivial) level structure defined over a sub-$p$-adic field $\k$. 
By Grothendieck's theory of the \'etale fundamental group and its generalization to DM stacks (cf.\ \cite{Noohi} and \cite{Oda}, for instance), 
for a geometric base point $\ol\xi$ on $\cM^\l_\k$, the structural morphism $\cM^\l_\k\ra\Spec\k$ induces the short exact sequence 
of \'etale fundamental groups:
\begin{equation}\label{fundamentalseq}
1\to\pi_1^\mathrm{et}(\cM^\l_{\ol{\k}},\ol{\xi})\to\pi_1^\mathrm{et}(\cM^\l_\k,\ol{\xi})\to G_\k\ra 1.
\end{equation}

The left term  $ \pi_1^\mathrm{et}(\cM^\l_{\ol{\k}},\ol{\xi})$ of this short exact sequence is the \emph{geometric \'etale fundamental group of} 
$\cM^\l_\k$. After fixing an embedding $\ol{\k}\hookra\C$, this is naturally isomorphic to the profinite completion of the topological fundamental group  
of the complex analytic stack associated to the complex DM stack $\cM^\l_\C$ with base point $\ol{\xi}$ and then $\pi_1^\mathrm{et}(\cM^\l_{\ol{\k}},\ol{\xi})$ 
can be identified with the profinite completion of the level $\G^\l$ of the mapping class group $\G(S)$. 
The rightmost map $\pi_1^\mathrm{et}(\cM^\l_\k,\ol{\xi})\to G_\k$, in the above short exact 
sequence~(\ref{fundamentalseq}), is called the \emph{augmentation map}.

For another (possibly trivial) level structure $\cM^\mu\ra\cM(S')$, also defined over $\k$, and a geometric base point $\ol\xi'$ on $\cM^\mu_\k$, 
the set of isomorphisms $\pi_1^\mathrm{et}(\cM^\l_\k,\ol{\xi})\to \pi_1^\mathrm{et}(\cM^\mu_\k,\ol{\xi}')$ which are compatible 
with  the augmentation maps to the Galois group $G_\k$ is denoted 
\[\Isom_{G_\k}(\pi_1^\mathrm{et}(\cM^\l_\k,\ol{\xi}),\pi_1^\mathrm{et}(\cM^\mu_\k,\ol{\xi}')).\] 
Then, we denote by
\[ \Isom_{G_\k}(\pi_1^\mathrm{et}(\cM^\l_\k,\ol{\xi}),\pi_1^\mathrm{et}(\cM^\mu_\k,\ol{\xi}'))^\mathrm{out}\] 
the set of orbits of $\Isom_{G_\k}(\pi_1^\mathrm{et}(\cM^\l_\k,\ol{\xi}),\pi_1^\mathrm{et}(\cM^\mu_\k,\ol{\xi}'))$ 
by the inner action of the fundamental group $\pi_1^\mathrm{et}(\cM^\mu_{\ol{\k}},\ol{\xi}')$. 
The latter set is independent of the choice of base points. 

There is an alternative description of the above sets which only involves the geometric part of the \'etale fundamental group. The short exact
sequences~(\ref{fundamentalseq}) for the levels $\l$ and $\mu$ define representations $G_\k\to\Out(\pi_1^\mathrm{et}(\cM^\l_{\ol{\k}},\ol{\xi}))$ 
and $G_\k\to\Out(\pi_1^\mathrm{et}(\cM^\mu_{\ol{\k}},\ol{\xi}'))$. Then, from the group theoretic Corollary~1.5.7 in \cite{[N]}, 
it follows that there is a canonical bijection
\[\Isom_{G_\k}(\pi_1^\mathrm{et}(\cM^\l_\k,\ol{\xi}),\pi_1^\mathrm{et}(\cM^\mu_\k,\ol{\xi}'))
\simeq\Isom_{G_\k}(\pi_1^\mathrm{et}(\cM^\l_{\ol{\k}},\ol{\xi}),\pi_1^\mathrm{et}(\cM^\mu_{\ol{\k}},\ol{\xi}')),\]
where $\Isom_{G_\k}(\pi_1^\mathrm{et}(\cM^\l_{\ol{\k}},\ol{\xi}),\pi_1^\mathrm{et}(\cM^\mu_{\ol{\k}},\ol{\xi}'))$ denotes the set 
of isomorphisms which are $G_\k$-equivariant modulo inner automorphisms. Taking the respective sets of orbits 
by the inner action of the fundamental group $\pi_1^\mathrm{et}(\cM^\mu_{\ol{\k}},\ol{\xi}')$, we then get a canonical bijection:
\begin{equation}\label{geometricformulation}\Isom_{G_\k}(\pi_1^\mathrm{et}(\cM^\l_\k,\ol{\xi}),\pi_1^\mathrm{et}(\cM^\mu_\k,\ol{\xi}'))^\mathrm{out}\simeq
\Isom_{G_\k}(\pi_1^\mathrm{et}(\cM^\l_{\ol{\k}},\ol{\xi}),\pi_1^\mathrm{et}(\cM^\mu_{\ol{\k}},\ol{\xi}'))^\mathrm{out}.
\end{equation}

\begin{definition}\label{genericaut}
For a level structure $\cM^\l\ra\cM(S)$, let us denote by ${\mathcal Z}^\l$ the group of generic automorphisms of the DM stack $\cM^\l$ (for $\cM^\l=\cM(S)$, 
we simply let ${\mathcal Z}^\l={\mathcal Z}$). 
\end{definition}

Note that, for $\G^\l$ a level of $\G(S)$ and $S\neq S_{0,4}$, the group of generic automorphisms ${\mathcal Z}^\l$ of the stack $\cM^\l$ 
is a constant local system whose fibers identify with the center of $\G^\l$. 
For $S= S_{0,4}$, the group ${\mathcal Z}^\l$ of generic automorphisms of the stack $\cM^\l$ is, in general, a nonconstant local system 
whose fibers are isomorphic to the intersection $K^\l:=\G^\l\cap K_4$, where $K_4$ is the Klein subgroup of $\G(S_{0,4})$.

\begin{definition}\label{centerMCG}For $\G^\l$ a level of $\G(S)$ and $S\neq S_{0,4}$, we denote by $Z^\l$ the center of $\G^\l$. For a
congruence level, by Theorem~\ref{slim}, we can identify $Z^\l$ also with the center of $\kG^\l$. 
Since, in this case, ${\mathcal Z}^\l$ is a constant local system, we simply also denote ${\mathcal Z}^\l$ by $Z^\l$. 

For $S= S_{0,4}$, we instead put $Z^\l=\G^\l\cap K_4$, where $K_4$ is the Klein subgroup of $\G(S_{0,4})$. 
By Theorem~\ref{slim}, we can identify $Z^\l$ with the maximal normal finite subgroup of $\kG^\l=\hG^\l$.

For $\G^\l=\G(S)$, we simply let $Z^\l=Z$.
\end{definition}

The fibers of the groups of generic automorphisms of $\cM^\l$ and $\cM^\mu$ over the base points $\ol{\xi}$ and $\ol{\xi}'$ can, in all cases, 
be identified with normal (abelian) subgroups $Z^\l$ and $Z^\mu$ of $\pi_1^\mathrm{et}(\cM^\l_\k,\ol{\xi})$ and 
$\pi_1^\mathrm{et}(\cM^\mu_\k,\ol{\xi}')$, respectively. 

Then, by Remark~\ref{2-iso}, if the set $1$-$\Isom_{G_\k}(\pi_1^\mathrm{et}(\cM^\l_\k,\ol{\xi}),\pi_1^\mathrm{et}(\cM^\mu_\k,\ol{\xi}'))^\mathrm{out}$ 
is nonempty, the set $2$-$\Isom_{G_\k}(\pi_1^\mathrm{et}(\cM^\l_\k,\ol{\xi}),\pi_1^\mathrm{et}(\cM^\mu_\k,\ol{\xi}'))^\mathrm{out}$ has a natural structure 
of trivial $(Z^\mu)^{\pi_1^\mathrm{et}(\cM^\mu_\k,\ol{\xi}')}$-torsor 
over the set $1$-$\Isom_{G_\k}(\pi_1^\mathrm{et}(\cM^\l_\k,\ol{\xi}),\pi_1^\mathrm{et}(\cM^\mu_\k,\ol{\xi}'))^\mathrm{out}$.

We formulate an {\em anabelian conjecture for moduli stacks of curves with level structure} as the assertion that the groupoid
$\Isom_\k(\cM^\l_\k,\cM^\mu_\k)$ can be recovered from the augmented arithmetic fundamental groups of $\cM^\l_\k$ and $\cM^\mu_\k$
in the following way:

\begin{conjecture}\label{anabelianconjmod}With the above notations, we have:
\begin{enumerate}
\item There is a natural bijection:
\[1\mbox{-}\Isom_\k(\cM^\l_\k,\cM^\mu_\k)\stackrel{\sim}{\to}
\Isom_{G_\k}(\pi_1^\mathrm{et}(\cM^\l_\k,\ol{\xi}),\pi_1^\mathrm{et}(\cM^\mu_\k,\ol{\xi}'))^\mathrm{out}.\]
\item The fiber $Z^\mu$ over $\ol{\xi}'$ of the group of generic automorphisms of $\cM^\mu_\k$ identifies with the maximal normal finite subgroup
of the group $\pi_1^\mathrm{et}(\cM^\mu_{\ol{\k}},\ol{\xi}')$.
\end{enumerate}
\end{conjecture}
 
\subsection{The procongruence anabelian conjecture for moduli stacks}
We were not able to prove the anabelian conjecture as stated above and we need to modify it in two ways. 
In the first place, we have to replace the profinite completion with the procongruence completion and then to impose a 
further condition on automorphisms (cf.\ Definition~\ref{starconditioniso}). 

In this setting, we proceed as follows. Let us recall that the moduli stack $\cM(S)$ carries a universal $n$-punctured, genus $g$ smooth curve
$\cC(S)\to\cM(S)$, where $n=n(S)$ and $g=g(S)$. Let then $\cC^\l\to\cM^\l$ be the pull-back of this curve via the natural morphism $\cM^\l\to\cM(S)$ 
and let $\cC^\l_{\ol\xi}$ be its fiber over the geometric point $\ol\xi$. There is a short exact sequence of algebraic fundamental groups:
\[1\to\pi_1^\mathrm{et}(\cC^\l_{\ol\xi},\tilde{\xi})\to\pi_1^\mathrm{et}(\cC^\l,\tilde{\xi})\to\pi_1^\mathrm{et}(\cM^\l_\k,\ol{\xi})\ra 1.\]
The associated outer representation
\[\rho^\l\co\pi_1^\mathrm{et}(\cM^\l_\k,\ol{\xi})\to\Out(\pi_1^\mathrm{et}(\cC^\l_{\ol\xi},\tilde{\xi}))\]
is the {\it universal \'etale monodromy representation} associated to the level structure $\cM^\l_\k$. 

We then set $\check{\pi}_1(\cM^\l_\k,\ol{\xi})=\rho^\l(\pi_1^\mathrm{et}(\cM^\l_\k,\ol{\xi}))$
and $\check{\pi}_1(\cM^\l_{\ol{\k}},\ol{\xi})=\rho^\l(\pi_1^\mathrm{et}(\cM^\l_{\ol{\k}},\ol{\xi}))$, respectively. 
Note that $\check{\pi}_1(\cM^\l_{\ol{\k}},\ol{\xi})$ can be identified with the congruence completion $\kG^\l$ of the level $\G^\l$.

Hoshi and Mochizuki in \cite{HM} showed that the kernel of $\rho^\l$ identifies with the congruence kernel, i.e.\ the kernel of the natural epimorphism 
$\hG(S)\to\kG(S)$ (see also Corollary~7.10 in \cite{[B3]}). Therefore, the short exact sequence~(\ref{fundamentalseq}) descends to a
short exact sequence:
\begin{equation}\label{fundamentalseq2}
1\to\check{\pi}_1(\cM^\l_{\ol{\k}},\ol{\xi})\to\check{\pi}_1(\cM^\l_\k,\ol{\xi})\to G_\k\ra 1.
\end{equation}
In particular, there is an \emph{augmentation map} $\check{\pi}_1(\cM^\l_\k,\ol{\xi})\to G_\k$.  

As above, the set of isomorphisms $\check{\pi}_1(\cM^\l_\k,\ol{\xi})\to \check{\pi}_1(\cM^\mu_\k,\ol{\xi}')$,
which are compatible with the augmentation maps to the absolute Galois group $G_\k$, is then denoted by
\[ \Isom_{G_\k}(\check{\pi}_1(\cM^\l_\k,\ol{\xi}),\check{\pi}_1(\cM^\mu_\k,\ol{\xi}'))\]
and we let
\[ \Isom_{G_\k}(\check{\pi}_1(\cM^\l_\k,\ol{\xi}),\check{\pi}_1(\cM^\mu_\k,\ol{\xi}'))^\mathrm{out}\] 
be the set of orbits of this set by the inner action of the group $\check{\pi}_1(\cM^\mu_{\ol{\k}},\ol{\xi}')$.

Let $\kG^\l$ and $\kG^\mu$ be open subgroups of $\kG(S)$ and $\kG(S')$, respectively, such that the associated level structures
are both defined over $\k$. Then, the \emph{procongruence anabelian conjecture for congruence level structures} states that:

\begin{conjecture}\label{anabelianconjcong}With the above notations, we have:
\begin{enumerate}
\item There is a natural bijection:
\[1\mbox{-}\Isom_\k(\cM^\l_\k,\cM^\mu_\k)\stackrel{\sim}{\to}\Isom_{G_\k}(\check{\pi}_1(\cM^\l_\k,\ol{\xi}),\check{\pi}_1(\cM^\mu_\k,\ol{\xi}'))^\mathrm{out}.\]
\item The fiber $Z^\mu$ over $\ol{\xi}'$ of the group of generic automorphisms of $\cM^\mu_\k$ identifies with the maximal normal finite subgroup
of the group $\check{\pi}_1(\cM^\mu_{\ol{\k}},\ol{\xi}')$.
\end{enumerate}
\end{conjecture}

The isomorphism in the first item of Conjecture~\ref{anabelianconjcong} can be formulated in terms of procongruence mapping class groups
as follows. The short exact sequence~(\ref{fundamentalseq}), applied to the levels $\l$ and $\mu$, defines representations 
$G_\k\to \Out(\check{\pi}_1(\cM^\l_{\ol{\k}},\ol{\xi}))$ and $G_\k\to \Out(\check{\pi}_1(\cM^\mu_{\ol{\k}},\ol{\xi}'))$. We then define
\[ \Isom_{G_\k}(\check{\pi}_1(\cM^\l_{\ol{\k}},\ol{\xi}),\check{\pi}_1(\cM^\mu_{\ol{\k}},\ol{\xi}'))\] 
to be the set of isomorphisms which are $G_\k$-equivariant modulo inner automorphisms and 
\[\Isom_{G_\k}(\check{\pi}_1(\cM^\l_{\ol{\k}},\ol{\xi}),\check{\pi}_1(\cM^\mu_{\ol{\k}},\ol{\xi}'))^\mathrm{out}\] 
to be its set of orbits by the inner action of $\check{\pi}_1(\cM^\mu_{\ol{\k}},\ol{\xi}')$.

By the group theoretic Corollary~1.5.7 in \cite{[N]}, there are natural bijections
\[\Isom_{G_\k}(\check{\pi}_1(\cM^\l_\k,\ol{\xi}),\check{\pi}_1(\cM^\mu_\k,\ol{\xi}'))\simeq
\Isom_{G_\k}(\check{\pi}_1(\cM^\l_{\ol{\k}},\ol{\xi}),\check{\pi}_1(\cM^\mu_{\ol{\k}},\ol{\xi}'))\]
and
\[\Isom_{G_\k}(\check{\pi}_1(\cM^\l_\k,\ol{\xi}),\check{\pi}_1(\cM^\mu_\k,\ol{\xi}'))^\mathrm{out}\simeq
\Isom_{G_\k}(\check{\pi}_1(\cM^\l_{\ol{\k}},\ol{\xi}),\check{\pi}_1(\cM^\mu_{\ol{\k}},\ol{\xi}'))^\mathrm{out}.\]

After identifying $\check{\pi}_1(\cM^\l_{\ol{\k}},\ol{\xi})$ with $\kG^\l$ and $\check{\pi}_1(\cM^\mu_{\ol{\k}},\ol{\xi}')$ 
with $\kG^\mu$, we get the natural bijection:
\begin{equation}\label{geometricformulation2}
\Isom_{G_\k}(\check{\pi}_1(\cM^\l_\k,\ol{\xi}),\check{\pi}_1(\cM^\mu_\k,\ol{\xi}'))^\mathrm{out}\simeq\Isom_{G_\k}(\kG^\l,\kG^\mu)^\mathrm{out}.
\end{equation}
So that we can now reformulate item (i) of Conjecture~\ref{anabelianconjcong} as the statement that there is a natural bijection:
\[1\mbox{-}\Isom_\k(\cM^\l_\k,\cM^\mu_\k)\stackrel{\sim}{\to}\Isom_{G_\k}(\kG^\l,\kG^\mu)^\mathrm{out}.\]

Conjecture~\ref{anabelianconjcong} is somewhat more treatable than Conjecture~\ref{anabelianconjmod}
thanks to the results in \cite{[B3]} and \cite{congtop} (cf.\ Section~\ref{section:isomorphism} and Section~\ref{section:prodehntwists}). 
For one thing, the second item of the conjecture follows from Theorem~\ref{slim}. However, in order to get some grip on the set 
$\Isom_{G_\k}(\kG^\l,\kG^\mu)^\mathrm{out}$, we need to relate this to the procongruence curve complexes we introduced
in Section~\ref{sect3} and studied thereafter. In other words, we need, for instance, to restrict to those automorphisms of a procongruence
mapping class group which induce an automorphism of the corresponding procongruence curve complex.
For this, we need one more definition:

\begin{definition}\label{starconditioniso}For $\kG^\l$ and $\kG^\mu$ open subgroups of $\kG(S)$ and $\kG(S')$, respectively, we
let $\Isom_{G_\k}^\ast(\kG^\l,\kG^\mu)$ be the set of isomorphisms which are $G_\k$-equivariant modulo  inner automorphisms  
of $\kG^\mu$ and map each subgroup of $\kG^\l$ in the set $\{\kG^\l_\g\}_{\g\in\hL(S)_0}$ onto a subgroup of $\kG^\mu$ in the set 
$\{\kG^\mu_\delta\}_{\delta\in\hL(S')_0}$. We then let $\Isom_{G_\k}^\ast(\kG^\l,\kG^\mu)^\mathrm{out}$ be its set of orbits for the 
action of $\kG^\mu$ by inner automorphisms.
\end{definition}  

In the topological case, as we already observed in Section~\ref{rigiditycurvecomplex}, all isomorphisms of mapping class groups
satisfy a version of the $\ast$-condition. This deep result depends upon a characterization of the so called pseudo-Anosov elements
in terms of their centralizers in the mapping class group. A similar result is not available for the procongruence mapping class group.
Therefore, even though we expect that, for $d(S)$ and $d(S')>1$, the $\ast$-condition, as formulated in Definition~\ref{starconditioniso}, 
is satisfied by all isomorphisms (even those who are not $G_\k$-equivariant modulo  inner automorphisms), at the current state of knowledge, 
we need to assume this as an hypothesis. We can now formulate our anabelian result:

\begin{theorem}\label{mainanabelian}Let $\kG^\l$ and $\kG^\mu$ be open subgroups of $\kG(S)$ and $\kG(S')$, respectively,
(where we assume that $S,S'\neq S_{0,4}$) and $\k$ a sub-$p$-adic field over which the associated level structures are both defined. 
We then have:
\begin{enumerate}
\item there is a natural bijection $1$-$\Isom_\k(\cM^\l_\k,\cM^\mu_\k)\stackrel{\sim}{\to}\Isom_{G_\k}^\ast(\kG^\l,\kG^\mu)^\mathrm{out}$;
\item the group of generic automorphisms $Z^\mu$ of $\cM_\k^\mu$ identifies with the center of $\kG^\mu$.
\end{enumerate}
\end{theorem}

\begin{remark}For $\kG^\l=\kG^\mu=\kG_{0,n}$ and then, by the genus $0$ case of the congruence subgroup property for mapping class groups, 
$\kG^\l=\kG^\mu=\hG_{0,n}$, a version of Theorem~\ref{mainanabelian} without the $\ast$-condition was proved by Nakamura and Tamagawa 
(cf.\ Theorem~C and Corollary~C in \cite{IN}). Since moduli spaces of $n$-pointed, genus $0$ curves over a field $\k$ are examples of 
\emph{strongly hyperbolic Artin neighbourhood} (resp.\ \emph{hyperbolic polycurves}) in the sense of Definition~6.1 in \cite{ScSt}
(resp.\ Definition~1.9 in \cite{Hoshi}), this is also a particular case of the much more general Corollary~1.6 in \cite{ScSt} over a finite 
extension of $\Q$ (resp.\ Theorem~B in \cite{Hoshi} over a sub-$p$-adic field). 
\end{remark}

\subsection{Transporters and $G_\k$-isomorphisms}Let us recall (cf.\ Definition~\ref{transporter}) that, for two subgroups 
$H_1$ and $H_2$ of a group $G$, the transporter $T_{G}(H_1,H_2)$ of $H_1$ onto $H_2$
is the set of elements $\phi\in\Inn G$ such that $\phi(H_1)=H_2$. 
The following result, of independent interest, is then the first step in the proof of Theorem~\ref{mainanabelian}: 

\begin{theorem}\label{mainanabelian2}Let $\kG^\l$ and $\kG^\mu$ be open subgroups of $\kG(S)$ and let
$\k$ be a sub-$p$-adic field such that all automorphisms of $\cM^\l_\k$ and $\cM^\mu_\k$ are defined over $\k$.
Then, there is a natural continuous $\kG^\mu$-equivariant bijection:
\[\Isom_{G_\k}^\ast(\kG^\l/ Z^\l,\kG^\mu/ Z^\mu)\stackrel{\sim}{\to}T_{\kG(S)/Z}(\kG^\l/ Z^\l,\kG^\mu/ Z^\mu),\]
where  $Z^\l$, $Z^\mu$ and $Z$ are as in Definition~\ref{centerMCG}.  
\end{theorem}

\begin{proof}Let us consider first the case $d(S)=1$ (modular dimension $1$), which is the one we can directly reduce to Mochizuki anabelian theorem. 
Since $\kG(S_{1,1})/\langle\iota\rangle$ identifies with $\kG(S_{0,4})/K_4$ and $\cM(S_{1,1})\!\!\!\fatslash\langle\iota\rangle$ with 
$\cM(S_{0,4})\!\!\!\fatslash K_4$, the case $S=S_{1,1}$ reduces to the case $S=S_{0,4}$ and so we can assume that $S=S_{0,4}$.

Let us observe first that the weight characterization of inertia subgroups of profinite surface groups 
by Nakamura (cf.\ Theorem~2.1.1 in \cite{[N]}) easily generalizes to the inertia subgroups of profinite Fuchsian groups like $\kG^\l/Z^\l=\hG^\l/Z^\l$ 
and $\kG^\mu/Z^\mu=\hG^\mu/Z^\mu$. Let us then observe that the sets $\{\kG^\l_\g/Z^\l\}_{\g\in\hL(S_{0,4})_0}$ and 
$\{\kG^\mu_\delta/Z^\mu\}_{\delta\in\hL(S_{0,4})_0}$ of Definition~\ref{starconditioniso} coincide with the sets of inertia subgroups
considered by Nakamura in Section~2.1 of \cite{[N]}. Therefore, by Theorem~2.1.1 in \cite{[N]} (or, more generally, its version for profinite 
Fuchsian groups), we have that:
\[\Isom_{G_\k}(\hG^\l/Z^\l,\hG^\mu/Z^\mu)=\Isom_{G_\k}^\ast(\hG^\l/Z^\l,\hG^\mu/Z^\mu).\]
 
By Remark~\ref{grouptheoretic}, for $f\in\Isom_{G_\k}(\hG^\l/Z^\l,\hG^\mu/Z^\mu)$, the assignment $\g\mapsto f_\ast(\g)$, where
$\g\in\hL(S_{0,4})_0$ and $\hG^\mu_{f_\ast(\g)}/Z^\mu:=f(\hG^\l_\g/Z^\l)$, then defines an isomorphism $f_\ast\co \hC(S_{0,4})\to\hC(S_{0,4})$.
So there is a natural continuous $\hG^\mu$-equivariant map 
\[\Isom_{G_\k}(\hG^\l/Z^\l,\hG^\mu/Z^\mu)\to\Aut(\hC(S_{0,4})),\]
compatible with the natural embedding of the transporter $T_{\wh{\PSL}_2(\Z)}(\hG^\l/Z^\l,\hG^\mu/Z^\mu)$ inside 
$\Isom_{G_\k}(\hG^\l/Z^\l,\hG^\mu/Z^\mu)$ and its natural action on $\hC(S_{0,4})$. 
 
From (i) of Theorem~\ref{faithfulness}, it follows that the above map is injective. Since $\hC(S_{0,4})$ is the vertex set of the profinite Farey graph 
$\hC_P(S_{0,4})$, in order to prove the case $d(S)=1$ of Theorem~\ref{mainanabelian2}, by Proposition~\ref{completepantsrigidity1}, 
we have to show that this action also preserves the edges of the profinite Farey graph $\wh{F}$ and its orientation  as defined in 
Section~\ref{section:orientations}. It is then enough to show that, for $f\in \Isom_{G_\k}(\hG^\l/Z^\l,\hG^\mu/Z^\mu)$ and a normal 
finite index subgroup $\hG^{\l'}$ of $\hG^\l/Z^\l$, which we assume is contained in $\hG_{0,4}\subset\hG_{0,[4]}/Z$, the restriction 
$f\co\hG^{\l'}\to\hG^{f(\l')}$, where $\hG^{f(\l')}:=f(\hG^{\l'})\leq\hG^\mu/Z^\mu$, induces a map of graphs $f'\co C_P^{\l'}(S_{0,4})\to C_P^{f(\l')}(S_{0,4})$, 
which preserves the orientations.

From the group theoretic Corollary~1.5.7 in \cite{[N]}, it follows that the isomorphism $f$ extends to an isomorphism 
$\tilde{f}\co\pi_1^\mathrm{et}(\cM^\l_\k\!\!\fatslash Z^\l)\to\pi_1^\mathrm{et}(\cM^\mu_\k\!\!\fatslash Z^\mu)$ compatible with the 
augmentation maps to $G_\k$. Let $\k'$ be a finite extension of $\k$ such that the level structures associated to $\hG^{\l'}$ and $\hG^{f(\l')}$ 
can be both defined over $\k'$, so that there are short exact sequences of \'etale fundamental groups:
\[\begin{array}{l}
1\to\hG^{\l'}\to\pi_1^\mathrm{et}(\cM^{\l'}_{\k'})\to G_{\k'}\to 1,\\
\\
1\to\hG^{f(\l')}\to\pi_1^\mathrm{et}(\cM^{f(\l')}_{\k'})\to G_{\k'}\to 1.
\end{array}\]
Then, $\tilde{f}$ restricts to an isomorphism $\tilde{f}\co\pi_1^\mathrm{et}(\cM^{\l'}_{\k'})\to\pi_1^\mathrm{et}(\cM^{f(\l')}_{\k'})$ compatible with the 
augmentation maps to $G_{\k'}$. Let us observe that both stacks $\cM^{\l'}_{\k'}$ and $\cM^{f(\l')}_{\k'}$ associated to the levels 
$\hG^{\l'}$ and $\hG^{f(\l')}$ are representable and hence are smooth curves defined over $\k'$. 

Thus, from Theorem~16.5 in \cite{[Mo]}, we conclude that $\tilde{f}$ (and so $f$) 
is induced by a $\k'$-isomorphism of algebraic curves $\alpha_f\co\cM^{\l'}_{\k'}\to\cM^{f(\l')}_{\k'}$. 
This isomorphism then induces the orientation preserving isomorphism $f'\co C_P^{\l'}(S_{0,4})\to C_P^{f(\l')}(S_{0,4})$ we were looking for.
\medskip

For $d(S)>1$, let us observe first that, since $\kG(S_{1,2})/\langle\iota\rangle$ identifies with an open subgroup of $\kG(S_{0,5})$ and 
$\cM(S_{1,2})\!\!\fatslash\langle\iota\rangle$ with the associated level structure over $\cM(S_{0,5})$ (resp.\ $\kG(S_{2})/\langle\iota\rangle$ 
identifies with an open subgroup of $\kG(S_{0,6})$ and $\cM(S_{2})\!\!\fatslash\langle\iota\rangle$ with the associated level structure over 
$\cM(S_{0,6})$), the cases $S=S_{1,2}$ and $S_2$ reduce to the cases $S=S_{0,5}$ and $S_{0,6}$, respectively. 
Therefore, we can assume that $S\neq S_{1,2},S_2$ and, in particular, $Z(\kG(S))=\{1\}$.

By Remark~\ref{grouptheoretic}, for $f\in\Isom_{G_\k}^\ast(\kG^\l,\kG^\mu)$, the assignment $\g\mapsto f_\ast(\g)$, where
$\g\in\hL(S_{g,n})_0$ and $\hG^\mu_{f_\ast(\g)}:=f(\hG^\l_\g)$, defines a continuous bijection 
$f_\ast\co \hC(S_{g,n})_0\to\hC(S_{g,n})_0$. The same argument which we used in  the proof of Proposition~\ref{inertiapreservingprop} 
then shows that this map extends to an isomorphism $\psi(f)\co\kC(S_{g,n})\to\kC(S_{g,n})$. It follows that there is a natural continuous 
$\kG^\mu$-equivariant map: 
\[\psi\co\Isom_{G_\k}^\ast(\kG^\l,\kG^\mu)\to\Aut(\kC(S_{g,n})),\]
compatible with the natural embedding of $T_{\kG(S)}(\kG^\l,\kG^\mu)$ in $\Isom_{G_\k}^\ast(\kG^\l,\kG^\mu)$ 
and its natural faithful action on the curve complex $\kC(S_{g,n})$. From (i) of Theorem~\ref{faithfulness}, it then follows that $\psi$ is injective. 

Therefore, in order to conclude the proof Theorem~\ref{mainanabelian2} for $S\neq S_{0,4}$, we have to show that, for any 
$f\in\Isom_{G_\k}^\ast(\kG^\l,\kG^\mu)$, its image $\psi(f)\in\Aut(\hC(S_{g,n}))$ induces a compatible element of the transporter 
$T_{\kG(S)}(\kG^\l,\kG^\mu)$. 

The automorphism $\psi(f)$ 
acts faithfully on the set of vertices of the procongruence pants complex $\kC_P(S_{g,n})$. By Theorem~\ref{completepantsrigidity},  
it is then enough to show that this action preserves the edges of $\kC_P(S_{g,n})$ and the orientations 
of its profinite Farey subgraphs.

For simplicity, let us denote $\psi(f)$ simply by $f$.
By Lemma~\ref{produalgraph}, $f$ induces an automorphism of the dual graph $\kC^\ast(S_{g,n})$ of $\kC(S_{g,n})$. 
By (iii) of Lemma~\ref{edgespantsgraph}, for an edge $e$ of $\kC^\ast(S_{g,n})$ with vertices $v_0$ and $v_1$, the intersection
$\Star_{v_0}\cap\Star_{v_1}$ is the vertex set of a profinite Farey subgraph $\wh{F}_e$ of $\kC_P(S_{g,n})$. In particular, $f$ maps the vertex set of 
$\wh{F}_e$ to the vertex set of $\wh{F}_{f(e)}$. So, for every $(d_{g,n}-1)$-dimensional simplex $\s$ of $\kC(S_{g,n})$, the automorphism
$f$ maps the vertex set of the profinite Farey graph $\wh{F}_\s$ to the vertex set of $\wh{F}_{f(\s)}$. In order to complete the proof 
we have to show that $f$ maps every edge of $\wh{F}_\s$ to an edge of $\wh{F}_{f(\s)}$ and that the resulting isomorphism
$\wh{F}_\s\stackrel{\sim}{\to}\wh{F}_{f(\s)}$ is orientation preserving. This can be reduced to the case of modular dimension $1$, treated in the
first part of the proof.

In fact, the profinite Farey graphs $\wh{F}_\s$ and $\wh{F}_{f(\s)}$ identify, respectively, with the profinite Farey graphs associated to the $1$-dimensional
open strata $\dot{\Delta}_\s^\l$ and $\dot{\Delta}_{f(\s)}^\mu$ of the DM boundary of $\cM^\l_\k$ and $\cM^\mu_\k$ parameterized by $\s$ and $f(\s)$, 
respectively.

By the natural bijection~(\ref{geometricformulation2}), the given isomorphism $f\in\Isom_{G_\k}^\ast(\kG^\l,\kG^\mu)$ 
extends to an isomorphism: 
\[\tilde{f}\co\check{\pi}_1(\cM^\l_\k)\to\check{\pi}_1(\cM^\mu_\k),\]
compatible with the augmentation maps to $G_\k$. 
Let $\k'$ be a finite extension of $\k$ such that the both $\dot{\Delta}_\s^\l$ and $\dot{\Delta}_{f(\s)}^\mu$ can be defined over $\k'$. 
Then, $\tilde{f}$ induces an isomorphism
\[\tilde{f}_\s\co\pi_1^\mathrm{et}(\dot{\Delta}_\s^\l\times\Spec\k'\!\!\fatslash Z^\l_\s)\to
\pi_1^\mathrm{et}(\dot{\Delta}_{f(\s)}^\mu\times\Spec\k'\!\!\fatslash Z^\mu_{f(\s)}),\]
compatible with the augmentation maps to $G_{\k'}$, where $ Z^\l_\s$ and $Z^\mu_{f(\s)}$ are the generic automorphisms groups 
of $\dot{\Delta}_\s^\l$ and $\dot{\Delta}_{f(\s)}^\mu$, respectively. 

Arguing as in the first part of the proof, we then see that $\tilde{f}_\s$ 
induces an orientation preserving isomorphism between the profinite Farey graphs $\wh{F}_\s$ and $\wh{F}_{f(\s)}$.
\end{proof}

As a consequence of the generalized Royden's theorems, we can give a more geometric formulation of Theorem~\ref{mainanabelian2} 
which is a further step in the proof of Theorem~\ref{mainanabelian}: 

\begin{theorem}\label{mainanabelian3}For $\kG^\l$ and $\kG^\mu$ open subgroups of $\kG(S)$ and $\k$ a sub-$p$-adic field such that 
all automorphisms of $\cM^\l_\k$ and $\cM^\mu_\k$ are defined over $\k$, there is a natural bijection:
\begin{equation}\label{erased}
\Isom_{G_\k}^\ast(\kG^\l/ Z^\l,\kG^\mu/ Z^\mu)^\mathrm{out}\stackrel{\sim}{\to}
1\mbox{-}\Isom_\k(\cM^\l_\k\!\!\fatslash {\mathcal Z}^\l,\cM^\mu_\k\!\!\fatslash {\mathcal Z}^\mu),
\end{equation}
where  $Z^\l$, $Z^\mu$ and ${\mathcal Z}^\l$, ${\mathcal Z}^\mu$ are as in Definition~\ref{centerMCG} and Definition~\ref{genericaut}, respectively.  
\end{theorem}

For the proof of Theorem~\ref{mainanabelian3}, we will need the following definition and two lemmas:

\begin{definition}\label{outtransporter}For $H_1$ and $H_2$ subgroups of a group $G$, we let $T_{G}(H_1,H_2)^\mathrm{out}$ 
be the set of orbits of $T_{G}(H_1,H_2)$ under the action of $\Inn H_2$.  
\end{definition}

\begin{lemma}\label{transportercomp}Let $\phi\co G\to\check{G}$ be a homomorphism from a group to a profinite group with dense image. Let $\check{U}$, 
$\check{V}$ be open subgroups of $\check{G}$ and put $U:=\phi^{-1}(\check{U})$, $V:=\phi^{-1}(\check{V})$. Then, $\phi$ induces a natural bijection:
$T_G(U,V)^\mathrm{out}\stackrel{\sim}{\to}T_{\check{G}}(\check{U},\check{V})^\mathrm{out}$.
\end{lemma}

\begin{proof}If $T_G(U,V)$ is non empty, of course, the same is true for
$T_{\check{G}}(\check{U},\check{V})$. Let us show that, if $T_{\check{G}}(\check{U},\check{V})$ is non empty, also $T_G(U,V)$ is non empty.
Let $x\in\check{G}$ be such that $x\check{U}x^{-1}=\check{V}$. Every element in the open coset $x \check{U}$ then has the same property.
Since $\phi$ has dense image, the set $x \check{U}\cap\phi(G)$ is nonempty. 
For $y\in x \check{U}\cap\phi(G)$, let then $\tilde{y}\in\phi^{-1}(y)$. We clearly have $\tilde{y} U\tilde{y}^{-1}=V$.
So, in the proof of the lemma, we can assume that both transporters $T_G(U,V)$ and $T_{\check{G}}(\check{U},\check{V})$ are non empty.

There is a natural map $T_G(U,V)^\mathrm{out}\to T_{\check{G}}(\check{U},\check{V})^\mathrm{out}$ whose domain is a transitive free $N_G(V)/V$-set 
and whose codomain a transitive free $N_{\check{G}}(\check{V})/\check{V}$-set. Therefore, in order to prove that this map is bijective, 
it is enough to show that the natural homomorphism $N_G(V)/V\to N_{\check{G}}(\check{V})/\check{V}$, induced by $\phi$, is an isomorphism.

Let $\check{N}$ be an open normal subgroup of $\check{G}$ contained in $\check{V}$ and put $N:=\phi^{-1}(\check{N})$. Since $\phi$ has dense
image, we can identify the quotient $G/N$ with the quotient $\check{G}/\check{N}$ and the quotient $V/N$ with the quotient $\check{V}/\check{N}$.
There is then a series of natural isomorphisms:
\[N_G(V)/V\cong N_{G/N}(V/N)\left/(V/N)\right.\cong N_{\check{G}/\check{N}}(\check{V}/\check{N})\left/(\check{V}/\check{N})\right.
\cong N_{\check{G}}(\check{V})/\check{V}.\]
\end{proof} 
 
\begin{lemma}\label{Royden}There is a natural bijection:
\[1\mbox{-}\Isom_\C(\cM^\l_\C\!\!\fatslash {\mathcal Z}^\l,\cM^\mu_\C\!\!\fatslash {\mathcal Z}^\mu)\stackrel{\sim}{\to}
T_{\G(S)/Z}(\G^\l/Z^\l,\G^\mu/Z^\mu)^\mathrm{out}.\]
\end{lemma}

\begin{proof}It is clear that an inner automorphism of $\G(S)/Z$ which maps $\G^\l/Z^\l$ onto $\G^\mu/Z^\mu$ induces an automorphism
of the Teichm\"uller space $\cT(S)$ which descends to a $1$-isomorphism between $\cM^\l_\C\!\!\fatslash {\mathcal Z}^\l$ and 
$\cM^\mu_\C\!\!\fatslash {\mathcal Z}^\mu$ and that two such inner automorphisms of $\G(S)/Z$ which differ by an inner automorphism
of $\G^\mu/Z^\mu$ induce the same $1$-isomorphism.

Thus, we only have to show that every $1$-isomorphism between $\cM^\l_\C\!\!\fatslash {\mathcal Z}^\l$ and 
$\cM^\mu_\C\!\!\fatslash {\mathcal Z}^\mu$ is induced by an inner automorphism of $\G(S)/Z$, which is unique modulo $\Inn(\G^\mu/Z^\mu)$. 

By the Theorem in Section~2  and Corollary~1 in Section~4 of \cite{EK}, a $1$-isomorphism between 
$\cM^\l_\C\!\!\fatslash {\mathcal Z}^\l$ and $\cM^\mu_\C\!\!\fatslash {\mathcal Z}^\mu$ lifts to a self-isometry of the Teichm\"uller space $\cT(S)$, 
which is unique modulo covering transformations of $\cT(S)\to\cM^\mu_\C\!\!\fatslash {\mathcal Z}^\mu$. 
The restriction to non-exceptional types of surfaces in the hypothesis of the Theorem in Section~2 of \cite{EK}  is not necessary in our situation, since we are already assuming that $(g,n)=(g',n')$. 
By Corollary~3 in Section~4 of \cite{EK}, this self-isometry is induced by an element of $\Inn(\G(S)/Z)$
and the group of covering transformations of $\cT(S)\to\cM^\mu_\C\!\!\fatslash {\mathcal Z}^\mu$ identifies with $\Inn(\G^\mu/Z^\mu)$. 
\end{proof}

 \begin{proof}[Proof of Theorem~\ref{mainanabelian3}]
By Lemma~\ref{transportercomp}, if we let $\G^\l:=\kG^\l\cap\G(S)$ and $\G^\mu:=\kG^\mu\cap\G(S)$, there is a natural bijection of finite sets:
\[T_{\G(S)/Z}(\G^\l/Z^\l,\G^\mu/Z^\mu)^\mathrm{out}\stackrel{\sim}{\to}T_{\kG(S)/Z}(\kG^\l/Z^\l,\kG^\mu/Z^\mu)^\mathrm{out}.\]
Theorem~\ref{mainanabelian3} now follows from Theorem~\ref{mainanabelian2}, the above bijection and Lemma \ref{Royden}. 
\end{proof}

\begin{proof}[Proof of Theorem~\ref{mainanabelian}] Item (ii) of the theorem immediately follows from Theorem~\ref{slim}. 
Let us then prove the first item of the theorem under the additional hypothesis that all automorphisms of $\cM^\l_\k$ and $\cM^\mu_\k$ 
are defined over $\k$. By Theorem in Section~2 of \cite{EK}, the set $1\mbox{-}\Isom_\k(\cM^\l_\k,\cM^\mu_\k)$ is empty unless $S=S'$, 
except, possibly, for the exceptional cases: $S=S_{1,2}$ or $S_{0,5}$ and $S'=S_{0,5}$ or $S_{1,2}$, $S=S_2$ or $S_{0,6}$ and 
$S'=S_{0,6}$ or $S_2$ (note that we are assuming that $S,S'\neq S_{0,4}$). 

Similarly, since an element $f\in\Isom^\ast(\kG^\l,\kG^\mu)$ induces an isomorphism $f_\ast\co\kC(S)\to\kC(S')$, from Theorem~\ref{nonisomorphism},
it follows that, except for the same exceptional cases considered above, $\Isom_{G_\k}^\ast(\kG^\l,\kG^\mu)^\mathrm{out}$ is also empty unless $S=S'$. 

The exceptional cases can also be reduced to the case $S=S'$. For the exceptional cases such that $d(S)=d(S')=2$, let us observe that, if $Z^\l\neq\{1\}$, 
then also $Z^\mu\neq\{1\}$, so that $S=S'=S_{1,2}$. If instead $Z^\l=\{1\}$, then $Z^\mu=\{1\}$ and both $\kG^\l$, $\kG^\mu$ identify with open
subgroups of $\kG(S_{0,5})$ and $\cM^\l$, $\cM^\mu$ with level structures over $\cM(S_{0,5})$. For the case $d(S)=d(S')=3$, we can argue 
in a similar way.

We can then assume $S=S'\neq S_{0,4}$. By Theorem~\ref{mainanabelian3}, there is a natural bijection~(\ref{erased}) of finite sets:
\[\Isom_{G_\k}^\ast(\kG^\l/ Z^\l,\kG^\mu/ Z^\mu)^\mathrm{out}\stackrel{\sim}{\to}1\mbox{-}\Isom_\k(\cM^\l_\k\!\!\fatslash Z^\l,\cM^\mu_\k\!\!\fatslash Z^\mu).\]

Moreover, since the sets $\Isom_\k(\cM^\l_\k,\cM^\mu_\k)$ and $\Isom(\kG^\l,\kG^\mu)$ are both nonempty or empty according to whether 
$Z^\l=Z^\mu$ or $Z^\l\neq Z^\mu$, we can also assume that $Z^\l=Z^\mu$. 

With all the above assumptions, by (ii) of Lemma~\ref{1-2-aut}, $1\mbox{-}\Isom_\k(\cM^\l_\k ,\cM^\mu_\k)$ is a trivial 
$H_1(\kG^\mu/Z^\mu,Z^\mu)$-torsor over the set $1\mbox{-}\Isom_\k(\cM^\l_\k\!\!\fatslash Z^\l,\cM^\mu_\k\!\!\fatslash Z^\mu)$. 
Similarly, by Lemma~\ref{kernelautpro}, $\Isom_{G_\k}^\ast(\kG^\l,\kG^\mu)^\mathrm{out}$ is a trivial $H_1(\kG^\mu/Z^\mu,Z^\mu)$-torsor 
over the set $\Isom_{G_\k}^\ast(\kG^\l/ Z^\l,\kG^\mu/ Z^\mu)^\mathrm{out}$. 
These torsor structures are clearly compatible with each other and with the bijection~(\ref{erased}). Therefore, we get a natural bijection:
\[1\mbox{-}\Isom_\k(\cM^\l_\k,\cM^\mu_\k)\stackrel{\sim}{\to}\Isom_{G_\k}^\ast(\kG^\l,\kG^\mu)^\mathrm{out}.\]

This concludes the proof of Theorem~\ref{mainanabelian} for the case in which all automorphisms of $\cM^\l_\k$ and $\cM^\mu_\k$ are defined over $\k$.

The general case follows from standard descent techniques. Let $\k'$ be a finite Galois extension of $\k$ such that all automorphisms of 
$\cM^\l_\k$ and $\cM^\mu_\k$ are defined over $\k'$. By the previous part of the proof, there is a natural bijection:
\[1\mbox{-}\Isom_{\k'}(\cM^\l_\k,\cM^\mu_\k)\stackrel{\sim}{\to}\Isom_{G_{\k'}}^\ast(\kG^\l,\kG^\mu)^\mathrm{out}\] 
and $1\mbox{-}\Isom_\k(\cM^\l_\k,\cM^\mu_\k)$ identifies with the subset of invariants for the natural action of the finite Galois group $G_{\k'/\k}$ on
$1\mbox{-}\Isom_{\k'}(\cM^\l_\k,\cM^\mu_\k)$. Since a similar statement is also true for the right hand side of the bijection above, the conclusion follows.
\end{proof}

\subsection{Automorphisms of arithmetic procongruence mapping class groups}In the introduction,
we defined $\kG(S)_\Q:=\check{\pi}_1(\cM(S)_\Q,\ol{\xi})$ to be the \emph{arithmetic procongruence mapping class group}. More generally, 
we let $\kG^\l_\k:=\check{\pi}_1(\cM^\l_\k,\ol{\xi})$. Note that, by the geometric interpretation of the procongruence curve complex $\kC(S)$ 
given in Section~\ref{curvecomplexint}, there is a natural action of $\kG(S)_\Q$ on $\kC(S)$. For a number field $\k$, the group $\kG^\l_\k$ identifies 
with an open subgroup of $\kG(S)_\Q$. There is then also an action of $\kG^\l_\k$ on $\kC(S)$.

\begin{definition}\label{decpreservingdef}For $U$ an open subgroup of $\kG(S)_\Q$, let $\Aut^\ast(U)$ be the closed subgroup 
of $\Aut(U)$ consisting of those automorphisms which preserve the set of stabilizers $\{U_\g\}_{\g\in\hL(S)_0}$ for the action 
of $U$ on $\kC(S)$. 
\end{definition}

We then have:

\begin{theorem}\label{absoluteanabelian}For $d(S)>1$, let $U$ be an open normal subgroup of the arithmetic procongruence mapping class group 
$\kG(S)_\Q$. There is then a short exact sequence:
\[1\to \Hom(U/Z(U),Z(U))\to\Aut^\ast(U)\to\Inn(\kG(S)_\Q)\to 1.\]
In particular, for $S\neq S_{1,2},S_{2}$, we have:
\[\Aut^\ast(U)=\Inn(\kG(S)_\Q)\cong\kG(S)_\Q.\]
\end{theorem}

\begin{remark}\label{absan}Note that $\Inn(\kG(S)_\Q)/\Inn(U)\cong\Aut_\k(\cM^U)$, where $\cM^U$ and $\k$ are, respectively, the \'etale
covering of $\cM(S)_\Q$ and the field extension of $\Q$ determined by the open subgroup $U$ of $\kG(S)_\Q$. Therefore, the last identity
of Theorem~\ref{absoluteanabelian} implies the isomorphism $\Out^\ast(U)\cong\Aut_\k(\cM^U)$, which can be interpreted as an absolute 
anabelian property for $U$. 
\end{remark}

\begin{proof}Let us denote by $\pi\co\kG(S)_\Q\to G_\Q$ the augmentation map. 
The image $\pi(U)$ is an open normal subgroup of $G_\Q$ and identifies with the absolute Galois group $G_\k$ of a finite normal 
extension $\k$ of $\Q$.  The kernel $U\cap\kG(S)$ of the augmentation map restricted to $U$ is a (topologically) finitely generated subgroup. 
The same argument of Lemma~1.6.2 in \cite{[N]} then implies that $U\cap\kG(S)$ is a characteristic subgroup of $U$. 
Therefore, there are two natural representations:
\[\Aut(U)\to\Aut(G_\k)\hspace{0.6cm}\mbox{and}\hspace{0.6cm}\Aut(U)\to\Aut(U\cap\kG(S)).\]

Note that, by the Neukirch-Uchida theorem (cf.\ Theorem~12.2.1 in \cite{NSW}), since $\Q\subseteq\k$ is a normal extension, 
there holds $\Aut(G_\k)=\Inn(G_\Q)$. In particular, the natural representation $\Aut(U)\to\Aut(G_\k)=\Inn(G_\Q)$ is surjective. For the other one, we have:

\begin{lemma}\label{characteristic2}If the center $Z(U)$ of $U$ is trivial, then the natural representation $\Aut(U)\to\Aut(U\cap\kG(S))$ is injective.
\end{lemma}

\begin{proof}By Proposition~7.9 in \cite{[B3]}, for $Z(U)=\{1\}$, the natural representation $\Inn(U)\to\Aut(U\cap\kG(S))$ is injective.
The conclusion then follows from Lemma~\ref{grouplemma}.
\end{proof}

Let us denote by $\Aut_{G_\k}(U)$ the kernel of the representation $\Aut(U)\to\Aut G_\k$ and put 
$\Aut_{G_\k}^\ast(U):=\Aut_{G_\k}(U)\cap\Aut^\ast(U)$. By Definition~\ref{inertiapreservingdef},  
$\Aut^\ast(U\cap\kG(S))$ is the closed subgroup of $\Aut(U\cap\kG(S))$ consisting of those automorphisms which preserve
the set of subgroups $\{(U\cap\kG(S))_\g\}_{\g\in\hL(S)_0}$. Then, we have:

\begin{lemma}\label{restrictionmap}The image of $\Aut_{G_\k}^\ast(U)$ by the homomorphism $\Aut(U)\to\Aut(U\cap\kG(S))$ 
is the subgroup $\Aut^\ast_{G_\k}(U\cap\kG(S))$ of $\Aut^\ast(U\cap\kG(S))$ which consists of those automorphisms which commute 
with $G_\k$ modulo inner automorphisms.
\end{lemma}

\begin{proof}By Corollary~1.5.7 in \cite{[N]}, the image of
$\Aut_{G_\k}(U)$ by $\Aut(U)\to\Aut(U\cap\kG(S))$ is precisely the subgroup $\Aut_{G_\k}(U\cap\kG(S))$ 
of $\Aut(U\cap\kG(S))$ which consists of those automorphisms which commute with $G_\k$ modulo inner automorphisms.

Since $U_\s\cap\kG(S)=(U\cap\kG(S))_\s=U\cap\kG(S)_\s$, for all $\s\in\kC(S)$, it is also clear that the epimorphism 
$\Aut_{G_\k}(U)\to\Aut_{G_\k}(U\cap\kG(S))$ maps the subgroup $\Aut^\ast_{G_\k}(U)$ to the subgroup $\Aut^\ast_{G_\k}(U\cap\kG(S))$. 
In order to prove that this is onto, we have to show that an element $f\in\Aut(U)$ which preserves the set of subgroups $\{U\cap\kG(S)_\g\}_{\g\in\hL(S)_0}$ 
also preserves the set of stabilizers $\{U_\g\}_{\g\in\hL(S)_0}$. For this, we need the lemma:

\begin{lemma}\label{decompstructure}\leavevmode
\begin{enumerate}
\item For $\s\in\kC(S)$, the center $Z(U\cap\kG(S)_\s)$ of $U\cap\kG(S)_\s$ is generated by $\hI_\s\cap U$ and by,
if $S\ssm\s=S_{1,1}$ or $S_{1,2}$, the hyperelliptic involution.
\item $Z(U\cap\kG(S)_\s)$ is  a characteristic subgroup of the stabilizer $U_\s$.
\end{enumerate}
\end{lemma}

\begin{proof}(i): This follows from Theorem~\ref{stabilizers} and Corollary~6.2 in \cite{[B3]}.
\smallskip

\noindent
(ii): As above, the same argument of Lemma 1.6.2 in \cite{[N]} implies that $U\cap\kG(S)_\s$ is a characteristic subgroup of 
$U_\s$. Since the center $Z(U\cap\kG(S)_\s)$ is a characteristic subgroup of $U\cap\kG(S)_\s$, the claim follows. 
\end{proof}

The conclusion  of Lemma \ref{restrictionmap} then follows because, by (ii) of Lemma~\ref{decompstructure}, the given $f$   
preserves the set of subgroups $\{Z(U\cap\kG(S)_\g)\}_{\g\in\hL(S)_0}$ and so preserves their normalizers in $U$ which, 
by (i) of Lemma~\ref{decompstructure} and Corollary~\ref{centralizers multitwists}, is just the set of stabilizers $\{U_\g\}_{\g\in\hL(S)_0}$.
\end{proof}
 
 We need one more lemma:
 
\begin{lemma}\label{centerarithmetic}For every open subgroup $U$ of $\kG(S)_\Q$, we have:
\[Z_{\kG(S)_\Q}(U)=Z_{\kG(S)}(U\cap\kG(S)).\]
In particular, $Z(U)=Z(U\cap\kG(S))$.
\end{lemma}
\begin{proof}It is well known that, for every open subgroup $V$ of $G_\Q$, we have $Z_{G_\Q}(V)=\{1\}$ (cf.\ Corollary~12.1.6 in \cite{NSW}). 
Therefore, at least, we have $Z_{\kG(S)_\Q}(U)\subseteq Z_{\kG(S)}(U\cap\kG(S))$. 
The centralizer of $U\cap\kG(S)$ in $\kG(S)$, if not trivial, is generated by the hyperelliptic involution which is defined over 
$\Q$ and thus is invariant under the action of $G_\Q$. The conclusion then follows.
\end{proof}

By Lemma~\ref{centerarithmetic}, we have $Z(U\cap\kG(S))=Z(U)$. Hence
$Z(U)$ is a characteristic subgroup of $U$ of order at most $2$. By Lemma~\ref{kernelautpro}, there is then an exact sequence:
\[1\to \Hom(U/Z(U),Z(U))\to\Aut^\ast(U)\to\Aut^\ast(U/Z(U)).\]
Therefore, it is enough to prove Theorem~\ref{absoluteanabelian} for the case $Z(U)=\{1\}$.

By Lemma~\ref{characteristic2} and Lemma~\ref{restrictionmap}, for $Z(U\cap\kG(S))=Z(U)=\{1\}$, there is a split short exact sequence:
\[1\to\Aut^\ast_{G_\k}(U\cap\kG(S))\to\Aut^\ast(U)\to\Inn G_\Q\to 1.\]
Since, by Theorem~\ref{mainanabelian2}, we have that $\Aut^\ast_{G_\k}(U\cap\kG(S))=\Inn(\kG(S))$, we conclude
that $\Aut^\ast(U)=\Inn(\kG(S)_\Q)$.
\end{proof}

\begin{remark}For $S=S_{0,n}$ and $U=\hG(S)_\Q$, a $\ast$-free version of the isomorphism $\Out^\ast(U)\cong\Aut_\k(\cM^U)$ (cf.\
Remark~\ref{absan}) was proved by Nakamura and Tamagawa (cf.\ Corollary~C in \cite{IN}).
\end{remark}

\end{document}